\documentclass[noams]{compositio}
\usepackage[displaymath, mathlines, right, pagewise]{lineno}

\let\openbox\relax

\usepackage{amsmath,amsthm,amsfonts,amssymb,mathrsfs,bm,bbm}
\usepackage[usenames]{color}
\usepackage{hyperref}
\usepackage{amssymb}
\usepackage{appendix}

\renewcommand{\contentsname}{Contents}
\usepackage{amsmath}

\renewcommand{\epsilon}{\varepsilon}
\renewcommand{\phi}{\varphi}

\newtheorem{theorem}{Theorem}[section]

\newtheorem{lemma}[theorem]{Lemma}
\newtheorem{corollary}[theorem]{Corollary}
\newtheorem{conjecture}[theorem]{Conjecture}
\newtheorem{proposition}[theorem]{Proposition} 
\newtheorem{definition}[theorem]{Definition} 

\usepackage{tikz}
\usepackage{tikz-cd}
\usetikzlibrary{arrows}
\usetikzlibrary{cd}

\begin{document}

\numberwithin{equation}{section}

\title[The local twisted Gan-Gross-Prasad conjecture for $\text{U}(V_K)/\text{U}(V)$]
{The local twisted Gan-Gross-Prasad conjecture for $U(V_K)/U(V)$}

\author{Nhat Hoang Le}
\email{lnhoang@nus.edu.sg}
\address{Department of Mathematics, Block S17, National University of Singapore, 10 Lower Kent Ridge Drive, 119076.}

\classification{22E50, 22E35, 20G05 (primary)}
\keywords{Gan-Gross-Prasad conjecture, branching laws, local relative trace formula, unitary groups, Weil representation, skew-Hermitian spaces, Langlands parameters, tempered representations.}


\begin{abstract}
    In this paper, we obtain geometric expansions of a local trace formula and its twisted variant for the twisted Gan-Gross-Prasad conjecture. As an application, we prove the local twisted Gan-Gross-Prasad conjecture for $U(V_K)/U(V)$ for tempered $L$-parameters over nonarchimedean fields of odd residual characteristic.
\end{abstract}

\maketitle

\tableofcontents



\section{Introduction}\label{sec1}

The Gan-Gross-Prasad conjectures \cite{GP92,GP94,GGP12a,GGP12b} study a family of branching problems for classical groups. Namely, for a pair $W\subseteq V$ of orthogonal, hermitian, skew-hermitian or symplectic spaces, the three authors give a precise description of the multiplicity $\dim \text{Hom}_{H}(\pi,\nu)$ via local and global Langlands correspondences, where $\pi$ is an irreducible representation in a generic $L$-packet of $G$ (the product of isometric groups of $V$ and $W$) and $\nu$ is a certain representation of a certain subgroup $H$ of $G$ (see \cite{GGP12a} for more details). In local field cases, the conjecture has been solved by a series of works by Waldspurger, Moeglin, Beuzart-Plessis, Gan, Ichino, Atobe, Xue, Luo and Chen in \cite{Wal10,Wal12a,Wal12c,Wal12d,Wal12e,MW12,BP14,BP15,BP16,BP20,GI16,Ato18,Xue23a,Xue23b,Luo21,CL22,Ch21,Ch23}.

In \cite{GGP23}, Gan, Gross and Prasad proposed a twisted variant of the Fourier-Jacobi model for unitary groups. Let $F$ be a $p$-adic field and $E$ and $K$ be two quadratic field extensions of $F$. Let $V$ be an $n$-dimensional skew-hermitian space relative to $E/F$ and $V_K=V\otimes_F K$. Let $\psi$ be a nontrivial additive character of $F$ and $\mu$ be a conjugate-symplectic character of $E^{\times}$. Let $\omega_{V,\psi,\mu}$
be the Weil representation of the isometry group $U\left(V\right)$. The three authors consider the following multiplicity corresponding to the twisted Gan-Gross-Prasad triple $(G,H_V,\omega_{V,\psi,\mu})$, where $G=\text{Res}_{K/F}\,U(V_K)$ and $H_V=U(V)$,
$$
m_V(\pi)= \dim \text{Hom}_{H_V}(\pi,\omega_{V,\psi,\mu}),
$$
for any irreducible representations $\pi$ of $G(F)$. Noting that $G$ does not depend on choices of $n$-dimensional skew-hermitian forms $V$. We recall \cite[Conjecture 8.3]{GGP23}.
\begin{conjecture}\label{mainconjecture}
    \begin{enumerate}
    \item For each irreducible representation $\pi$ of $G(F)$, 
    $$
    m_V(\pi)\leq 1.
    $$
    \item Let $M$ be a generic $L$-parameter for $G$ with associated $L$-packet $\Pi_M\subset \text{Irr}(G)$, then  
    $$
    \underset{V}{\sum}\sum_{\pi\in \Pi_M} m_{V}\left(\pi\right)=1,
    $$
    where the first sum runs over the two skew-hermitian spaces over $E$ of dimension $n$ and the second runs over the $L$-packet $\Pi_M$.
    \item The unique $V_0$ which gives a nonzero contribution to the above sum corresponds to 
    $$
    \mu\left(\det V_0\right)=\epsilon(1/2,\text{As}_{L/E}(M)\times \mu^{-1},\psi_E)\cdot \det(\text{As}_{L/E}(M))(e)\cdot \omega_{K/F}(e^2)^{n(n-1)/2},
    $$
    where $L=K\otimes_F E$ and $e\in E_0^\times$, so that $E=F(e)$.
    \item The unique $\pi\in \Pi_M$ which gives a nonzero contribution to the sum in the second part corresponds to the following character of $A_M=\prod_{i\in I}\mathbb{Z}/2\mathbb{Z}\cdot a_i$:
    $$
    \chi(a_i)=\epsilon(1/2,[\text{As}(M_i)+\text{As}(M)+\text{As}(M/M_i)]\cdot \mu^{-1},\psi_{E,e}),
    $$
    where $\psi_{E,e}$ is the additive character of $E/F$ defined by $\psi_{E,e}=\psi(\text{Tr}_{E/F}(ex))$.
    \end{enumerate}
\end{conjecture}
In \cite{Le25}, the author has proved the above conjecture when $K=E$ and $\pi$ is tempered, under the assumption that $F$ is of odd residual characteristic. We remark that this restriction arises from the condition in \cite[Theorem 4.1]{Kon02}. We plan to extend his result to all residual characteristics in future work. We now consider the case $K\neq E$. In \cite{CG22}, Chen and Gan have proved the conjecture when $M$ is a tempered $L$-parameter for $G$ of the form
$$
M=M_1 + \ldots + M_n
$$
with each $M_i$ one-dimensional and conjugate self-dual of parity $(-1)^{n-1}$. 
\subsection{Main results}
We now assume $F$ is of odd residual characteristic. In this paper, we prove Conjecture \ref{mainconjecture} holds for any tempered $L$-parameter $M$.
\begin{theorem}\label{maintheorem}
    \begin{enumerate}
    \item Let $M$ be a tempered $L$-parameter for $G$. Then  
    $$
    \underset{V}{\sum}\sum_{\pi\in \Pi_M} m_{V}\left(\pi\right)=1,
    $$
    where the first sum runs over the two skew-hermitian spaces over $E$ of dimension $n$ and the second runs over the $L$-packet $\Pi_M$.
    \item The unique $V_0$ which gives a nonzero contribution to the above sum corresponds to 
    $$
    \mu\left(\det V_0\right)=\epsilon(1/2,\text{As}_{L/E}(M)\times \mu^{-1},\psi_E)\cdot \det(\text{As}_{L/E}(M))(e)\cdot \omega_{K/F}(e^2)^{n(n-1)/2},
    $$
    where $L=K\otimes_F E$ and $e\in E_0^\times$, so that $E=F(e)$.
    \item The unique $\pi\in \Pi_M$ which gives a nonzero contribution to the sum in the second part corresponds to the following character of $A_M=\prod_{i\in I}\mathbb{Z}/2\mathbb{Z}\cdot a_i$:
    $$
    \chi(a_i)=\epsilon(1/2,[\text{As}(M_i)+\text{As}(M)+\text{As}(M/M_i)]\cdot \mu^{-1},\psi_{E,e}),
    $$
    where $\psi_{E,e}$ is the additive character of $E/F$ defined by $\psi_{E,e}=\psi(\text{Tr}_{E/F}(ex))$.
    \end{enumerate}
\end{theorem}

Theorem \ref{maintheorem} follows from a geometric formula for the multiplicity $m_V(\pi)$ as well as its twisted variant, together with (twisted) endoscopic comparisons. This method has been successfully developed by Waldspurger, Beuzart-Plessis and Luo in the case of Bessel models. We remark that geometric formulae for branching problems of some spherical varieties (as well as their Whittaker inductions) have been achieved by Beuzart-Plessis, Wan and Zhang in \cite{BP18,Wan19,BW19,WZ23,Wan21b,BW23}. Moreover, a conjectural geometric multiplicity formula for spherical varieties and their Whittaker inductions has been formulated in \cite{Wan21a}.

On the other hand, the Fourier-Jacobi model and its twisted variant do not belong to the framework of spherical varieties. A new feature in this situation is the appearance of the Weil representation of the subgroup $H_V$ (which is of infinite-dimensional) instead of a one-dimensional character in the previous cases. A main contribution in this paper is to obtain a geometric multiplicity formula for the twisted Gan-Gross-Prasad model. As far as we are aware, our work is the first instance where a geometric multiplicity formula outside the context of spherical varieties has been achieved. A geometric multiplicity formula for the Fourier-Jacobi model in the Gan-Gross-Prasad conjecture can be obtained by the same approach without any difficulties, thus gives an alternative proof for the tempered case in \cite{GI16}. We will revisit it in Appendix \ref{appendixB}. 

\subsection{Geometric formula for $m_V(\pi)$}

Our geometric multiplicity formula for $m_V(\pi)$ is formulated in terms of the Harish-Chandra character of $\pi$. Recall that  there exists
a locally integrable smooth function $\theta_\pi$  on the regular semisimple locus $G_{\text{reg}}(F)$ such that 
$$
\text{Trace}(\pi(f))=\int_{G(F)}\theta_\pi(x)f(x)dx,
$$
for any $f\in C^\infty_c(G(F))$. This function is unique and is called the Harish-Chandra character of $\pi$. Moreover, we can regularize $\theta_\pi$ to a function $c_{\theta_\pi}$ defined on the semisimple locus $G_{ss}(F)$. For a precise definition, see Section \ref{sec2.5}. Other ingredients of our formula involves the set $\mathcal{T}_{\text{ell}}(H_V)$ containing representatives of $H_V(F)$-conjugacy classes of elliptic maximal tori of $H_V$, the discriminant $D^G$, as well as the Weil index $\gamma_\psi$ (see Section \ref{sec4.1}). We now state our multiplicity formula. We refer the reader to Theorem \ref{maintheorem1} for more details.
\begin{theorem}\label{1.3}
    For any irreducible tempered representation $\pi$ of $G(F)$, we have
    $$
    m_V(\pi)=\frac{1}{2}c_{\theta_\pi}\left(1\right)+\mu\left(\det V\right)
\underset{T\in{\mathcal{T}}_{\text{ell}}\left(H\right)}{\sum}\frac{\gamma_\psi(T)}{\left|W\left(H,T\right)\right|}
\underset{s\rightarrow0^{+}}{\lim}\int_{T\left(F\right)}D^{G}\left(x\right)^{1/2}c_{\theta_\pi}\left(x\right)\frac{\mu\left(\det\left(1-x^{-1}\right)\right)}{\left|\det\left(1-x\right)\right|_{E}^{1/2-s}}dx.
    $$
\end{theorem}
We remark that the occurrence of Weil indexes and the factor $\frac{\mu\left(\det\left(1-x^{-1}\right)\right)}{\left|\det\left(1-x\right)\right|_{E}^{1/2}}$ can be understood as an incarnation of the character of the Weil representation. As we explain in Section \ref{sec5.7}, Theorem \ref{maintheorem}(i) follows from Theorem \ref{1.3}. Let $M$ be a tempered $L$-parameter for $G$ and $\Pi_M$ be the corresponding $L$-packet. Using Theorem \ref{1.3} and the fact that $\theta_M=\sum_{\pi\in \Pi_M}c_{\theta_\pi}(1)$ is stably invariant, we can see that the sum
$$
\sum_V\sum_{\pi\in \Pi_M}m_V(\pi)
$$
have some cancellations which come from the sign $\mu(\det V)$. The only remaining term is $c_{\theta_M}(1)$, which is equal to $1$ by the generic packet conjecture.

Theorem \ref{1.3} is achieved via a local trace formula approach. In \cite{Le25}, the author has formulated a local trace formula $J_V$ for the twisted Gan-Gross-Prasad conjecture, which we recall as follows. Let $f\in \mathcal{C}_{\text{scusp}}(G(F))$ be a strongly cuspidal function on $G(F)$. Let $\{\phi_i\}_{i\in I}$ be an orthonormal basis for $\omega_{V,\psi,\mu}$. For $x,y\in G(F)$, we set
$$
K_f(x,y)=\sum_{i\in I}\int_{H_V(F)}f(x^{-1}hy)\langle \phi_i,\omega_{V,\psi,\mu}(h)\phi_i \rangle dh.
$$
We define a distribution $J_V$ on the space of strongly cuspidal functions
$$
J_V(f)= \int_{H_V(F)\backslash G(F)} K_f(x,x) dx.
$$
By \cite[Theorem 5.1]{Le25}, the above two integrals are absolutely convergent. As a spectral expansion has already been achieved in \cite{Le25}, our contribution is its geometric expansion. In Section \ref{sec4} and \ref{sec5}, we prove the following theorem (cf. Theorem \ref{maintheorem1}).
\begin{theorem}\label{1.4}
    For any strongly cuspidal function $f\in \mathcal{C}_{\text{scusp}}(G(F))$, we have
    $$
    J_V(f)=\frac{1}{2}c_{\theta_f}\left(1\right)+\mu\left(\det V\right)
\underset{T\in{\mathcal{T}}_{\text{ell}}\left(H\right)}{\sum}\frac{\gamma_\psi(T)}{\left|W\left(H,T\right)\right|}
\underset{s\rightarrow0^{+}}{\lim}\int_{T\left(F\right)}D^{G}\left(x\right)^{1/2}c_{\theta_f}\left(x\right)\frac{\mu\left(\det\left(1-x^{-1}\right)\right)}{\left|\det\left(1-x\right)\right|_{E}^{1/2-s}}dx.
    $$
\end{theorem}
By Harish-Chandra semisimple descent, we are able to deduce the above theorem to a comparison near central elements. This comparison follows from a local character expansion of the Weil representation near central elements, which plays a key role in the local trace formula approach to Fourier-Jacobi type models.

\subsection{Geometric expansion of the twisted multiplicity $\epsilon_\psi(\tilde{\pi})$}
The next step is to prove a geometric formula for the twisted multiplicity $\epsilon_\psi(\tilde{\pi})$, which carries the relevant $\epsilon$-factor. Similar to the multiplicity $m_V(\pi)$, in \cite{Le25}, the author also formulated a twisted trace formula $\tilde{J}$, which gives us information about $\epsilon_\psi(\tilde{\pi})$. Let us briefly recall it here (see Section \ref{sec7.1} for more details). We set $M=\text{Res}_{L/F}\text{GL}_{n}$ and $N=\text{Res}_{E/F}\text{GL}_{n}$, where $L=K\otimes_{F}E$. Let $g\mapsto J_{n}{}^{t}\bar{g}^{-1}J_{n}^{-1}$ be an involution on $M$, and we set $\tilde{M}=M\theta_{n}$ and $\tilde{N}=N\theta_{n}$. We denote by $\omega_\mu$ the Weil representation of $N(F)$ and $\tilde{\omega}_{\psi,\mu}$
its extension to $\tilde{N}\left(F\right)$. Let $\left\{ \phi_{i}\right\} _{i\in I}$
be an orthonormal basis for $\omega_{\mu}$ . For
any $m_1,\,m_2\in M\left(F\right)$, we set 
\[
K_{\tilde{f}}\left(m_1,m_2\right)=\underset{i}{\sum}\int_{\tilde{N}\left(F\right)}\tilde{f}\left(m_1^{-1}\tilde{n}m_2\right)\left\langle \phi_i,\tilde{\omega}_{\psi,\mu,\chi}\left(\tilde{n}\right)\phi_{i}\right\rangle d\tilde{n},
\]
where $\tilde{f}\in\mathcal{C}_{\text{scusp}}\left( \tilde{M}\left(F\right)\right)$. 
We define the following linear form 
\[
\tilde{J}\left(\tilde{f}\right)=\int_{N\left(F\right)\backslash M\left(F\right)}K_{\tilde{f}}\left(m,m\right)dm,
\]
for $\tilde{f}\in\mathcal{C}_{\text{scusp}}\left(\tilde{M}\left(F\right)\right)$. Similar to Theorem \ref{spectral}, the integrals defining $K_{\tilde{f}}$ and $\tilde{J}$ are absolutely convergent. In Section \ref{sec8}, we give a geometric expansion for the linear form $\tilde{J}$ (cf. Theorem \ref{geometrictwisted}).
\begin{theorem}
    For any strongly cuspidal function $\tilde{f}\in \mathcal{C}_{\text{scusp}}\left(\tilde{M}\left(F\right)\right)$, we have
    $$
    \tilde{J}(\tilde{f})=\sum_{\tilde{T}\in \mathcal{T}_{\text{ell}}(\tilde{N})}\frac{\gamma_\psi(\tilde{T})}{|W(N,\tilde{T})|}\lim_{s\rightarrow 0^+}\int_{\tilde{T}(F)/\theta}D^{\tilde{M}}(\tilde{x})^{1/2}c_{\theta_{\tilde{f}}}(\tilde{x})\frac{\mu(\det(1-x^{-1}))}{|\det(1-x)|_E^{1/2-s}}d\tilde{x}.
    $$
\end{theorem}
A novelty in our proof is that instead of treating singularities of $\tilde{J}$ via local harmonic analysis methods as in \cite{Wal12a,BP14}, which have not been developed in the setting of Weil representations of twisted groups yet, we make a shortcut via a twisted endoscopic transfer and the result in \cite{CG22}. Namely, by using descent methods and a twisted endoscopic transfer, it suffices to show that Theorem \ref{maintheorem}(ii) holds for an elliptic tempered $L$-parameter, and the main theorem in \cite{CG22} gives us such an example. This is the first instance that a local theta correspondence argument can be used to help establish a geometric expansion in a local trace formula. In Appendix \ref{appendixB}, motivated by \cite{Xue23a,Xue23b}, we use theta correspondences to obtain geometric sides of twisted trace formulae corresponding to the local Gan-Gross-Prasad conjectures for unitary groups in both Bessel and Fourier-Jacobi models. This contributes to alternative proofs for the main result in \cite{BP14} (Bessel models), as well as the tempered case in \cite{GI16} (Fourier-Jacobi models).

\subsection{Organization of the paper}
We give a description of the content of each section.

In Section \ref{sec2}, we introduce basic notations and conventions in this paper, including weighted orbital integrals and germ expansions. Section \ref{llc} is a brief introduction to the local Langlands correspondence for unitary groups, including matching of orbits and endoscopic relations. In Section \ref{sec3}, we recall the Weil representation for unitary groups and prove its local character expansions near the identity element.

Section \ref{sec4} and \ref{sec5} are devoted to prove the first part of Theorem \ref{maintheorem}. To be more precise, in Section \ref{sec4}, we formulate a geometric expansion for the linear form $J_V$ and deduce Theorem \ref{maintheorem1} to a comparison on Lie algebra level. In Section \ref{sec5}, we establish a spectral expansion of the infinitesimal trace formula $J_V^{\text{Lie}}$ and use it to finish our proof for Theorem \ref{maintheorem1}. Then by using the stability property of the local Langlands correspondence for unitary groups, we establish Theorem \ref{maintheorem}(i).

Section \ref{sec8} and \ref{sec9} are devoted to prove Theorem \ref{maintheorem}(ii) and (iii). In Section \ref{sec8}, we establish a geometric expansion of the twisted trace formula $\tilde{J}$ and use it to prove Theorem \ref{maintheorem}(ii). In Section \ref{sec9}, by using (twisted) endoscopic relations of the local Langlands correspondence for unitary groups, we prove the last part of Theorem \ref{maintheorem}.

Finally, in Appendix \ref{appendixB}, we briefly give alternative proofs for the tempered part of the local Gan-Gross-Prasad conjectures for unitary groups in both Bessel and Fourier-Jacobi models.

\section{Preliminaries}\label{sec2}
\subsection{Groups, measures and notations}\label{sec2.1}

Let $F$ be a $p$-adic field for which we fix an algebraic closure
$\bar{F}$. We denote by $\left|\,\cdot\,\right|_{F}$ the canonical absolute
value on $F$ as well as its unique extension to $\bar{F}$. Let $G$
be a connected reductive group defined over $F$. We denote by $A_{G}$
the split component of the connected component of the center
$Z_{G}$ of $G$. Let $X^{*}\left(G\right)$ be the group of algebraic characters of $G$
defined over $F$ and ${\mathcal{A}}_{G}^{*}=X^{*}\left(G\right)\otimes_{\mathbb{Z}}\mathbb{R}$
and ${\mathcal{A}}_{G}=\text{Hom}\left(X^{*}\left(G\right),\mathbb{R}\right)$.
We define the homomorphism 
$$
\begin{array}{ccccc}
H_{G} & : & G\left(F\right) & \longrightarrow & {\mathcal{A}}_{G}\\
 &  & g & \mapsto & \left(\chi\mapsto\log\left|\chi\left(g\right)\right|_{F}\right)
\end{array}.
$$
Let ${\mathcal{A}}_{G,F}$ and $\tilde{{\mathcal{A}}}_{G,F}$ be images of $G\left(F\right)$
and $A_{G}\left(F\right)$ in ${\mathcal{A}}_{G}$ via $H_{G}$. They are
lattices in ${\mathcal{A}}_{G}$. We define ${\mathcal{A}}_{G,F}^{\vee}=\text{Hom}\left({\mathcal{A}}_{G,F},2\pi\mathbb{Z}\right)$
and $\tilde{{\mathcal{A}}}_{G,F}^{\vee}=\text{Hom}\left(\tilde{{\mathcal{A}}}_{G,F},2\pi\mathbb{Z}\right)$
to be lattices in ${\mathcal{A}}_{G}^{*}$. For a maximal torus $T$ of $G$,
let $\delta\left(G\right)=\dim G-\dim T$, noting that it does not depend on
choices of $T$.

We denote by $\mathfrak{g}$ the Lie algebra of $G$ and 
$$
\begin{array}{ccc}
G\times\mathfrak{g} & \longrightarrow & \mathfrak{g}\\
\left(g,X\right) & \mapsto & gXg^{-1}
\end{array}
$$
the adjoint action. For $x\in G$, we denote by $Z_{G}\left(x\right)$
the centralizer of $x$ in $G$ and by $G_{x}$ its identity component.
We call an element $x$ in $G$ semisimple if it is contained in a
maximal torus of $G$. We denote by $G_{ss}$ the subset of $G$ containing
its semisimple elements. For $x\in G_{ss}$, we set 
$$
D^{G}\left(x\right)=\left|\det\left(1-\text{Ad}\left(x\right)\right)_{\mid\mathfrak{g}/\mathfrak{g}_{x}}\right|_{F}.
$$
An element $x\in G$ is called regular if $Z_{G}\left(x\right)$ is
abelian and $G_{x}$ is a torus. We denote by $G_{reg}$ the subset
of regular elements in $G$. Let ${\mathcal{T}}\left(G\right)$ be a set
of representatives for the conjugacy classes of maximal tori in $G$.
A maximal torus $T$ of $G$ is elliptic if $A_{T}=A_{G}$. An element
$x\in G\left(F\right)$ is said to be elliptic if it belongs to some
elliptic maximal torus. We set $G\left(F\right)_{\text{ell}}$ and
$G_{\text{reg}}\left(F\right)_{\text{ell}}$ the subsets of elliptic
elements in $G\left(F\right)$ and $G_{\text{reg}}\left(F\right)$.

Let us fix a minimal parabolic subgroup $P_{\min}$ of $G$ and a
Levi component $M_{\min}$. We fix a maximal compact subgroup $K$
of $G\left(F\right)$ in good relative position to $M_{\min}$. Let $P=MU$
be a parabolic subgroup of $G$. We have the Iwasawa decomposition
$G\left(F\right)=M\left(F\right)U\left(F\right)K$. We can choose
maps 
$$m_{P}:G\left(F\right)\rightarrow M\left(F\right),\ \ \ u_{P}:G\left(F\right)\rightarrow U\left(F\right),
\ \ \ k_{P}:G\left(F\right)\rightarrow K$$
such that $g=m_{P}\left(g\right)u_{P}\left(g\right)k_{P}\left(g\right)$,
for all $g\in G\left(F\right)$. Then we extend the homomorphism
$H_{M}:M\left(F\right)\rightarrow{\mathcal{A}}_{M}$ to the following map 
$$
\begin{array}{ccccc}
H_{P} & : & G\left(F\right) & \longrightarrow & {\mathcal{A}}_{M}\\
 &  & g & \mapsto & H_{M}\left(m_{P}\left(g\right)\right)
\end{array}.
$$
The above map depends on the maximal compact subgroup $K$ but its
restriction to $P\left(F\right)$ does not and is given by $H_P(mu)=H_M(m)$ for all $m\in M(F)$ and $u \in U(F)$. For a Levi subgroup $M$
of $G$, we denote by ${\mathcal{P}}\left(M\right)$, ${\mathcal{L}}\left(M\right)$
and ${\mathcal{F}}\left(M\right)$ the finite sets of parabolic subgroups
admitting $M$ as their Levi component, of Levi subgroups containing $M$
and of parabolic subgroups containing $M$ respectively. If $M\subset L$
are two Levi subgroups, we set ${\mathcal{A}}_{M}^{L}={\mathcal{A}}_{M}/{\mathcal{A}}_{L}$.

For every Levi subgroup $M$ and maximal torus $T$ of $G$, we denote
by $W\left(G,M\right)$ and $W\left(G,T\right)$ the Weyl groups of
$M\left(F\right)$ and $T\left(F\right)$ respectively, that is 
$$
W\left(G,M\right)=\text{Norm}_{G\left(F\right)}\left(M\right)/M\left(F\right)\text{ and }W\left(G,T\right)=\text{Norm}_{G\left(F\right)}\left(T\right)/T\left(F\right).
$$
We have the Weyl integration formula 
$$
\int_{G\left(F\right)}f\left(g\right)dg=\underset{T\in{\mathcal{T}}\left(G\right)}{\sum}\left|W\left(G,T\right)\right|^{-1}\int_{T\left(F\right)}D^{G}\left(t\right)\left(\int_{T\left(F\right)\backslash G\left(F\right)}f\left(g^{-1}tg\right)dg\right)dt,
$$
for any $f\in C_{c}^{\infty}\left(G\left(F\right)\right)$, where the measure on $T\left(F\right)\backslash G\left(F\right)$ (which we also denote by $dg$) arises from the quotient of ones on $G(F)$ and $T(F)$.

A twisted group is a pair $\left(G,\tilde{G}\right)$, where $G$
is a connected reductive group defined over $F$ and $\tilde{G}$
is a $G$-bitorsor, i.e. an algebraic variety defined over $F$ with
two left and right commutative actions 
$$
\begin{array}{ccc}
G\times\tilde{G}\times G & \longrightarrow & \tilde{G}\\
\left(g,\tilde{\gamma},g^{\prime}\right) & \mapsto & g\tilde{\gamma}g^{\prime}
\end{array},
$$
each of them making $\tilde{G}$ into a principal homogeneous space
under $G$. The underlying group $G$ is usually omitted and we denote
the twisted group $\left(G,\tilde{G}\right)$ by $\tilde{G}$. Note
that when $\tilde{G}=G$ and $G$-actions are group actions,
we have $\left(G,\tilde{G}\right)$ coincides with $G$.

Let $\tilde{G}$ be a twisted group. For any $\tilde{x}\in\tilde{G}$,
there exists a unique automorphism $\theta_{\tilde{x}}$ of $G$ such
that $\tilde{x}g=\theta_{\tilde{x}}\left(g\right)\tilde{x}$ for all
$g\in G$. This induces automorphisms on $X^{*}\left(G\right)$,
$A_{G}$ and ${\mathcal{A}}_{G}$, which are independent of choices of $\tilde{x}$.
For simplicity, we denote the three automorphisms by $\theta_{\tilde{G}}$. 

Assume that $\theta_{\tilde{G}}$ is of finite order. Denote 
$$
A_{\tilde{G}}=\left(A_{G}^{\theta_{\tilde{G}}=1}\right)^{0},\ {\mathcal{A}}_{\tilde{G}}={\mathcal{A}}_{G}^{\theta_{\tilde{G}}=1},\ {\mathcal{A}}_{\tilde{G}}^{*}=\left({\mathcal{A}}_{G}^{*}\right)^{\theta_{\tilde{G}}=1},\,a_{\tilde{G}}=\dim\left({\mathcal{A}}_{\tilde{G}}\right).
$$
We define the homomorphism 
$$
\begin{array}{ccccc}
H_{\tilde{G}} & : & G\left(F\right) & \longrightarrow & {\mathcal{A}}_{\tilde{G}}\\
 &  & g & \mapsto & \left(\chi\mapsto\log\left|\chi\left(g\right)\right|_{F}\right)
\end{array}.
$$
The group $G$ admits a conjugation action on $\tilde{G}$ by $\left(g,\tilde{x}\right)=g\tilde{x}g^{-1}$.
For a subset $\tilde{X}$ of $\tilde{G}$, we denote by $\text{Norm}_{G}\left(\tilde{X}\right)$
resp. $Z_{G}\left(\tilde{X}\right)$ resp. $G_{\tilde{X}}$ the normalizer
resp. the centralizer resp. the identity component of the centralizer
of $\tilde{X}$. For a subset $X$ of $G$, we denote by $N_{\tilde{G}}\left(X\right)$
and $Z_{\tilde{G}}\left(X\right)$ the normalizer and centralizer
of $X$ in $\tilde{G}$ via the action $\left(\tilde{x},g\right)\mapsto\theta_{\tilde{x}}\left(g\right)$.

We call an element $\tilde{x}$ in $\tilde{G}$ semisimple if there
exists a pair $\left(B,T\right)$ consisting of a Borel subgroup $B$
of $G$ and a maximal torus $T$ of $B$ defined over $\bar{F}$ such
that $\tilde{x}$ normalizes $B$ and $T$. We denote by $\tilde{G}_{ss}$
the subset of $\tilde{G}$ containing its semisimple elements. For
$\tilde{x}\in\tilde{G}_{ss}$, let 
$$
D^{\tilde{G}}\left(\tilde{x}\right)=\left|\det\left(1-\theta_{\tilde{x}}\right)_{\mid\mathfrak{g}/\mathfrak{g}_{\tilde{x}}}\right|_{F}.
$$
An element $\tilde{x}\in\tilde{G}$ is called regular if $Z_{G}\left(\tilde{x}\right)$
is abelian and $G_{\tilde{x}}$ is a torus. We denote by $\tilde{G}_{reg}$
the subset of regular elements.

We denote a twisted parabolic subgroup of $\tilde{G}$ by a pair $\left(P,\tilde{P}\right)$,
where $P$ is a parabolic subgroup of $G$ defined over $F$ and $\tilde{P}$
is the normalizer of $P$ in $\tilde{G}$ such that $\tilde{P}\left(F\right)\neq0$.
For such pair, $\tilde{P}$ completely determines $P$, so we often
call $\tilde{P}$ as a twisted parabolic subgroup. 
A twisted Levi component of $\tilde{P}$ is a pair $\left(M,\tilde{M}\right)$
consisting of a Levi component $M$ of $P$ (defined over $F$) and
the normalizer $\tilde{M}$ of $M$ in $\tilde{P}$ such that $\tilde{M}\left(F\right)\neq\emptyset$.
As the second term completely determines the first term, we call $\tilde{M}$
a twisted Levi component of $\tilde{P}$. Let $\tilde{P}=\tilde{M}U$. We can
naturally extend the modulus character $\delta_{P}$ to $\tilde{P}\left(F\right)$. Namely, for any $\tilde{x}=\tilde{m}u \in \tilde{P}(F)$, we set $\delta_P(\tilde{x})=\det(\theta_{\tilde{m}\mid\mathfrak{u}})$. For a twisted Levi subgroup
$\tilde{M}$, we denote by ${\mathcal{P}}\left(\tilde{M}\right),\,{\mathcal{F}}\left(\tilde{M}\right),\,{\mathcal{L}}\left(\tilde{M}\right)$
the finite sets of twisted parabolic subgroups admitting
$\tilde{M}$ as a twisted Levi component, of twisted parabolic subgroups
containing $\tilde{M}$ and of twisted Levi subgroups containing $\tilde{M}$,
respectively. For twisted Levi subgroups $\tilde{M},\tilde{L}$ and
a twisted parabolic subgroup $\tilde{P}$, we notice that $\tilde{M}\subset\tilde{L}$
and $\tilde{M}\subset\tilde{P}$ imply $M\subset L$ and $M\subset P$,
respectively. Let $\tilde{Q}$ be a twisted parabolic subgroup. Then
$\tilde{Q}=\tilde{L}U$, where $\tilde{L}$ is a twisted Levi component
of $\tilde{Q}$ and $U$ is the unipotent radical of $Q$.
Twisted Levi subgroups are characterized as centralizers in $\tilde{G}$
of split tori. In other words, if $A$ is a split subtorus of $G$
such that $Z_{\tilde{G}}\left(A\right)\left(F\right)\neq\emptyset$,
then $Z_{\tilde{G}}\left(A\right)$ is a twisted Levi of $\tilde{G}$.
Conversely, if $\tilde{M}$ is a twisted Levi of $\tilde{G}$, then
$\tilde{M}=Z_{\tilde{G}}\left(A_{\tilde{M}}\right)$.

Let $W^{G}=\text{Norm}_{G\left(F\right)}\left(M_{\min}\right)/M_{\min}\left(F\right)$.
We denote $\tilde{P}_{\min}=N_{\tilde{G}}\left(P_{\min}\right)$ and $\tilde{M}_{\min}=N_{\tilde{G}}\left(P_{\min},M_{\min}\right)$.
Then $\tilde{P}_{\min}$ is a minimal twisted parabolic subgroup and
$\tilde{M}_{\min}$ is a minimal Levi component. Let ${\mathcal{L}}^{\tilde{G}}={\mathcal{L}}\left(\tilde{M}_{\min}\right)$.

We fix a maximal compact subgroup $K$ of $G\left(F\right)$ in good
relative position to $M_{\min}$. Let $\tilde{M}\in{\mathcal{L}}^{\tilde{G}}$
and $\tilde{P}=\tilde{M}U\in{\mathcal{P}}\left(\tilde{M}\right)$. One
has $G\left(F\right)=M\left(F\right)U\left(F\right)K$. We define
a map 
$$
\begin{array}{ccccc}
H_{\tilde{P}} & : & G\left(F\right) & \longrightarrow & {\mathcal{A}}_{\tilde{M}}\\
 &  & g=muk & \mapsto & H_{\tilde{M}}\left(m\right)
\end{array}.
$$

A twisted maximal torus of $\tilde{G}$ is a pair $\left(T,\tilde{T}\right)$
consisting of a maximal torus $T$ of $G$ defined over $F$ and a
subvariety $\tilde{T}$ of $\tilde{G}$ (defined over $F$), which
is the intersection of normalizers of $T$ and a Borel subgroup
$B$ (defined over $\bar{F}$) containing $T$ in $\tilde{G}$, such that $\tilde{T}\left(F\right)\neq\emptyset$.
For such pair, the restriction to $T$ of automorphisms $\theta_{\tilde{x}}$
for $\tilde{x}\in\tilde{T}$ does not depend on $\tilde{x}$. We denote
this restriction by $\theta_{\tilde{T}}$, or simply $\theta$ if
there is no confusion. Let $T_{\theta}$ be the connected component
of the subgroup of fixed points $T^{\theta}$. We denote by $\tilde{T}\left(F\right)/\theta$
the set of orbits of the action of $T\left(F\right)$ on $\tilde{T}\left(F\right)$
by conjugation. It is naturally an $F$-analytic manifold and for
all $\tilde{t}\in\tilde{T}\left(F\right)/\theta$, the map 
$$
\begin{array}{ccc}
T_{\theta}\left(F\right) & \rightarrow & \tilde{T}\left(F\right)/\theta\\
t & \mapsto & t\tilde{t}
\end{array}
$$
is a local isomorphism. If $T$ is split, we define a Haar measure
of $T\left(F\right)$ such that the volume of its maximal compact
subgroup is equal to $1$. In general, we provide $A_{T}\left(F\right)$
with this measure and choose a measure for $T\left(F\right)$ such
that $\text{vol}\left(T\left(F\right)/A_{T}\left(F\right)\right)=1$.
There exists a unique measure of $\tilde{T}\left(F\right)/\theta$
such that the above map preserves local measure for any $\tilde{t}\in\tilde{T}\left(F\right)/\theta$.
We equip $\tilde{T}\left(F\right)/\theta$ with this measure. Moreover, the principal homogeneous space $\tilde{G}(F)$ inherits the measure of $G(F)$. We have the Weyl integration formula 
$$
\int_{\tilde{G}\left(F\right)}\tilde{f}\left(\tilde{x}\right)d\tilde{x}=\underset{\tilde{T}\in{\mathcal{T}}\left(\tilde{G}\right)}{\sum}\left|W\left(G,\tilde{T}\right)\right|^{-1}\left|T^{\theta}\left(F\right):T_{\theta}\left(F\right)\right|^{-1}$$
$$\int_{\tilde{T}\left(F\right)/\theta}D^{\tilde{G}}\left(\tilde{t}\right)\int_{T_{\theta}\left(F\right)\backslash G\left(F\right)}f\left(g^{-1}\tilde{t}g\right)dgd\tilde{t},
$$
for any $\tilde{f}\in C_{c}^{\infty}\left(\tilde{G}\left(F\right)\right)$, where $dg$ is the measure defined by the quotient of ones on $G(F)$ and $T_\theta(F)$.

We denote $\mathcal{A}_{\tilde{G},F}=H_{\tilde{G}}\left(G\left(F\right)\right)$,
${\mathcal{A}}_{A_{\tilde{G}},F}=H_{\tilde{G}}\left(A_{\tilde{G}}\left(F\right)\right)$,
${\mathcal{A}}_{\tilde{G},F}^{\vee}=\text{Hom}\left({\mathcal{A}}_{\tilde{G},F},2\pi\mathbb{Z}\right)$,
${\mathcal{A}}_{A_{\tilde{G}},F}^{\vee}=\text{Hom}\left({\mathcal{A}}_{A_{\tilde{G}},F},2\pi\mathbb{Z}\right)$.
Then ${\mathcal{A}}_{\tilde{G},F}$ and ${\mathcal{A}}_{A_{\tilde{G}},F}$ are
lattices in ${\mathcal{A}}_{\tilde{G}}$, whereas ${\mathcal{A}}_{\tilde{G},F}^{\vee}$ and ${\mathcal{A}}_{A_{\tilde{G}},F}^{\vee}$ are lattices in ${\mathcal{A}}_{\tilde{G}}^{*}$.
We equip these lattices with the counting measure. We set Haar measures
on ${\mathcal{A}}_{\tilde{G}}$ and ${\mathcal{A}}_{\tilde{G}}^{*}$ such that
volumes of ${\mathcal{A}}_{\tilde{G}}/{\mathcal{A}}_{A_{\tilde{G}},F}$ and
${\mathcal{A}}_{\tilde{G}}^{*}/{\mathcal{A}}_{A_{\tilde{G}},F}^{\vee}$ are
equal to $1$, respectively.

For an affine algebraic variety $X$ over $\bar{F}$, let ${\mathcal{O}}\left(X\right)$
be the ring of regular functions. We choose a finite set of generators
$\left\{ f_{1},\ldots,f_{m}\right\} $ of ${\mathcal{O}}\left(X\right)$ as an $\bar{F}$-algebra. Define 
$$
\sigma_{X}\left(x\right)=1+\log\left(\max\left\{ 1,\left|f_{1}\left(x\right)\right|,\ldots,\left|f_{m}\left(x\right)\right|\right\} \right),\ \text{for } x\in X.
$$
Two such functions $\sigma_{X}$ and $\tilde{\sigma}_{X}$ are called
equivalent if $\sigma_{X}\sim\tilde{\sigma}_{X}$, i.e. there exists
$C_{1},C_{2}>0$ such that $C_{1}\tilde{\sigma}_{X}<\sigma_{X}<C_{2}\tilde{\sigma}_{X}$.
A log-norm on $X$ is a particular function $\sigma_{X}$ inside its
equivalence class. Generally, for any algebraic variety $X$ over
$\bar{F}$, we choose a finite open affine covering $\left(U_{i}\right)_{i\in I}$
of $X$ and fix log-norms $\sigma_{U_{i}}$ on $U_{i}$, for $i\in I$.
We define a log-norm on $X$ by letting 
$$
\sigma_{X}\left(x\right)=\inf\left\{ \sigma_{U_{i}}\left(x\right)\mid\ i\in I\text{ and }x\in U_{i}\right\} .
$$
We denote by $\Xi^G$ the Harish-Chandra function on $G(F)$ (see \cite[Section 1.5]{BP20} for a precise definition). Let us fix an element $\tilde{x}\in\tilde{G}\left(F\right)$. We define the Harish-Chandra-Schwartz
space ${\mathcal{C}}\left(\tilde{G}\left(F\right)\right)$ as the space
of functions $\tilde{f}\in C\left(\tilde{G}\left(F\right)\right)$
such that $f(g):=\tilde{f}(g\tilde{x})$ lies in $\mathcal{C}(G(F))$.

\subsection{Representations}\label{sec2.2}

A unitary representation of $G\left(F\right)$ is a continuous representation
$\left(\pi,V_{\pi}\right)$ of $G\left(F\right)$ on a Hilbert space
$V_{\pi}$ such that for any $g\in G\left(F\right)$, the operator
$\pi\left(g\right)$ is unitary. There is an action of $i{\mathcal{A}}_{G}^{*}$
on unitary representations given by $\left(\lambda,\pi\right)\mapsto\pi_{\lambda}$,
where $\pi_{\lambda}\left(g\right)=e^{\lambda\left(H_{G}\left(g\right)\right)}\pi\left(g\right)$
for any $g\in G\left(F\right)$. We denote by $i{\mathcal{A}}_{G,\pi}^{*}$
the stabilizer of $\pi$ for this action. Let $\text{End}\left(\pi\right)$ be the space of continuous endomorphisms of the space of $\pi$ and $\text{End}\left(\pi\right)^\infty$ be its subspace containing smooth vectors. From now, we assume any representations that we consider are of finite length. Let $\text{Temp}(G)$ and $\Pi_2(G)$ be the sets of isomorphism classes of irreducible tempered representations and irreducible square-integrable representations, respectively. 

Square-integrable representations are preserved by unramified twists. Let $\Pi_{2}\left(G\right)/i\mathcal{A}^*_{G,F}$ be the set of orbits in $\Pi_{2}\left(G\right)$ via this action. Let ${\mathcal{X}}_{\text{temp}}\left(G\right)$ be the set of isomorphism classes of tempered representations of $G\left(F\right)$ of the form $i_{M}^{G}\left(\sigma\right)$, where $M$ is a Levi subgroup of $G$ and $\sigma$ is an irreducible square-integrable representation of $M\left(F\right)$. Let ${\mathcal{M}}$ be a set of representatives for the conjugacy classes
of Levi subgroups of $G$. Then ${\mathcal{X}}_{\text{temp}}\left(G\right)$ is naturally a quotient of 
$$
\tilde{{\mathcal{X}}}_{\text{temp}}\left(G\right)=\underset{M\in{\mathcal{M}}}{\bigsqcup}\ \ \underset{{\mathcal{O}}\in\Pi_{2}\left(M\right)/i{\mathcal{A}}_{M,F}^{*}}{\bigsqcup}{\mathcal{O}}.
$$
Since each orbit $\mathcal{O}\in \{\Pi_{2}\left(M\right)\}$ is a quotient of $i{\mathcal{A}}_{M,F}^{*}$ by a finite subgroup, ${\mathcal{X}}_{\text{temp}}\left(G\right)$ is a real smooth manifold. We equip ${\mathcal{X}}_{\text{temp}}\left(G\right)$
with the quotient topology.

Let $V$ be a locally convex topological vector space. A function
$f:{\mathcal{X}}_{\text{temp}}\left(G\right)\rightarrow V$ is smooth if
the pullback of $f$ to $\tilde{{\mathcal{X}}}_{\text{temp}}\left(G\right)$
is a smooth function. We denote by $C^{\infty}\left({\mathcal{X}}_{\text{temp}}\left(G\right),V\right)$
the space of smooth functions on ${\mathcal{X}}_{\text{temp}}\left(G\right)$
taking values in $V$. For simplicity, we set $C^{\infty}\left({\mathcal{X}}_{\text{temp}}\left(G\right)\right)=C^{\infty}\left({\mathcal{X}}_{\text{temp}}\left(G\right),\mathbb{C}\right)$.

Let $R_{\text{temp}}\left(G\right)$ be the space of complex virtual tempered representations of $G\left(F\right)$, i.e. the complex vector space with basis $\text{Temp}\left(G\right)$ consisting of irreducible tempered representations of $G\left(F\right)$.

In \cite{Art93}, Arthur defines a set ${\mathcal{X}}_{\text{ell}}\left(G\right)$
of virtual tempered representations of $G\left(F\right)$ called elliptic
representations, which are actually well-defined up to scalar of module
1, i.e. ${\mathcal{X}}_{\text{ell}}\left(G\right)\subset R_{\text{temp}}\left(G\right)/\mathbb{S}^{1}$. Let $R_{\text{ell}}\left(G\right)$
be the subspace of $R_{\text{temp}}\left(G\right)$ generated by ${\mathcal{X}}_{\text{ell}}\left(G\right)$
and denote by $R_{\text{ind}}\left(G\right)$ the subspace of $R_{\text{temp}}\left(G\right)$
generated by the image of all the linear maps $i_{M}^{G}:R_{\text{temp}}\left(M\right)\rightarrow R_{\text{temp}}\left(G\right)$,
where $M$ is a proper Levi subgroup of $G$. We have 
$$
R_{\text{temp}}\left(G\right)=R_{\text{ind}}\left(G\right)\oplus R_{\text{ell}}\left(G\right).
$$
The set ${\mathcal{X}}_{\text{ell}}\left(G\right)$ is invariant under
unramified twists. Let ${\mathcal{X}}_{\text{ell}}\left(G\right)/i{\mathcal{A}}_{G,F}^{*}$
be the set of unramified orbits in ${\mathcal{X}}_{\text{ell}}\left(G\right)$.
Let $\underline{{\mathcal{X}}}_{\text{ell}}\left(G\right)$ be the inverse
image of ${\mathcal{X}}_{\text{ell}}\left(G\right)$ in $R_{\text{temp}}\left(G\right)$.
This set is invariant under multiplication by $\mathbb{S}^{1}$.

We denote by ${\mathcal{X}}\left(G\right)$ the subset of $R_{\text{temp}}\left(G\right)/\mathbb{S}^{1}$
consisting of virtual representations of the form $i_{M}^{G}\left(\sigma\right)$,
where $M$ is a Levi subgroup of $G$ and $\sigma\in{\mathcal{X}}_{\text{ell}}\left(M\right)$.
Also, let $\underline{{\mathcal{X}}}\left(G\right)$ be the inverse image
of ${\mathcal{X}}\left(G\right)$ in $R_{\text{temp}}\left(G\right)$.
The fibers of the natural projection $\underline{{\mathcal{X}}}\left(G\right)\rightarrow{\mathcal{X}}\left(G\right)$
are all isomorphic to $\mathbb{S}^{1}$. Let ${\mathcal{M}}$ be a set
of representatives for the conjugacy classes of Levi subgroups of
$G$. Then ${\mathcal{X}}\left(G\right)$ is naturally a quotient of 
$$
\underset{M\in{\mathcal{M}}}{\bigsqcup}\ \ \underset{{\mathcal{O}}\in{\mathcal{X}}_{\text{ell}}\left(M\right)/i{\mathcal{A}}_{M,F}^{*}}{\bigsqcup}{\mathcal{O}}.
$$
This defines a structure of topological space on ${\mathcal{X}}\left(G\right)$.
Let us define a regular Borel measure $d\pi$ on ${\mathcal{X}}\left(G\right)$
by 
$$
\int_{{\mathcal{X}}\left(G\left(F\right)\right)}\phi\left(\pi\right)d\pi=\underset{M\in{\mathcal{M}}}{\sum}\left|W\left(G,M\right)\right|^{-1}\underset{{\mathcal{O}}\in{\mathcal{X}}_{\text{ell}}\left(M\right)/i{\mathcal{A}}_{M,F}^{*}}{\sum}\left[i{\mathcal{A}}_{M,\sigma}^{\vee}:i{\mathcal{A}}_{M,F}^{\vee}\right]^{-1}\int_{i{\mathcal{A}}_{M,F}^{*}}\phi\left(i_{M}^{G}\left(\sigma_{\lambda}\right)\right)d\lambda,
$$
for any continuous and compactly supported function $\phi$ on ${\mathcal{X}}\left(G\right)$,
where a base point $\sigma\in{\mathcal{O}}$ is fixed for every orbit ${\mathcal{O}}\in{\mathcal{X}}_{\text{ell}}\left(M\right)/i{\mathcal{A}}_{M,F}^{*}$.

We extend the function $\pi\mapsto D\left(\pi\right)$ to
${\mathcal{X}}\left(G\right)$ by setting $D\left(\pi\right)=D\left(\sigma\right)$ for any $\pi=i_{M}^{G}\left(\sigma\right)$, where $M$ is a Levi subgroup and $\sigma\in{\mathcal{X}}_{\text{ell}}\left(M\right)$.

We now consider representations of a twisted group $\tilde{G}(F)$. A representation of $\tilde{G}\left(F\right)$ is a triple $\left(\pi,\tilde{\pi},E_{\pi}\right)$,
where $\pi$ is a smooth representation of $G\left(F\right)$ with an underlying space $E_{\pi}$ and $\tilde{\pi}:\tilde{G}\left(F\right)\rightarrow\text{Aut}_{\mathbb{C}}\left(E_{\pi}\right)$
satisfying $\tilde{\pi}\left(g\tilde{x}g^{\prime}\right)=\pi\left(g\right)\tilde{\pi}\left(\tilde{x}\right)\pi\left(g^{\prime}\right)$,
for any $g,g^{\prime}\in G\left(F\right)$ and $\tilde{x}\in\tilde{G}\left(F\right)$.
Two representations $\left(\pi_{1},\tilde{\pi}_{1},E_{\pi_{1}}\right)$
and $\left(\pi_{2},\tilde{\pi}_{2},E_{\pi_{2}}\right)$ are equivalent
if there exists linear isomorphisms $A:E_{\pi_{1}}\rightarrow E_{\pi_{2}}$
and $B:E_{\pi_{1}}\rightarrow E_{\pi_{2}}$ which intertwine $\pi_{1}$
and $\pi_{2}$ and satisfy $B\tilde{\pi}_{1}\left(\tilde{x}\right)=\tilde{\pi}_{2}\left(\tilde{x}\right)A$,
for any $\tilde{x}\in\tilde{G}\left(F\right)$. We say a representation
$\left(\pi,\tilde{\pi},E_{\pi}\right)$ of $\tilde{G}\left(F\right)$
is admissible if $\pi$ is, and unitary if there exists a positive
definite hermitian product which is invariant under the image of $\tilde{\pi}$.
A representation $\left(\pi,\tilde{\pi},E_{\pi}\right)$ is tempered
if it is unitary, and $\pi$ is of finite length and any irreducible
subrepresentations of $\pi$ are tempered. In general, we omit the
term $\left(\pi,E_{\pi}\right)$ and denote by $\tilde{\pi}$ a representation
of $\tilde{G}\left(F\right)$.

Let $\text{Temp}\left(\tilde{G}\right)$
be the set of $G\left(F\right)$-irreducible tempered representations
of $\tilde{G}\left(F\right)$. Let $R_{\text{temp}}\left(\tilde{G}\right)$
be the space of complex virtual tempered representations of $G\left(F\right)$,
i.e. the complex vector space with basis $\text{Temp}\left(\tilde{G}\right)$. We recall the subsets $E_\text{disc}(\tilde{G})$ and $E_\text{ell}(\tilde{G})$ of $\text{Temp}\left(\tilde{G}\right)/conj$ defined in \cite[Section 2.8]{BW23}.
We equip $E_\text{disc}(\tilde{G})$ with the unique measure such that for every $\tau \in E_{disc}(\tilde{G})$, the action map $\lambda \in i\mathcal{A}_{\tilde{G}}^* \mapsto \lambda \cdot r$ is locally measure preserving. For every sufficiently nice function $\varphi : E_\text{disc}(\tilde{G})\rightarrow \mathbb{C}$, we have
$$\int_{E_\text{disc}(\tilde{G})}\varphi(\tau) d\tau = \underset{\tau \in E_\text{disc}(\tilde{G})/i\mathcal{A}_{\tilde{G},F}^*}{\sum} |\text{Stab}(i\mathcal{A}_{\tilde{G},F}^*,\tau)|^{-1} \int_{i\mathcal{A}_{\tilde{G},F}^*} \varphi(\lambda \cdot \tau)d\lambda,$$
where $\text{Stab}(i\mathcal{A}_{\tilde{G},F}^*,\tau)$ is the stabilizer of $\tau$ in $i\mathcal{A}_{\tilde{G},F}^*$.

\subsection{$(G,M)$-families and $(\tilde{G},\tilde{M})$-families}\label{sec2.3}

We recall some facts about $(G,M)$-families in \cite[Section 17]{Art05}. Let $M$ be a Levi subgroup of $G$ and $V$ be a locally convex topological
space. A $\left(G,M\right)$-family with values in $V$ is a family
$\left(c_{P}\right)_{P\in{\mathcal{P}}\left(M\right)}$ of smooth functions
on $i{\mathcal{A}}_{M}^{*}$ taking values in $V$ such that for all adjacent
parabolic subgroups $P,P^{\prime}\in{\mathcal{P}}\left(M\right)$, the
functions $c_{P}$ and $c_{P^{\prime}}$ coincide on the hyperplane
supporting the wall that separates the positive chambers for $P$
and $P^{\prime}$. For any $\left(G,M\right)$-family $\left(c_{P}\right)_{P\in{\mathcal{P}}\left(M\right)}$,
Arthur associates an element $c_{M}$ of $V$ as follows. The function
$$
c_{M}\left(\lambda\right)=\underset{P\in{\mathcal{P}}\left(M\right)}{\sum}c_{P}\left(\lambda\right)\theta_{P}\left(\lambda\right)^{-1}
$$
extends to a smooth function on $i{\mathcal{A}}_{M}^{*}$ where 
$$
\theta_{P}\left(\lambda\right)=\text{meas}\left({\mathcal{A}}_{M}^{G}/\mathbb{Z}\Delta_{P}^{\vee}\right)^{-1}\underset{\alpha\in\Delta_{P}}{\prod}\lambda\left(\alpha^{\vee}\right),\ \ P\in{\mathcal{P}}\left(M\right)
$$
and set $c_{M}=c_{M}\left(0\right)$. Here $\Delta_{P}$ is the set
of simple roots of $A_{M}$ in $P$, $\Delta_{P}^{\vee}$ is the corresponding
set of simple coroots, and for every $\alpha\in\Delta_{P}$, $\alpha^{\vee}$
is denoted as the corresponding simple coroot.

A $\left(G,M\right)$-orthogonal set is a family $\left(Y_{P}\right)_{P\in{\mathcal{P}}\left(M\right)}$
of points in ${\mathcal{A}}_{M}$ such that for any adjacent parabolic
subgroups $P,P^{\prime}\in{\mathcal{P}}\left(M\right)$ there exists a
real number $r_{P,P^{\prime}}$ such that $Y_{P}-Y_{P^{\prime}}=r_{P,P^{\prime}}\alpha^{\vee}$,
where $\alpha$ is the unique root of $A_{M}$ that is positive for
$P$ and negative for $P^{\prime}$. If moreover we have $r_{P,P^{\prime}}\geq0$
for any adjacent $P,P^{\prime}\in{\mathcal{P}}\left(M\right)$, then we
say that the family is positive. Clearly if $\left(Y_{P}\right)_{P\in{\mathcal{P}}\left(M\right)}$
is a $\left(G,M\right)$-orthogonal set, then the family $\left(c_{P}\right)_{P\in{\mathcal{P}}\left(M\right)}$
defined by $c_{P}\left(\lambda\right)=e^{\lambda\left(Y_{P}\right)}$
is a $\left(G,M\right)$-family. If the family $\left(Y_{P}\right)_{P\in{\mathcal{P}}\left(M\right)}$
is positive, then $c_{M}$ is the volume in ${\mathcal{A}}_{M}^{G}$ of
the convex hull of the set $\left\{ Y_{P}:\ P\in{\mathcal{P}}\left(M\right)\right\} $.

Let $\tilde{M}$ be a twisted Levi subgroup of $\tilde{G}$. As in Section 2.3 in \cite{Wal12b}, we extend previous notions to the setting of twisted groups, which are $\left(\tilde{G},\tilde{M}\right)$-families and $\left(\tilde{G},\tilde{M}\right)$-orthogonal
sets. A family $\mathcal{Y}=(Y_{\tilde{P}})_{\tilde{P}\in \mathcal{P}(\tilde{M})}$ of points in $\mathcal{A}_{\tilde{M}}$ is a $(\tilde{G},\tilde{M})$-orthogonal set if for every adjacent twisted parabolic subgroups $\tilde{P},\tilde{Q}\in \mathcal{P}(\tilde{M})$, we have
$$Y_{\tilde{P}}-Y_{\tilde{Q}} \in \mathbb{R}\alpha_{\tilde{P},\tilde{Q}}^\vee,$$
where $\alpha_{\tilde{P},\tilde{Q}}^\vee$ is the unique root of $A_{\tilde{M}}$ that is positive for $\tilde{P}$ and negative for $\tilde{Q}$. Moreover, if $Y_{\tilde{P}}-Y_{\tilde{Q}}$ lies in $\mathbb{R}_{>0}\alpha_{\tilde{P},\tilde{Q}}^\vee$, for every pair of adjacent twisted parabolic subgroups $\tilde{P},\tilde{Q}\in \mathcal{P}(\tilde{M})$, we say $\mathcal{Y}$ is positive. Similar to the group case, if $\mathcal{Y}$ is positive, we can take $v^{\tilde{Q}}_{\tilde{L}}(\mathcal{Y})$ to be the volume of the convex hull of $(Y_{\tilde{P}})_{\tilde{P}\in \mathcal{P}(\tilde{L}),\tilde{P}\subset \tilde{Q}}$, for every $\tilde{L}\in \mathcal{L}(\tilde{M})$ and $\tilde{Q}\in \mathcal{F}(\tilde{L})$. We drop the superscript when $\tilde{Q}=\tilde{G}$.

\subsection{Harish-Chandra semisimple descent and the descent to Lie algebra}\label{sec2.4}

A subset $\omega\subset\mathfrak{g}\left(F\right)$ (resp. $\Omega\subset G\left(F\right)$)
is called completely $G\left(F\right)$-invariant if it is invariant
under $G\left(F\right)$-conjugation and for any $X\in\omega$ (resp.
$x\in\Omega$), its semisimple part $X_{s}$ (resp. $x_{s}$) under the Jordan decomposition also
lies in $\omega$ (resp. $\Omega$). Let us fix a completely $G\left(F\right)$-invariant
open subset $\omega\subset\mathfrak{g}\left(F\right)$ (resp. $\Omega\subset G\left(F\right)$).
Let $C^{\infty}\left(\omega\right)^{G}$ (resp. $C^{\infty}\left(\Omega\right)^{G}$)
be the space of smooth and $G\left(F\right)$-invariant functions
on $\omega$ (resp. $\Omega$). It is a closed subset of $C^{\infty}\left(\omega\right)$
(resp. $C^{\infty}\left(\Omega\right)$) and we endow it with the
induced locally convex topology.

Let $x\in G_{ss}\left(F\right)$. Let $\Omega_{x}\subseteq G_{x}\left(F\right)$ be a $G$-good open neighborhood of $x$ (see \cite[Section 3.2]{BP20} for this definition) and we set $\Omega=\Omega_{x}^{G}$. The following
integration formula holds for any $f$ which is integrable on $\Omega$.
$$
\int_{\Omega}f\left(y\right)dy=\int_{Z_{G}\left(x\right)\left(F\right)\backslash G\left(F\right)}\int_{\Omega_{x}}f\left(g^{-1}yg\right)\eta_{x}^{G}\left(y\right)dydg
$$
$$
=\left[Z_{G}\left(x\right)\left(F\right):G_{x}\left(F\right)\right]^{-1}\int_{G_{x}\left(F\right)\backslash G\left(F\right)}\int_{\Omega_{x}}f\left(g^{-1}yg\right)\eta_{x}^{G}\left(y\right)dydg.
$$
For a function $f$ on $\Omega$, let $f_{x,\Omega_{x}}$ be the
function on $\Omega_{x}$ given by $f_{x,\Omega_{x}}\left(y\right)=\eta_{x}^{G}\left(y\right)^{1/2}f\left(y\right)$.
The map $f\mapsto f_{x,\Omega_{x}}$ induces topological isomorphisms
$$
C^{\infty}\left(\Omega\right)^{G}\simeq C^{\infty}\left(\Omega_{x}\right)^{Z_{G}\left(x\right)}\text{ and }C^{\infty}\left(\Omega_{\text{rss}}\right)^{G}\simeq C^{\infty}\left(\Omega_{x,\text{ rss}}\right)^{Z_{G}\left(x\right)}.
$$
We now consider its Lie algebra counterpart. Let $\omega\subseteq\mathfrak{g}\left(F\right)$ be a $G$-excellent
open subset (see \cite[Section 3.3]{BP20} for this definition) and set $\Omega=\exp\left(\omega\right)$. The Jacobian of the exponential map 
$$
\begin{array}{ccccc}
\exp & : & \omega & \rightarrow & \Omega\\
 &  & X & \mapsto & e^{X}
\end{array}
$$
at $X\in\omega_{ss}$ is given by $j^{G}\left(X\right)=D^{G}\left(e^{X}\right)D^{G}\left(X\right)^{-1}$. Hence, the following integration formula holds for any $f$ which is integrable on $\Omega$
$$
\int_{\Omega}f\left(g\right)dg=\int_{\omega}f\left(e^{X}\right)j^{G}\left(X\right)dX.
$$
For any function $f$ on $\Omega$, we set $f_{\omega}$ the function
on $\omega$ defined by $f_{\omega}\left(X\right)=j^{G}\left(X\right)^{1/2}f\left(e^{X}\right)$.
The map $f\rightarrow f_{\omega}$ induces topological isomorphisms
$$
C^{\infty}\left(\Omega\right)\simeq C^{\infty}\left(\omega\right)\text{ and }C^{\infty}\left(\Omega_{\text{rss}}\right)\simeq C^{\infty}\left(\omega_{\text{rss}}\right).
$$
We can easily adapt Harish-Chandra descent to twisted groups, see \cite[Section 2.5-2.6]{Le25}.
\subsection{Quasi-characters}\label{sec2.5}

Let $\omega\subseteq\mathfrak{g}\left(F\right)$ be a completely $G\left(F\right)$-invariant
open subset. A quasi-character on $\omega$ is a $G\left(F\right)$-invariant
smooth function $\theta:\omega_{\text{reg}}\rightarrow\mathbb{C}$
satisfying the following condition: for any $X\in\omega_{ss}$, there
exists a $G$-good open neighborhood $\omega_{X}\subseteq\mathfrak{g}_{X}\left(F\right)$
of $X$ satisfying $\omega_{X}^{G}\subseteq\omega$ and coefficients
$c_{\theta,{\mathcal{O}}}\left(X\right)$, where ${\mathcal{O}}$ is in the set $\text{Nil}\left(\mathfrak{g}_{X}\right)$ containing nilpotent orbits of $\mathfrak{g}_X$,
such that 
$$
\theta\left(Y\right) = \underset{\mathcal{O} \in \text{Nil} \left(g_{X}\right)}{\sum}c_{\theta,{\mathcal{O}}}\left(X\right) \hat{j}\left(\mathcal{O},Y\right),\ \forall\,Y\in\omega_{X,\text{reg}}.
$$
Here $\hat{j}\left(\mathcal{O},\cdot\right)$ is the Fourier transform of the orbital integral of $\mathcal{O}$. If $\theta$ is a quasi-character on $\omega$ and $f\in C^{\infty}\left(\omega\right)^{G}$,
then $f\theta$ is also a quasi-character on $\omega$. We denote
by $QC\left(\omega\right)$ the space of all quasi-characters on $\omega$
and by $QC_{c}\left(\omega\right)$ the subspace of quasi-characters
on $\omega$ whose support is compact modulo conjugation.

Let $\Omega\subseteq G\left(F\right)$ be a completely $G\left(F\right)$-invariant
open subset. Similar to the setting of Lie algebras, a quasi-character on $\Omega$ is a $G\left(F\right)$-invariant
smooth function $\theta:\Omega_{\text{reg}}\rightarrow\mathbb{C}$
satisfying the following condition: for any $x\in\Omega_{ss}$, there
exists a $G_{x}$-excellent open neighborhood $\omega_{x}\subseteq\mathfrak{g}_{x}\left(F\right)$ of $0$ satisfying $\left(x\exp\left(\omega_{x}\right)\right)^{G}\subseteq\Omega$
and coefficients $c_{\theta,{\mathcal{O}}}\left(x\right)$, where ${\mathcal{O}}\in\text{Nil}\left(\mathfrak{g}_{x}\right)$,
such that 
$$
\theta\left(xe^{Y}\right)=\underset{{\mathcal{O}}\in\text{Nil}\left(g_{x}\right)}{\sum}c_{\theta,{\mathcal{O}}}\left(x\right)\hat{j}\left({\mathcal{O}},Y\right),\ \forall\,Y\in\omega_{x,\text{reg}}.
$$
We recall some basic properties of quasi-characters (c.f. \cite[Proposition 4.1.1]{BP20}).
\begin{proposition}\label{2.1}
$ $
\begin{enumerate}
\item For all $X\in\mathfrak{g}_{\text{reg}}\left(F\right)$, $\hat{j}\left(X,\cdot\right)$
is a quasi-character on $\mathfrak{g}\left(F\right)$. For any nilpotent orbit $\mathcal{O}$ in $\text{Nil}\left(\mathfrak{g}\right)$,
$\hat{j}\left({\mathcal{O}},\cdot\right)$ is a quasi-character on $\mathfrak{g}\left(F\right)$.
For every irreducible admissible representation $\pi$ of $G\left(F\right)$,
the character $\theta_{\pi}$ is a quasi-character on $G\left(F\right)$.
\item Let $\omega\subseteq\mathfrak{g}\left(F\right)$ be a $G$-excellent
open subset. Set $\Omega=\exp\left(\omega\right)$. The linear map
$$
\theta\mapsto\theta_{\omega}
$$
induces topological isomorphisms $QC\left(\Omega\right)\simeq QC\left(\omega\right)$
and $QC_{c}\left(\Omega\right)\simeq QC_{c}\left(\omega\right)$.
\item Let $x\in G_{ss}\left(F\right)$ and let $\Omega_{x}\subseteq G_{x}\left(F\right)$
be a $G$-good open neighborhood of $x$. Set $\Omega=\Omega_{x}^{G}$.
The linear map 
$$
\theta\mapsto\theta_{x,\Omega_{x}}
$$
induces topological isomorphisms $QC\left(\Omega\right)\simeq QC\left(\Omega_{x}\right)^{Z_{G}\left(x\right)\left(F\right)}$
and $QC_{c}\left(\Omega\right)\simeq QC_{c}\left(\Omega_{x}\right)^{Z_{G}\left(x\right)\left(F\right)}$.
\end{enumerate}
\end{proposition}
Let $\theta$ be a quasi-character on $G\left(F\right)$. For any $x\in G_{ss}\left(F\right)$, we have a local expansion 
$$
D^{G}\left(xe^{X}\right)^{1/2}\theta\left(xe^{X}\right)=D^{G}\left(xe^{X}\right)^{1/2}\underset{{\mathcal{O}}\in\text{Nil}_{\text{reg}}\left(g_{x}\right)}{\sum}c_{\theta,{\mathcal{O}}}\left(x\right)\hat{j}\left({\mathcal{O}},X\right)+O\left(\left|X\right|\right),
$$
for all $X\in\mathfrak{g}_{x,\text{reg}}\left(F\right)$ sufficiently
near $0$. By the homogeneity property of the functions $\hat{j}\left({\mathcal{O}},\cdot\right)$
and their linear independence, we can see that coefficients $c_{\theta,{\mathcal{O}}}\left(x\right)$,
where ${\mathcal{O}}\in\text{Nil}_{\text{reg}}\left(\mathfrak{g}_{x}\right)$,
are uniquely defined. We set 
$$
c_{\theta}\left(x\right)=\frac{1}{\left|\text{Nil}_{\text{reg}}\left(g_{x}\right)\right|}\underset{{\mathcal{O}}\in\text{Nil}_{\text{reg}}\left(g_{x}\right)}{\sum}c_{\theta,{\mathcal{O}}}\left(x\right),
$$
for all $x\in G_{ss}\left(F\right)$. This gives us a complex-valued
function $c_{\theta}$ on $G_{ss}\left(F\right)$. Similarly, for
any quasi-character $\theta$ on $\mathfrak{g}\left(F\right)$, we
can associate to it a function $c_{\theta}$ on $\mathfrak{g}_{ss}\left(F\right)$.
We recall \cite[Proposition 4.5.1]{BP20} for later use.
\begin{proposition}\label{2.2}
Let $\theta$ be a quasi-character on $G\left(F\right)$ and let $x\in G_{ss}\left(F\right)$.
We have the following properties. 
\begin{enumerate}
\item Assume $G_{x}$ is quasi-split. Let $B_{x}\subset G_{x}$ be a Borel
subgroup and $T_{\text{qd,}x}\subset B_{x}$ be a maximal torus (both
defined over $F$). Then
$$
D^{G}\left(x\right)^{1/2}c_{\theta}\left(x\right)=\left|W\left(G_{x},T_{\text{qd,}x}\right)\right|^{-1}\underset{x^{\prime}\in T_{\text{qd,}x}\left(F\right)\rightarrow x}{\lim}D^{G}\left(x^{\prime}\right)^{1/2}\theta\left(x^{\prime}\right).
$$
\item The function $\left(D^{G}\right)^{1/2}c_{\theta}$ is locally bounded
on $G\left(F\right)$. To be more precise, for any invariant and compact
modulo conjugation subset $L$ of $G\left(F\right)$, there exists
a continuous semi-norm $\nu_{L}$ on $QC\left(G\left(F\right)\right)$
such that 
$$
\underset{x\in L_{ss}}{\sup}D^{G}\left(x\right)^{1/2}\left|c_{\theta}\left(x\right)\right|\leq\nu_{L}\left(\theta\right),
$$
for all $\theta\in QC\left(G\left(F\right)\right)$.
\item Let $\Omega_{x}\subseteq G_{x}\left(F\right)$ be a $G$-good open
neighborhood of $x$. Then
$$
D^{G}\left(y\right)^{1/2}c_{\theta}\left(y\right)=D^{G_{x}}\left(y\right)^{1/2}c_{\theta_{x,\Omega_{x}}}\left(y\right),
$$
for any $y\in\Omega_{x,ss}$.
\end{enumerate}
\end{proposition}
We can easily extend the above setting to twisted groups, including the definition of quasi-characters, Proposition \ref{2.1} and Proposition \ref{2.2}.

\subsection{Strongly cuspidal functions}\label{sec2.6}
\begin{definition}\label{2.3}
Let $P=MU$ be a parabolic subgroup of $G$. For $f\in{\mathcal{C}}\left(G\left(F\right)\right)$,
we define 
$$
f^{P}\left(m\right)=\delta_{P}\left(m\right)^{1/2}\int_{U\left(F\right)}f\left(mu\right)du,\ \text{where }m\in M\left(F\right)
$$
as a function in ${\mathcal{C}}\left(M\left(F\right)\right)$.
\end{definition}

A function $f\in{\mathcal{C}}\left(G\left(F\right)\right)$ is called strongly
cuspidal if $f^{P}=0$ for any proper parabolic subgroup of $P$ of
$G$. We denote by ${\mathcal{C}}_{\text{scusp}}\left(G\left(F\right)\right)$
the subspace of strongly cuspidal functions in ${\mathcal{C}}\left(G\left(F\right)\right)$.
More generally, for a completely $G\left(F\right)$-invariant open
subset $\Omega\subseteq G\left(F\right)$, we set ${\mathcal{C}}_{\text{scusp}}\left(\Omega\right)={\mathcal{C}}\left(\Omega\right)\cap{\mathcal{C}}_{\text{scusp}}\left(G\left(F\right)\right)$.

Let $M$ be a Levi subgroup of $G$. As in Section \ref{sec2.1}, by fixing
a maximal compact subgroup $K$ of $G\left(F\right)$, we have the
following map $H_{P}:G\left(F\right)\rightarrow{\mathcal{A}}_{M},$ where $P\in{\mathcal{P}}\left(M\right)$. For every $g\in G\left(F\right)$,
the family $\left(H_{P}\left(g\right)\right)_{P\in{\mathcal{P}}\left(M\right)}$
is a positive $\left(G,M\right)$-orthogonal set. By the previous
section, this gives us a $\left(G,M\right)$-family $\left(v_{p}\left(g,\cdot\right)\right)_{P\in{\mathcal{P}}\left(M\right)}$
and the number $v_{M}\left(g\right)$ associated to this $\left(G,M\right)$-family,
i.e. the volume in ${\mathcal{A}}_{M}^{G}$ of the convex hull determined
by $\left(H_{P}\left(g\right)\right)_{P\in{\mathcal{P}}\left(M\right)}$.
The function $g\mapsto v_{M}\left(g\right)$ is left $M\left(F\right)$
and right $K$-invariant.

Let $x\in M\left(F\right)\cap G_{\text{rss}}\left(F\right)$. For
any $f\in{\mathcal{C}}\left(G\left(F\right)\right)$, we define the weighted
orbital integral of $f$ at $x$ by setting 
$$
J_{M}\left(x,f\right)=D^{G}\left(x\right)^{1/2}\int_{G_{x}\left(F\right)\backslash G\left(F\right)}f\left(g^{-1}xg\right)v_{M}\left(g\right)dg.
$$
This integral is absolutely convergent and defines a tempered distribution
$J_{M}\left(x,\cdot\right)$ on $G\left(F\right)$. More generally,
for the $\left(G,M\right)$-family $\left(v_{P}\left(g,\cdot\right)\right)_{P\in{\mathcal{P}}\left(M\right)}$,
it is possible to associate to it a complex number $v_{L}^{Q}\left(g\right)$,
for any $L\in{\mathcal{L}}\left(M\right)$ and $Q\in{\mathcal{F}}\left(L\right)$.
This allows us to define a tempered distribution $J_{L}^{Q}\left(x,\cdot\right)$
on $G\left(F\right)$, for any $L\in{\mathcal{L}}\left(M\right)$ and $Q\in{\mathcal{F}}\left(L\right)$,
by setting 
$$
J_{L}^{Q}\left(x,f\right)=D^{G}\left(x\right)^{1/2}\int_{G_{x}\left(F\right)\backslash G\left(F\right)}f\left(g^{-1}xg\right)v_{L}^{Q}\left(g\right)dg,\ \ f\in{\mathcal{C}}\left(G\left(F\right)\right).
$$
When $L=M$ and $Q=G$, this recovers the definition of $J_M$. We recall the following lemma in
\cite[Section 5.2]{BP20}.
\begin{lemma}\label{2.4}
Let $f\in{\mathcal{C}}_{\text{scusp}}\left(G\left(F\right)\right)$ be
a strongly cuspidal function and fix $x\in M\left(F\right)\cap G_{\text{rss}}\left(F\right)$.
Then 
\begin{enumerate}
\item For any $L\in{\mathcal{L}}\left(M\right)$ and $Q\in{\mathcal{F}}\left(L\right)$,
if $L\neq M$ or $Q\neq G$, one has
$$
J_{L}^{Q}\left(x,f\right)=0.
$$
\item If $x\notin M\left(F\right)_{\text{ell}}$, then 
$$
J_{M}^{G}\left(x,f\right)=0.
$$
\item For any $y\in G\left(F\right)$,
$$
J_{yMy^{-1}}^{G}\left(yxy^{-1},f\right)=J_{M}^{G}\left(x,f\right).
$$
\end{enumerate}
\end{lemma}

For a regular semisimple element $x\in G_{\text{rss}}\left(F\right)$, let $M\left(x\right)$
be the centralizer of $A_{G_{x}}$ in $G$. For a strongly cuspidal function $f\in{\mathcal{C}}_{\text{scusp}}\left(G\left(F\right)\right)$,
we set 
$$
\theta_{f}\left(x\right)=\left(-1\right)^{a_{G}-a_{M\left(x\right)}}\nu\left(G_{x}\right)^{-1}D^{G}\left(x\right)^{-1/2}J_{M\left(x\right)}^{G}\left(x,f\right),
$$
where $x\in G_{\text{rss}}\left(F\right)$. By Lemma \ref{2.4}, the function
$\theta_{f}$ is invariant under conjugation. Moreover, it is a quasi-character
of $G\left(F\right)$.

We define weighted characters of strongly cuspidal functions. Let $M$ be a Levi subgroup of $G$ and $K$ be the special maximal compact subgroup of $G\left(F\right)$. Let $\sigma$ be a tempered representation of $M\left(F\right)$. We fix $P\in{\mathcal{P}}\left(M\right)$. Following Arthur (cf. \cite{Art94}), for any $f\in{\mathcal{C}}\left(G\left(F\right)\right)$, $L\in{\mathcal{L}}\left(M\right)$
and $Q\in{\mathcal{F}}\left(L\right)$, we can define a weighted character $J_{L}^{Q}\left(\sigma,f\right)$. In particular, when $L=Q=G$, this reduces to the usual character, i.e. 
$
J_{G}^{G}\left(\sigma,f\right)=\text{Trace}_{i_{M}^{G}\left(\sigma\right)}\left(f\right)
$
for any $f\in{\mathcal{C}}\left(G\left(F\right)\right)$.

In Section \ref{sec2.2}, we have defined a set $\underline{{\mathcal{X}}}\left(G\right)$
of virtual tempered representations of $G\left(F\right)$. Let $\pi\in\underline{{\mathcal{X}}}\left(G\right)$.
There exists a pair $\left(M,\sigma\right)$, where $M$ is a Levi
subgroup of $G$ and $\sigma\in\underline{{\mathcal{X}}}_{\text{ell}}\left(M\right)$,
such that $\pi=i_{M}^{G}\left(\sigma\right)$. For any $f\in{\mathcal{C}}_{\text{scusp}}\left(G\left(F\right)\right)$,
we set 
$$
\hat{\theta}_{f}\left(\pi\right)=\left(-1\right)^{a_{G}-a_{M}}J_{M}^{G}\left(\sigma,f\right).
$$
By \cite[Lemma 5.4.1]{BP20}, this definition is well-defined since the pair $\left(M,\sigma\right)$
is well-defined up to conjugation.

Likewise, we can extend the definition of strongly cuspidal functions to twisted groups and we denote the space of strongly cuspidal functions on $\tilde{G}\left(F\right)$
by ${\mathcal{C}}_{\text{scusp}}\left(\tilde{G}\left(F\right)\right)$.
\subsection{Weighted orbital integrals of strongly cuspidal functions}\label{sec2.7}

Let $M$ be a Levi subgroup of $G$. As in Section \ref{sec2.1}, by fixing
a maximal compact subgroup $K$ of $G\left(F\right)$, we have the
following map $H_{P}:G\left(F\right)\rightarrow{\mathcal{A}}_{M},$ where $P\in{\mathcal{P}}\left(M\right)$. For every $g\in G\left(F\right)$,
the family $\left(H_{P}\left(g\right)\right)_{P\in{\mathcal{P}}\left(M\right)}$
is a positive $\left(G,M\right)$-orthogonal set. By the previous
section, this gives us a $\left(G,M\right)$-family $\left(v_{p}\left(g,\cdot\right)\right)_{P\in{\mathcal{P}}\left(M\right)}$
and the number $v_{M}\left(g\right)$ associated to this $\left(G,M\right)$-family,
i.e. the volume in ${\mathcal{A}}_{M}^{G}$ of the convex hull determined
by $\left(H_{P}\left(g\right)\right)_{P\in{\mathcal{P}}\left(M\right)}$.
The function $g\mapsto v_{M}\left(g\right)$ is left $M\left(F\right)$
and right $K$-invariant.

Let $x\in M\left(F\right)\cap G_{\text{rss}}\left(F\right)$. For
any $f\in{\mathcal{C}}\left(G\left(F\right)\right)$, we define the weighted
orbital integral of $f$ at $x$ by setting 
$$
J_{M}\left(x,f\right)=D^{G}\left(x\right)^{1/2}\int_{G_{x}\left(F\right)\backslash G\left(F\right)}f\left(g^{-1}xg\right)v_{M}\left(g\right)dg.
$$
This integral is absolutely convergent and defines a tempered distribution
$J_{M}\left(x,\cdot\right)$ on $G\left(F\right)$. More generally,
for the $\left(G,M\right)$-family $\left(v_{P}\left(g,\cdot\right)\right)_{P\in{\mathcal{P}}\left(M\right)}$,
it is possible to associate to it a complex number $v_{L}^{Q}\left(g\right)$,
for any $L\in{\mathcal{L}}\left(M\right)$ and $Q\in{\mathcal{F}}\left(L\right)$.
This allows us to define a tempered distribution $J_{L}^{Q}\left(x,\cdot\right)$
on $G\left(F\right)$, for any $L\in{\mathcal{L}}\left(M\right)$ and $Q\in{\mathcal{F}}\left(L\right)$,
by setting 
$$
J_{L}^{Q}\left(x,f\right)=D^{G}\left(x\right)^{1/2}\int_{G_{x}\left(F\right)\backslash G\left(F\right)}f\left(g^{-1}xg\right)v_{L}^{Q}\left(g\right)dg,\ \ f\in{\mathcal{C}}\left(G\left(F\right)\right).
$$
When $L=M$ and $Q=G$, this recovers the definition of $J_M$. We recall the following lemma in
\cite[Section 5.2]{BP20}.
\begin{lemma}\label{2.5}
Let $f\in{\mathcal{C}}_{\text{scusp}}\left(G\left(F\right)\right)$ be
a strongly cuspidal function and fix $x\in M\left(F\right)\cap G_{\text{rss}}\left(F\right)$.
Then 
\begin{enumerate}
\item For any $L\in{\mathcal{L}}\left(M\right)$ and $Q\in{\mathcal{F}}\left(L\right)$,
if $L\neq M$ or $Q\neq G$, one has
$$
J_{L}^{Q}\left(x,f\right)=0.
$$
\item If $x\notin M\left(F\right)_{\text{ell}}$, then 
$$
J_{M}^{G}\left(x,f\right)=0.
$$
\item For any $y\in G\left(F\right)$,
$$
J_{yMy^{-1}}^{G}\left(yxy^{-1},f\right)=J_{M}^{G}\left(x,f\right).
$$
\end{enumerate}
\end{lemma}

For a regular semisimple element $x\in G_{\text{rss}}\left(F\right)$, let $M\left(x\right)$
be the centralizer of $A_{G_{x}}$ in $G$. For a strongly cuspidal function $f\in{\mathcal{C}}_{\text{scusp}}\left(G\left(F\right)\right)$,
we set 
$$
\theta_{f}\left(x\right)=\left(-1\right)^{a_{G}-a_{M\left(x\right)}}\nu\left(G_{x}\right)^{-1}D^{G}\left(x\right)^{-1/2}J_{M\left(x\right)}^{G}\left(x,f\right),
$$
where $x\in G_{\text{rss}}\left(F\right)$. By Lemma \ref{2.5}, the function
$\theta_{f}$ is invariant under conjugation. Moreover, it is a quasi-character
of $G\left(F\right)$.

We consider the case of twisted groups. Let $\tilde{M}\in{\mathcal{L}}^{\tilde{G}}$. For any $g\in G\left(F\right)$, we have  $\left(H_{\tilde{P}}\left(g\right)\right)_{\tilde{P}\in{\mathcal{P}}\left(\tilde{M}\right)}$
is $\left(\tilde{G},\tilde{M}\right)$-orthogonal and hence we can
associate to it a $\left(\tilde{G},\tilde{M}\right)$-family $\left(v_{\tilde{P}}\left(g\right)\right)_{\tilde{P}\in{\mathcal{P}}\left(\tilde{M}\right)}$
via $v_{\tilde{P}}\left(g,\lambda\right)=e^{-\lambda\left(H_{\tilde{P}}\left(g\right)\right)}$
for any $\lambda\in i{\mathcal{A}}_{\tilde{M}}^{*}$. For this $\left(\tilde{G},\tilde{M}\right)$-family,
we deduce a number $v_{\tilde{M}}\left(g\right)$. We are now able
to define weight orbital integrals. For $\tilde{f}\in C_{c}^{\infty}\left(\tilde{G}\left(F\right)\right)$
and $\tilde{x}\in\tilde{M}\left(F\right)\cap\tilde{G}_{\text{reg}}\left(F\right)$,
let 
$$
J_{\tilde{M}}\left(\tilde{x},\tilde{f}\right)=D^{\tilde{G}}\left(\tilde{x}\right)^{1/2}\int_{G_{\tilde{x}}\left(F\right)\backslash G\left(F\right)}\tilde{f}\left(g^{-1}\tilde{x}g\right)v_{\tilde{M}}\left(g\right)dg.
$$
Let $\tilde{f}$ be a strongly cuspidal function on $\tilde{G}\left(F\right)$.
We associate a quasi-character $\theta_{\tilde{f}}$ in the following
way. Let $\tilde{x}\in\tilde{G}_{\text{reg}}\left(F\right)$ and $\tilde{M}\left(\tilde{x}\right)$
be the centralizer of $A_{G_{\tilde{x}}}$ in $\tilde{G}$. It is
a twisted Levi subgroup of $\tilde{G}$ and let $g\in G\left(F\right)$
such that $g\tilde{M}\left(\tilde{x}\right)g^{-1}$ is a semistandard
Levi. We define 
$$
\theta_{\tilde{f}}\left(\tilde{x}\right)=\left(-1\right)^{a_{\tilde{M}\left(\tilde{x}\right)}-a_{\tilde{G}}}D^{\tilde{G}}\left(\tilde{x}\right)^{-1/2}J_{g\tilde{M}\left(\tilde{x}\right)g^{-1}}\left(g\tilde{x}g^{-1},\,^{g}\tilde{f}\right),
$$
where $^{g}\tilde{f}\left(\tilde{x}\right)=\tilde{f}\left(g^{-1}\tilde{x}g\right)$.
The definition does not depend on choices of $g$. Similar to the untwisted setting, the function $\theta_{\tilde{f}}$ is a quasi-character.

\subsection{Weighted characters of strongly cuspidal functions}\label{sec2.8}

Let $M$ be a Levi subgroup of $G$ and $K$ be the special maximal
compact subgroup of $G\left(F\right)$. Let $\sigma$ be a tempered
representation of $M\left(F\right)$. We fix $P\in{\mathcal{P}}\left(M\right)$.
Following Arthur (cf. \cite{Art94}), for any $f\in{\mathcal{C}}\left(G\left(F\right)\right)$, $L\in{\mathcal{L}}\left(M\right)$
and $Q\in{\mathcal{F}}\left(L\right)$, we can define a weighted character $J_{L}^{Q}\left(\sigma,f\right)$. In particular, when $L=Q=G$, this reduces to the usual character, i.e. 
$
J_{G}^{G}\left(\sigma,f\right)=\text{Trace}_{i_{M}^{G}\left(\sigma\right)}\left(f\right)
$
for any $f\in{\mathcal{C}}\left(G\left(F\right)\right)$.

Let $\tilde{P}=\tilde{M}U$ be a (semi-standard) twisted parabolic subgroup of $\tilde{G}$
and $\tilde{\tau}$ be a tempered representation of $\tilde{M}\left(F\right)$.
The definition of weighted characters can be extended to twisted groups. We denote a weighted character of $\tilde{\tau}$ by $J_{\tilde{M}}^{\tilde{G}}\left(\tilde{\tau},\tilde{f}\right)$, for any $\tilde{f}\in C_{c}^{\infty}\left(\tilde{G}\left(F\right)\right)$.
When $\tilde{M}=\tilde{G}$, it is the character of $\tilde{\tau}$,
which is denoted by $\theta_{\tilde{\tau}}$. By \cite[Theorem 2]{Clo87},
we have $\theta_{\tilde{\tau}}$ is a locally integrable distribution.

In Section \ref{sec2.2}, we have defined a set $\underline{{\mathcal{X}}}\left(G\right)$
of virtual tempered representations of $G\left(F\right)$. Let $\pi\in\underline{{\mathcal{X}}}\left(G\right)$.
There exists a pair $\left(M,\sigma\right)$, where $M$ is a Levi
subgroup of $G$ and $\sigma\in\underline{{\mathcal{X}}}_{\text{ell}}\left(M\right)$,
such that $\pi=i_{M}^{G}\left(\sigma\right)$. For any $f\in{\mathcal{C}}_{\text{scusp}}\left(G\left(F\right)\right)$,
we set 
$$
\hat{\theta}_{f}\left(\pi\right)=\left(-1\right)^{a_{G}-a_{M}}J_{M}^{G}\left(\sigma,f\right).
$$
By \cite[Lemma 5.4.1]{BP20}, this definition is well-defined since the pair $\left(M,\sigma\right)$
is well-defined up to conjugacy. We can also adapt the definition of $\hat{\theta}_f$ to twisted groups.

\subsection{Quasi-characters attached to strongly cuspidal functions}\label{sec2.9}

The following proposition is established in \cite[Proposition 5.6.1(ii)]{BP20}.
\begin{proposition}\label{2.6}
Let $f\in{\mathcal{C}}_{\text{scusp}}\left(G\left(F\right)\right)$. Then
the function $\theta_{f}$ is a quasi-character on $G\left(F\right)$
and we have an equality of quasi-characters 
$$
\theta_{f}\left(x\right)=\int_{{\mathcal{X}}\left(G\right)}D\left(\pi\right)\hat{\theta}_{f}\left(\pi\right)\bar{\theta}_{\pi}\left(x\right)d\pi
$$
where the integral on the right hand side is absolutely convergent
in $QC\left(G\left(F\right)\right)$.
\end{proposition}
We recall its twisted groups counterpart.
\begin{proposition}\label{2.7}
Let $\tilde{f}\in{\mathcal{C}}_{\text{scusp}}\left(\tilde{G}\left(F\right)\right)$.
Then we have an equality of quasi-characters 
$$
\theta_{\tilde{f}}=\underset{\tilde{M}\in\mathcal{L}(\tilde{M}_{\min})}{\sum}|\widetilde{W}^M||\widetilde{W}^G|^{-1}(-1)^{a_{\tilde{M}}-a_{\tilde{G}}}\int_{E_{\text{ell}}(\tilde{M})}D\left(\tilde{\pi}\right)\hat{\theta}_{\tilde{f}}\left(\tilde{\pi}\right)\overline{\theta_{\tilde{\pi}}}d\tilde{\pi}
$$
where the above integral is absolutely convergent in $\text{QC}\left(\tilde{G}\left(F\right)\right)$.
\end{proposition}

\begin{proof}
See \cite[Proposition 2.10]{BW23}.
\end{proof}
In \cite[Section 2.12]{Le25}, the author extends \cite[Proposition 5.7.1]{BP20} and \cite[Corollary 5.7.2]{BP20} to twisted groups. We recall \cite[Proposition 2.17 and Corollary 2.18]{Le25}.
\begin{proposition}\label{2.8}
Let $\tilde{x}\in\tilde{G}\left(F\right)_{\text{ell}}$ be an elliptic
element and let $\Omega_{\tilde{x}}\subset G_{\tilde{x}}\left(F\right)$
be a $G$-good open neighborhood of $1$ which is relatively compact
modulo $G_{\tilde{x}}$-conjugation. Set $\Omega=\left(\Omega_{\tilde{x}}\tilde{x}\right)^{G}$.
Then there exists a linear map 
$$
\begin{array}{ccc}
{\mathcal{C}}_{\text{scusp}}\left(\Omega_{\tilde{x}}\right) & \rightarrow & {\mathcal{C}}_{\text{scusp}}\left(\Omega\right)\\
f & \mapsto & \tilde{f}
\end{array}
$$
such that the following properties hold.
\begin{enumerate}
\item For any $f\in{\mathcal{C}}_{\text{scusp}}\left(\Omega_{\tilde{x}}\right)$, we have 
$$
\left(\theta_{\tilde{f}}\right)_{\tilde{x},\Omega_{\tilde{x}}}=\underset{z\in Z_{G}\left(\tilde{x}\right)\left(F\right)/G_{\tilde{x}}\left(F\right)}{\sum}{}^{z}\theta_{f}.
$$
\item There exists a function $\alpha\in C_{c}^{\infty}\left(Z_{G}\left(\tilde{x}\right)\left(F\right)\backslash G\left(F\right)\right)$
satisfying 
$$
\int_{Z_{G}\left(\tilde{x}\right)\left(F\right)\backslash G\left(F\right)}\alpha\left(g\right)dg=1
$$
and such that for any $f\in{\mathcal{C}}_{\text{scusp}}\left(\Omega_{\tilde{x}}\right)$
and $g\in G\left(F\right)$, there exists $z\in Z_{G}\left(\tilde{x}\right)\left(F\right)$
with 
$$
\left(^{zg}\tilde{f}\right)_{\tilde{x},\Omega_{\tilde{x}}}=\alpha\left(g\right)f.
$$
\end{enumerate}
\end{proposition}

\begin{corollary}\label{2.9}
Let $\chi$ be a character of $A_{\tilde{G}}\left(F\right)$. Then 
\begin{enumerate}
\item There exists $\Omega\subseteq A_{\tilde{G}}\left(F\right)\backslash\tilde{G}\left(F\right)$
a completely $G\left(F\right)$-invariant open subset which is relatively compact modulo conjugation and contains $A_{\tilde{G}}\left(F\right)\backslash\tilde{G}\left(F\right)_{\text{ell}}$ such that the linear map 
$$
f\in{\mathcal{C}}_{\text{scusp}}\left(\Omega,\chi\right)\mapsto\theta_{f}\in QC_{c}\left(\Omega,\chi\right)
$$
is surjective.
\item For all $\theta\in QC\left(\tilde{G}\left(F\right)\right)$, there
exists a compact subset $\Omega_{\theta}\subseteq{\mathcal{X}}_{\text{ell}}\left(\tilde{G}\right)$
such that 
$$
\int_{\Gamma_{\text{ell}}\left(\tilde{G}\right)}D^{\tilde{G}}\left(x\right)\theta\left(x\right)\theta_{\tilde{\pi}}\left(x\right)dx=0
$$
for all $\tilde{\pi}\in{\mathcal{X}}_{\text{ell}}\left(\tilde{G}\right)-\Omega_{\theta}$,
the integral above being absolutely convergent.
\item For all $\tilde{\pi}\in{\mathcal{X}}_{\text{ell}}\left(\tilde{G}\right)$,
there exists $f\in{\mathcal{C}}_{\text{scusp}}\left(\tilde{G}\left(F\right)\right)$
such that for all $\tilde{\pi}^{\prime}\in{\mathcal{X}}_{\text{ell}}\left(\tilde{G}\right)$,
we have 
$$
\hat{\theta}_{f}\left(\tilde{\pi}^{\prime}\right)\neq0\Leftrightarrow\tilde{\pi}^{\prime}=\tilde{\pi}.
$$
\end{enumerate}
\end{corollary}

\section{The local Langlands correspondence for unitary groups}\label{llc}
\subsection{Endoscopy and twisted endoscopy}\label{llc1}
Let $G$ be a connected reductive group defined over $F$ and $\left(H,s,^{L}\xi\right)$
be an endoscopic triple of $G$ (cf. \cite[1.2]{LS87}). This define
a map called endoscopic correspondence between $H_{\text{reg}}\left(F\right)/\text{stconj}$
and $G_{\text{reg}}\left(F\right)/\text{stconj}$. We say $x\in G_{\text{reg}}\left(F\right)/\text{conj}$
and $y\in H_{\text{reg}}\left(F\right)/\text{stconj}$ are corresponding
if they have the same image via $\phi_{G}^{\text{st}}/\text{conj}$.
In order to define transfer factors, we fix a quasi-split inner forms
$\underline{G}$ of $G$ and an inner torsor $\psi_{G}:G\rightarrow\underline{G}$
and a pinning of $\underline{G}$ defined over $F$ up to $\underline{G}\left(F\right)$-conjugation.
The last datum amount to a regular nilpotent orbit of $\underline{\mathfrak{g}}\left(F\right)$. With these data, we can define relative transfer factors $\Delta_{H,G}\left(y,x;\bar{y},\bar{x}\right)$, where $y\in H_{\text{reg}}\left(F\right)/\text{stconj}$ and $x\in G_{\text{reg}}\left(F\right)/\text{conj}$
correspond to $\bar{y}\in H_{\text{reg}}\left(F\right)/\text{stconj}$
and $\bar{x}\in G_{\text{reg}}\left(F\right)/\text{conj}$. According
to a remark by Kottwitz, if we fix a cocycle $u:\text{Gal}\left(\bar{F}/F\right)\rightarrow\underline{G}$ such that $\psi_{G}\circ\sigma\left(g\right)=u\left(\sigma\right)\sigma\circ\psi_{G}\left(g\right)u\left(\sigma\right)^{-1}$, for any $g\in G$ and $\sigma\in\text{Gal}\left(\bar{F}/F\right)$,
then we can define absolute transfer factors $\Delta_{H,G}\left(y,x\right)$. Noting that transfer factors depend on the cocycle $u$ and not only on its cohomology class. We now assume that we fix a cocycle $u$.

Let $\Theta$ be a locally integrable distribution on $G\left(F\right)$
which is invariant under conjugation and $\Theta^{H}$ be a locally
integrable distribution on $H\left(F\right)$ which is stably invariant.
We say $\Theta$ is a transfer of $\Theta^{H}$ if 
\[
D^{G}\left(x\right)^{1/2}\Theta\left(x\right)=\underset{y}{\sum}D^{H}\left(y\right)^{1/2}\Theta^{H}\left(y\right)\Delta_{H,G}\left(y,x\right),
\]
for any $x\in G_{\text{reg}}\left(F\right)/\text{conj}$, where the
sum is over $y\in H_{\text{reg}}\left(F\right)/\text{stconj}$ that
matches $x$.

Let $\left(M,\tilde{M}\right)$ be a twisted group. Assume $M$ is
split and we fix an element $\delta\in\tilde{M}\left(F\right)$, which gives us an automorphism $\theta_\delta$ of $M$.
We assume there exists a pinning of $M$ defined over $F$ and invariant
under $\theta_{\delta}$. We then fix a regular nilpotent orbit of $\mathfrak{m}_{\delta}\left(F\right)$. Let $\left(H,s,^{L}\xi\right)$
be an endoscopic triple of $\left(M,\tilde{M}\right)$ (cf. \cite[2.1]{KS99}).
We then have a twisted endoscopic correspondence between $H_{\text{reg}}\left(F\right)/\text{stconj}$ and $\tilde{M}_{\text{reg}}\left(F\right)/\text{stconj}$. We say
$\tilde{x}\in\tilde{M}_{\text{reg}}\left(F\right)/\text{conj}$ and
$y\in H_{\text{reg}}\left(F\right)/\text{stconj}$ match if they have the same image via $\phi_{\tilde{M}}^{\text{st}}/\text{conj}$. We define a transfer factor $\Delta_{H,\tilde{M}}\left(y,\tilde{x}\right)$ for any $\tilde{x}\in\tilde{M}_{\text{reg}}\left(F\right)/\text{conj}$ and $y\in H_{\text{reg}}\left(F\right)/\text{stconj}$ which is nonzero if and only if $\tilde{x}$ and $y$ correspond. Let $\tilde{\Theta}$ be a locally integrable distribution on $\tilde{M}\left(F\right)$ which is invariant under conjugation and $\Theta^{H}$ be a locally integrable distribution on $H\left(F\right)$ which is stably invariant. Similar to endoscopy transfers, we say $\tilde{\Theta}$ is a transfer of $\Theta^{H}$ if 
\[
D^{\tilde{M}}\left(\tilde{x}\right)^{1/2}\tilde{\Theta}\left(\tilde{x}\right)=\underset{y}{\sum}D^{H}\left(y\right)^{1/2}\Theta^{H}\left(y\right)\Delta_{H,\tilde{M}}\left(y,\tilde{x}\right),
\]
for any $\tilde{x}\in\tilde{M}_{\text{reg}}\left(F\right)/\text{conj}$,
where the sum is over $y\in H_{\text{reg}}\left(F\right)/\text{stconj}$
that corresponds to $\tilde{x}$.

\subsection{Base change from $\text{GL}_{n}$ to $\text{GL}_{n}\times\text{GL}_{n}$
and twisted endoscopy}\label{llc2}

In this subsection, we consider the setting when $K=E$. Let $E/F$
be a quadratic extension of $p$-adic fields. Let $\sigma$ be the
nontrivial element in $\text{Gal}\left(E/F\right)$ and we denote
$\bar{x}=\sigma\left(x\right)$ for any $x\in E$. Let 
\[
G=\text{Res}_{E/F}\text{GL}_{n}\overset{\Delta}{\hookrightarrow}M=\text{Res}_{E/F}\text{GL}_{n}\times\text{GL}_{n}.
\]
Let $\theta_{n}:\left(g,h\right)\mapsto\left(J_{n}{}^{t}\bar{h}^{-1}J_{n}^{-1},J_{n}{}^{t}\bar{g}^{-1}J_{n}^{-1}\right)$
be an involution on $M$, where 
$$
J_{n}=\left(\begin{array}{cccc}
0 & \cdots & 0 & -1\\
\vdots & 0 & 1 & 0\\
0 & \ddots & 0 & \vdots\\
\left(-1\right)^{n} & 0 & \cdots & 0
\end{array}\right).
$$
Let $\tilde{M}=M\theta_{n}$. We have the following 1-1 correspondence between $G_{\text{reg}}\left(F\right)/\text{conj}$ and $\tilde{M}_{\text{reg}}\left(F\right)/\text{conj}$
\[
\begin{array}{ccc}
G_{\text{reg}}\left(F\right)/\text{conj} & \longleftrightarrow & \tilde{M}_{\text{reg}}\left(F\right)/\text{conj}\\
x & \mapsto & \left(x,1\right)\theta_{n}
\end{array},
\]
whose inverse is $(x,y)\theta_n\mapsto xJ_n{}^t\bar{y}^{-1}J_n^{-1}$. In this case, the transfer factor $\Delta_{G,\tilde{M}}\left(y,\tilde{x}\right)$ is equal to $1$ whenever $y$ and $\tilde{x}$ are corresponding.

Let $\pi$ be a tempered irreducible representation of $G\left(F\right)$.
Then $\pi\times{}^{\sigma}\pi^{\vee}$ is a tempered irreducible representation of $M\left(F\right)$. By taking $\widetilde{\pi\times{}^{\sigma}\pi^{\vee}}\left(\theta_{n}\right)\left(v\otimes w\right)=\pi\left(J_{n}\right)w\otimes\pi\left(J_{n}\right)v$,
we extend it to a representation
$\widetilde{\pi\times{}^{\sigma}\pi^{\vee}}$ of $\tilde{M}\left(F\right)$. We recall a twisted endoscopic character identity in this case, which is proved in \cite[Theorem 7.1]{Le25}.
\begin{theorem}\label{basechangeGL}
Let $\pi$ be a tempered irreducible representation of $G\left(F\right)$.
By taking the normalization stated above, we have 
\[
D^{G}\left(y\right)^{1/2}\Theta_{\pi}\left(y\right)=D^{\tilde{M}}\left(\tilde{x}\right)^{1/2}\Theta_{\widetilde{\pi\times{}^{\sigma}\pi^{\vee}}}\left(\tilde{x}\right),
\]
for any $y$ and $\tilde{x}$ match.
\end{theorem}

\subsection{Space of parameters in unitary groups}\label{llc3}
Following \cite{BP15}, we give a parametrization of conjugacy classes in unitary groups given by skew-hermitian spaces relative to a quadratic extension $E/F$. We fix an algebraic closure $\bar{F}$ of $F$ containing $E$ and consider finite extensions of $F$ inside $\bar{F}$. Let $\underline{\Xi}$ be the set containing $\xi=\left(I,\left(F_{\pm i}\right)_{i\in I},\left(F_{i}\right)_{i\in I},\left(y_{i}\right)_{i\in I}\right)$,
where 
\begin{itemize}
\item $I$ is a finite set;
\item For any $i\in I$, $F_{\pm i}$ is a finite extension of $F$ and
$F_{i}=F_{\pm i}\otimes_{F}E$. We have $\tau_{i}=\text{Id}\otimes\tau_{E/F}$
is the unique nontrivial $F_{\pm i}$-automorphism of $F_{i}$;
\item For any $i\in I$, $y_{i}$ is an element of $F_{i}^{\times}$ such
that $y_{i}\tau_{i}\left(y_{i}\right)=1$.
\end{itemize}
For such $\xi$, we denote by $I^{*}$ the subset containing $i\in I$
such that $F_{i}$ is a field and we set 
\[
d_{\xi}=\underset{i\in I}{\sum}\left[F_{\pm i}:F\right]=\underset{i\in I}{\sum}\left[F_{i}:E\right].
\]
For $d\in\mathbb{N}$, let $\underline{\Xi}_{d}$ be the subset containing
$\xi\in\underline{\Xi}$ such that $d_{\xi}=d$ and $\underline{\Xi}^{*}$
be the subset containing $\xi\in\underline{\Xi}$ such that $I^{*}=I$
and $\underline{\Xi}_{d}^{*}=\underline{\Xi}_{d}\cap\underline{\Xi}^{*}$.
We define an isomorphism between two elements $\xi=\left(I,\left(F_{\pm i}\right)_{i\in I},\left(F_{i}\right)_{i\in I},\left(y_{i}\right)_{i\in I}\right)$
and $\xi^{\prime}=\left(I^{\prime},\left(F_{\pm i^{\prime}}^{\prime}\right)_{i^{\prime}\in I^{\prime}},\left(F_{i^{\prime}}^{\prime}\right)_{i^{\prime}\in I^{\prime}},\left(y_{i^{\prime}}^{\prime}\right)_{i^{\prime}\in I^{\prime}}\right)$
of $\underline{\Xi}$ as a triple $\left(\iota,\left(\iota_{\pm i}\right)_{i\in I},\left(\iota_{i}\right)_{i\in I}\right)$,
where 
\begin{itemize}
\item $\iota:I\rightarrow I^{\prime}$ is a bijection;
\item For any $i\in I$, $\iota_{\pm i}:F_{\pm i}\rightarrow F_{\pm\iota\left(i\right)}^{\prime}$
is a $F$-isomorphism and $\iota_{i}=\iota_{\pm i}\otimes\text{Id}_{E}:F_{i}\rightarrow F_{\iota\left(i\right)}^{\prime}$
is the deduced $E$-isomorphism;
\item Moreover, this triple must satisfy 
\[
\iota_{i}\left(y_{i}\right)=y_{\iota\left(i\right)}^{\prime},\ \ \text{for all }i\in I.
\]
\end{itemize}
An element $\xi\in\underline{\Xi}$ is said to be regular if the identity
is its only automorphism. We denote by $\underline{\Xi}_{\text{reg}}$
the subset of $\underline{\Xi}$ containing regular elements. We define
$\Xi$ (resp. $\Xi_{\text{reg}}$ and $\Xi_{d}$ and $\Xi^{*}$) to
be sets of isomorphism classes of $\underline{\Xi}$ (resp. $\underline{\Xi}_{\text{reg}}$
and $\underline{\Xi}_{d}$ and $\underline{\Xi}^{*}$). We denote
$\Xi_{\text{reg},d}=\Xi_{d}\cap\Xi_{\text{reg}}$, $\Xi_{d}^{*}=\Xi_{d}\cap\Xi^{*}$,
$\Xi_{\text{reg},d}^{*}=\Xi_{\text{reg},d}\cap\Xi^{*}$. We always
identify an isomorphism class with an element representing it.

For $\xi=\left(I,\left(F_{\pm i}\right)_{i\in I},\left(F_{i}\right)_{i\in I},\left(y_{i}\right)_{i\in I}\right)\in\Xi_{\text{reg}}$,
we set 
\[
T_{\xi}=\underset{i\in I}{\prod}\text{Ker}\left(N_{F_{i}/F_{\pm i}}\right).
\]
It is a torus defined over $F$. We can equip to the space $\Xi_{\text{reg}}$
an analytic variety structure and a measure characterized as follows:
for any $\xi=\left(I,\left(F_{\pm i}\right)_{i\in I},\left(F_{i}\right)_{i\in I},\left(y_{i}\right)_{i\in I}\right)\in\Xi_{\text{reg}}$,
there exists an open neighborhood $\omega$ of $1$ in $T_{\xi}\left(F\right)$
such that the map 
\[
\left(t_{i}\right)_{i\in I}\mapsto\left(I,\left(F_{\pm i}\right)_{i\in I},\left(F_{i}\right)_{i\in I},\left(y_{i}t_{i}\right)_{i\in I}\right):\omega\rightarrow\Xi_{\text{reg}}
\]
induces an isomorphism preserving the measure of $\omega$ on an open
set of $\Xi_{\text{reg}}$. We denote by $\xi_{+}=\left(I_{+},\left(F_{\pm i}\right)_{i\in I_{+}},\left(F_{i}\right)_{i\in I_{+}},\left(y_{i}\right)_{i\in I_{+}}\right)$
and $\xi_{-}=\left(I_{-},\left(F_{\pm i}\right)_{i\in I_{-}},\left(F_{i}\right)_{i\in I_{-}},\left(y_{i}\right)_{i\in I_{-}}\right)$
in $\Xi$. We set $\xi_{+}\sqcup\xi_{-}\in\Xi$ to be $\left(I,\left(F_{\pm i}\right)_{i\in I},\left(F_{i}\right)_{i\in I},\left(y_{i}\right)_{i\in I}\right)$,
where $I=I_{+}\sqcup I_{-}$. If $\xi_{+}\sqcup\xi_{-}$ is regular,
the same to $\xi_{+}$ and $\xi_{-}$. Moreover, the map $\left(\xi_{+},\xi_{-}\right)\in\Xi^{2}\mapsto\xi_{+}\sqcup\xi_{-}\in\Xi$
locally preserves measures of neighborhoods of elements in $\Xi_{\text{reg}}$.

We consider $\xi=\left(I,\left(F_{\pm i}\right)_{i\in I},\left(F_{i}\right)_{i\in I},\left(y_{i}\right)_{i\in I}\right)\in\Xi_{\text{reg}}$.
We associate to it a finite abelian group 
\[
C\left(\xi\right)=\prod_{i\in I}F_{\pm i}^{\times}/N_{F_{i}/F_{\pm i}}\left(F_{i}^{\times}\right).
\]
The above group can be naturally identified with $\left\{ \pm1\right\} ^{I^{*}}$.
Let $C\left(\xi\right)^{1}$ be its subgroup containing elements whose
product of coordinates is equal to 1 and we set $C\left(\xi\right)^{-1}=C\left(\xi\right)\backslash C\left(\xi\right)^{1}$.
For $i\in I$, we denote by $\Gamma\left(y_{i}\right)$ the set containing
$\gamma_{i}\in F_{i}^{\times}$ such that $\gamma_{i}\tau_{i}\left(\gamma_{i}\right)^{-1}=y_{i}$.
We define 
\[
\Gamma\left(\xi\right)=\underset{i\in I}{\prod}\Gamma\left(y_{i}\right)/N_{F_{i}/F_{\pm i}}\left(F_{i}^{\times}\right).
\]
This is a homogeneous space under $C\left(\xi\right)$.

\subsection{Parametrization of conjugacy classes in unitary groups and twisted groups}\label{llc4}

Let $\left(V,h\right)$ be a skew-hermitian space of dimension $d$
and $G$ be its isometric group. We fix a nonzero trace-$0$ element
$\delta$ in $E$. Let $\xi=\left(I,\left(F_{\pm i}\right)_{i\in I},\left(F_{i}\right)_{i\in I},\left(y_{i}\right)_{i\in I}\right)\in\underline{\Xi}_{\text{reg},d}$
and $c=\left(c_{i}\right)_{i\in I}\in\underset{i\in I}{\prod}F_{\pm i}^{\times}$.
We define a skew-hermitian space $\left(V_{\xi,c},h_{\xi,c}\right)$
as follows:
\begin{itemize}
\item $V_{\xi,c}=\underset{i\in I}{\oplus}F_{i}$ ;
\item $h_{\xi,c}\left(\underset{i\in I}{\sum}x_{i},\underset{i\in I}{\sum}x_{i}^{\prime}\right)=\underset{i\in I}{\sum}\text{ Trace}_{F_{i}/F_{\pm i}}\left(\delta c_{i}x_{i}^{\prime}\tau_{i}\left(x_{i}\right)\right).$
\end{itemize}
The isomorphism class of $\left(V_{\xi,c},h_{\xi,c}\right)$ only
depends on images of $\xi$ and $c$ in $\Xi_{\text{reg},d}$ and
$C\left(\xi\right)$ respectively. Moreover, there exists a unique
$\epsilon\in\left\{ \pm1\right\} $ such that $\left(V_{\xi,c},h_{\xi,c}\right)$
is isomorphic to $\left(V,h\right)$ if and only if $c\in C\left(\xi\right)^{\epsilon}$.
For such $c$, we fix an isomorphism $\left(V,h\right)\simeq\left(V_{\xi,c},h_{\xi,c}\right)$
and we denote by $x\left(\xi,c\right)$ the element in $G_{\text{reg}}\left(F\right)$
which via this isomorphism acts on $V_{\xi,c}$ by multiplication
by $y_{i}\in F_{i}\subset V_{\xi,c}$. The conjugacy class of $x\left(\xi,c\right)$
only depends on images of $\xi$ and $c$ in $\Xi_{\text{reg},d}$
and $C\left(\xi\right)$ respectively. Therefore, conjugacy classes
in $G_{\text{reg}}\left(F\right)$ only depend on a unique $\xi\in\Xi_{\text{reg},d}$
and a unique $c\in C\left(\xi\right)$. Moreover, the stable conjugacy
class of $x\left(\xi,c\right)$ only depends on $\xi\in\Xi_{\text{reg},d}$.
We have the following commutative diagram 
$$
\begin{tikzcd}
G_{reg}(F)/conj \arrow[rr, "\phi_G^{st}/conj"] \arrow[rd, "p_G"] &               & G_{reg}(F)/stconj \arrow[ld, "p_G^{st}"'] \\
                                                                 & {\Xi_{reg,d}} &                                          
\end{tikzcd}
$$
where $p_{G}^{\text{st}}$ is a $F$-analytic isomorphism preserving
measures of $G_{\text{reg}}\left(F\right)/\text{stconj}$ on an open
set in $\Xi_{\text{reg},d}$ and $p_{G}$ is a covering of the same
open set. The fiber at a point $\xi$ of this covering can be identified
with a class $C\left(\xi\right)^{\epsilon}\subset C\left(\xi\right)$.
A conjugacy class (resp. stable conjugacy class) is anisotropic if
and only if its image via $p_{G}$ (resp. $p_{G}^{\text{st}}$) is
in $\Xi_{\text{reg},d}^{*}$.

Let $U$ be a vector space over $E$ of dimension $n$, and $M=\text{Res}_{E/F}\text{GL}\left(U\right)$
and $\tilde{M}$ is an algebraic variety defined over $F$ of nondegenerate
sesquilinear forms on $U$ which are linear on the second variable.
We have left and right actions of $M$ on $\tilde{M}$ as follows
\[
\left(m\tilde{m}m^{\prime}\right)\left(u,u^{\prime}\right)=\tilde{m}\left(m^{-1}u,m^{\prime}u^{\prime}\right),
\]
for $\tilde{m}\in\tilde{M}$ and $m,m^{\prime}\in M$. A pair $\left(M,\tilde{M}\right)$
is a twisted group. Let 
\[
\xi=\left(I,\left(F_{\pm i}\right)_{i\in I},\left(F_{i}\right)_{i\in I},\left(y_{i}\right)_{i\in I}\right)\in\underline{\Xi}_{\text{reg},n}
\]
and fix an $E$-linear isomorphism 
\[
M\simeq\underset{i\in I}{\bigoplus}F_{i}.
\]
For $\gamma=\left(\gamma_{i}\right)_{i\in I}\in\prod_{i\in I}\Gamma\left(y_{i}\right)$,
we set $\tilde{x}\left(\xi,\gamma\right)\in\tilde{M}_{\text{reg}}\left(F\right)$
to be 
\[
\tilde{x}\left(\xi,\gamma\right)\left(\underset{i\in I}{\sum}u_{i},\underset{i\in I}{\sum}u_{i}^{\prime}\right)=\underset{i\in I}{\sum}\text{ Trace}_{F_{i}/F_{\pm i}}\left(\delta\gamma_{i}\tau_{i}\left(u_{i}\right)u_{i}^{\prime}\right).
\]
The conjugacy class of $\tilde{x}\left(\xi,\gamma\right)$ only depends
on images of $\xi$ and $\gamma$ in $\Xi_{\text{reg},n}$ and $\Gamma\left(\xi\right)$
respectively. Hence, conjugacy classes in $\tilde{M}_{\text{reg}}\left(F\right)$
only depend on a unique $\xi\in\Xi_{\text{reg},n}$ and a unique $\gamma\in\Gamma\left(\xi\right)$.
Moreover, the stable conjugacy class of $\tilde{x}\left(\xi,\gamma\right)$
only depends on $\xi\in\Xi_{\text{reg},n}$. We have the following
commutative diagram
$$
\begin{tikzcd}
\tilde{M}_{reg}(F)/conj \arrow[rr, "\phi_{\tilde{M}}^{st}/conj"] \arrow[rd, "p_{\tilde{M}}"] &               & \tilde{M}_{reg}(F)/stconj \arrow[ld, "p_{\tilde{M}}^{st}"'] \\
                                                                 & {\Xi_{reg,n}} &                                          
\end{tikzcd}
$$
where $p_{\tilde{M}}^{\text{st}}$ is a $F$-analytic isomorphism
of $\tilde{M}_{\text{reg}}\left(F\right)/\text{stconj}$ on $\Xi_{\text{reg},n}$
and $p_{\tilde{M}}$ is a covering of the same open set. The fiber
at a point $\xi$ of this covering can be identified with $\Gamma\left(\xi\right)$.
Moreover, the map $p_{\tilde{M}}^{\text{st}}$ does not preserve measures
and its Jacobian is $\left|2\right|_{F}^{n}$. A conjugacy class (resp.
stable conjugacy class) is anisotropic if and only if its image via
$p_{\tilde{M}}$ (resp. $p_{\tilde{M}}^{\text{st}}$) is in $\Xi_{\text{reg},n}^{*}$.

\subsection{Calculation of transfer factors}\label{llc5}

Let $\xi=\left(I,\left(F_{\pm i}\right)_{i\in I},\left(F_{i}\right)_{i\in I},\left(y_{i}\right)_{i\in I}\right)\in\Xi_{\text{reg}}$
. For any $i\in I$, we denote by $\Phi_{i}$ the set containing nonzero
homomorphisms of $E$-algebras from $F_{i}$ to $\bar{F}$. We set
a polynomial $P_{\xi}\in E\left[T\right]$ defined by 
\[
P_{\xi}\left(T\right)=\underset{i\in I}{\prod}\underset{\phi\in\Phi_{i}}{\prod}\left(T-\phi\left(y_{i}\right)\right).
\]
Denote
\[
\Delta\left(\xi\right)=\left|P_{\xi}\left(1\right)\right|_{E}.
\]
Let $\left(V_{\xi},h_{\xi}\right)$ be a skew-hermitian space of dimension
$d_{\xi}$, with the associated unitary group $G_{\xi}$ and $t\in G_{\xi,\text{reg}}\left(F\right)$
whose stable conjugacy class is parametrized by $\xi$. We define
\[
\Delta\left(t\right)=\left|\det\left(1-t\right)_{\mid V_{\xi}}\right|_{E}.
\]
We have $\Delta\left(t\right)=\Delta\left(\xi\right)$. Moreover, since $t$ is a semisimple element in $G\left(F\right)$, we set $D^{d}\left(\xi\right)=D^{G}\left(t\right)$.

Let $\xi_{+}=\left(I_{+},\left(F_{\pm i}\right)_{i\in I_{+}},\left(F_{i}\right)_{i\in I_{+}},\left(y_{i}\right)_{i\in I_{+}}\right)$
and $\xi_{-}=\left(I_{-},\left(F_{\pm i}\right)_{i\in I_{-}},\left(F_{i}\right)_{i\in I_{-}},\left(y_{i}\right)_{i\in I_{-}}\right)$
in $\Xi_{\text{reg}}$ and $\mu_{+},\mu_{-}$ be continuous characters
of $E^{\times}$. We define $\xi=\xi_{+}\sqcup\xi_{-}$, $I=I_{+}\sqcup I_{-}$,
$d_{-}=d_{\xi_{-}}$, $d_{+}=d_{\xi_{+}}$ and $d=d_{\xi}$.

For $c=\left(c_{i}\right)_{i\in I}\in C\left(\xi\right)$ and $\nu\in F^{\times}$,
we set 
\[
\Delta_{\mu_{+},\mu_{-},\nu}\left(\xi_{+},\xi_{-},c\right)=\mu_{-}\left(P_{\xi_{-}}\left(-1\right)\right)\mu_{+}\left(P_{\xi_{+}}\left(-1\right)\right)\underset{i\in I_{-}}{\prod}\omega_{F_{i}/F_{\pm i}}\left(\nu C_{i}\right),
\]
where $\omega_{F_{i}/F_{\pm i}}$ is the quadratic character of $F_{i}/F_{\pm i}$
and 
\[
C_{i}=\begin{cases}
-\delta^{-d-1}c_{i}P_{\xi}^{\prime}\left(y_{i}\right)P_{\xi}\left(-1\right)^{-1}y_{i}^{1-d/2} & \text{if d is even,}\\
-\delta^{-d-1}c_{i}P_{\xi}^{\prime}\left(y_{i}\right)P_{\xi}\left(-1\right)^{-1}y_{i}^{(1-d)/2}\left(1+y_{i}\right) & \text{if d is odd.}
\end{cases}
\]
We will see that these functions $\Delta_{\mu_{+},\mu_{-},\nu}$ correspond
to transfer factors of unitary groups.

For $\gamma=\left(\gamma_{i}\right)_{i\in I}\in\Gamma\left(\xi\right)$,
we set 
\[
\Delta_{\mu_{+},\mu_{-}}\left(\xi_{+},\xi_{-},\gamma\right)=\mu_{-}\left(P_{\xi_{-}}\left(-1\right)\right)\mu_{+}\left(P_{\xi_{+}}\left(-1\right)\right)\underset{i\in I_{-}}{\prod}\omega_{F_{i}/F_{\pm i}}\left(C_{i}\right),
\]
where 
\[
C_{i}=\begin{cases}
-\delta^{-d-1}\gamma_{i}^{-1}P_{\xi}^{\prime}\left(y_{i}\right)P_{\xi}\left(-1\right)^{-1}y_{i}^{1-d/2}\left(1+y_{i}\right) & \text{if d is even,}\\
-\delta^{-d-1}\gamma_{i}^{-1}P_{\xi}^{\prime}\left(y_{i}\right)P_{\xi}\left(-1\right)^{-1}y_{i}^{(3-d)/2} & \text{if d is odd.}
\end{cases}
\]
Similarly, these functions $\Delta_{\mu_{+},\mu_{-}}$ correspond
to transfer factors for twisted endoscopy.

\subsection{Endoscopy for unitary groups}\label{llc6}

Let $\left(V,h\right)$, $\left(V_{+},h_{+}\right)$ and $\left(V_{-},h_{-}\right)$
be skew-hermitian spaces over $F$ of dimensions $d$, $d_{+}$ and
$d_{-}$ with corresponding unitary groups $G$, $G_{+}$ and $G_{-}$
respectively. We assume 
\begin{itemize}
\item $d=d_{+}+d_{-}$ ;
\item $G_{+}$ and $G_{-}$ are quasi-split.
\end{itemize}
We say $G_{+}\times G_{-}$ is an endoscopic group of $G$. There are
some choices that need to be fixed in order to define an endoscopic
datum as well as transfer factors. We fix continuous characters $\mu_{+}$ and $\mu_{-}$ of $E^{\times}$ whose restrictions to $F^{\times}$ coincide with $\omega_{E/F}^{d_{-}}$ and $\omega_{E/F}^{d_{+}}$. We still need to fix some additional data so that transfer factors
are well defined. They include a quasi-split inner form $\underline{G}$ of $G$, an inner torsor $\psi_{G}:G\rightarrow\underline{G}$, an $1$-cocycle $u:\text{Gal}\left(\bar{F}/F\right)\rightarrow\underline{G}$ and a regular nilpotent orbit of $\underline{\mathfrak{g}}\left(F\right)$. We fix an element $\nu_{0}\in F^{\times}$.

We now define a matching between regular semisimple stable conjugacy
classes of $G\left(F\right)$ and $G$-regular semisimple stable conjugacy
classes of $G_{+}\left(F\right)\times G_{-}\left(F\right)$. Via maps
$p_{G}^{\text{st}},\,p_{G_{+}}^{\text{st}}\times p_{G_{-}}^{\text{st}}$,
a semisimple stable conjugacy class $\left(\xi_{+},\xi_{-}\right)\in\Xi_{\text{reg},d_{+}}\times\Xi_{\text{reg},d_{-}}$
is $G$-regular if and only if $\xi_{+}\sqcup\xi_{-}\in\Xi_{\text{reg},d}$
and a stable conjugacy class $\xi\in\text{Im}\left(p_{G}^{\text{st}}\right)$
corresponds to it if and only if $\xi=\xi_{+}\sqcup\xi_{-}$. Let
$c\in C\left(\xi\right)$ parametrize the conjugacy class of $y$. Then 
\[
\Delta_{G_{+}\times G_{-},G}\left(\left(y_{+},y_{-}\right),y\right)=\Delta_{\mu_{+},\mu_{-},\nu_{0}}\left(\xi_{+},\xi_{-},c\right).
\]

\subsection{Base change of unitary groups and twisted endoscopy}\label{llc7}

Let $U$ be a $d$-dimensional vector space over $E$ and $\left(V_{+},h_{+}\right)$,
$\left(V_{-},h_{-}\right)$ are skew-hermitian spaces of dimensions $d_{+}$ and $d_{-}$. Let $\left(M,\tilde{M}\right)$ be the twisted group corresponding to $U$ and $G_{+}$, $G_{-}$ be unitary groups with respect to $\left(V_{+},h_{+}\right)$ and $\left(V_{-},h_{-}\right)$. We assume $G_{+}$ and $G_{-}$ are quasi-split and $d=d_{+}+d_{-}$. Then $G_{+}\times G_{-}$ is a twisted endoscopic group of $\left(M,\tilde{M}\right)$. We fix continuous characters $\mu_{+}$ and $\mu_{-}$ of $E^{\times}$, whose restrictions to $F^{\times}$ coincide with $\omega_{E/F}^{d_{-}}$
and $\omega_{E/F}^{d_{+}+1}$.

We define a twisted endoscopic relation in this setting. It depends
on choices of a base point of $\tilde{M}\left(F\right)$. We fix a
base $\left(u_{j}\right)_{j=\overline{1,d}}$ of $U$ and a base-point
element $\theta_{d}\in\tilde{M}\left(F\right)$ satisfying 
\[
\theta_{d}\left(u_{j},u_{k}\right)=\left(-1\right)^{k}\delta^{d+1}\delta_{j,d+1-k}.
\]
We now define an endoscopic correspondence using spaces of parameters:
a stable conjugacy class of $G_{+}\left(F\right)\times G_{-}\left(F\right)$
parametrized by $\left(\xi_{+},\xi_{-}\right)\in\Xi_{\text{reg},d_{+}}\times\Xi_{\text{reg},d_{-}}$
is $\tilde{M}$-regular if and only if $\xi_{+}\sqcup\xi_{-}\in\Xi_{\text{reg},d}$
and it corresponds to a stable conjugacy class $\xi\in\Xi_{\text{reg},d}$
of $\tilde{M}\left(F\right)$ if and only if $\xi=\xi_{+}\sqcup\xi_{-}$. 

Let $\tilde{x}\in\tilde{M}_{\text{reg}}\left(F\right)$ and $y=\left(y_{+},y_{-}\right)\in G_{+,\text{reg}}\left(F\right)\times G_{-,\text{reg}}\left(F\right)$, whose stable conjugacy classes match. The conjugacy class of $\tilde{x}$ is parametrized by $\xi\in\Xi_{\text{reg},d}$ and
$\gamma\in\Gamma\left(\xi\right)$, while the stable conjugacy class
of $y$ is parametrized by $\left(\xi_{+},\xi_{-}\right)\in\Xi_{\text{reg},d_{+}}\times\Xi_{\text{reg},d_{-}}$.
Then
\[
\Delta\left(\tilde{x},y\right)=\Delta_{\mu_{+},\mu_{-},\nu_{0}}\left(\xi_{+},\xi_{-},\gamma\right).
\]

\subsection{L-parameters and conjugate-dual representations of the Weil-Deligne
group}\label{llc8}

We denote the Weil-Deligne group of $E$ by $WD_{E}=W_{E}\times\text{SL}_{2}\left(\mathbb{C}\right)$,
where $W_{E}$ is the Weil group of $E$. Let $n\geq1$. Then $n$-dimensional Langlands parameters of the following form 
\[
\phi:WD_{E}\rightarrow\text{GL}_{n}\left(\mathbb{C}\right)
\]
satisfying the following conditions
\begin{itemize}
\item $\phi$ is semisimple.
\item The restriction of $\phi$ to $\text{SL}_{2}\left(\mathbb{C}\right)$
is algebraic.
\end{itemize}
Moreover, $\phi$ is tempered if the image of $W_{E}$ via $\phi$
is relatively compact. Two Langlands parameters $\phi$ and $\phi^{\prime}$
are said to be conjugate, which we denote by $\phi\simeq\phi^{\prime}$, if they are of the same dimension $n$ and they are conjugated by an element in $\text{GL}_{n}\left(\mathbb{C}\right)$. We denote by
$\Phi_{\text{\text{temp}}}\left(\text{GL}_{n}\right)$ the set of tempered Langlands parameters of dimension $n$ up to conjugation. 

The local Langlands correspondence of general linear groups, which is obtained by Harris-Taylor and Henniart (see \cite{Hen00,HT02}), associates an Langlands parameter $\phi\in\Phi_{\text{\text{temp}}}\left(\text{GL}_{n}\right)$ with an irreducible tempered representation $\pi\left(\phi\right)$ of $\text{GL}_{n}\left(L\right)$. We denote by $\Phi_{\text{\text{temp, irr}}}\left(\text{GL}_{n}\right)$
the subset containing irreducible $\phi\in\Phi_{\text{\text{temp}}}\left(\text{GL}_{n}\right)$. We fix $t\in W_{F}\backslash W_{E}$. We extend the action of $t$ on $W_{E}$ via conjugation to $WD_{E}$ by letting it act on $\text{SL}_{2}\left(\mathbb{C}\right)$ trivially. For any $\phi\in\Phi_{\text{\text{temp}}}\left(\text{GL}_{n}\right)$, we define a new parameter $\phi^{\theta}\in\Phi_{\text{\text{temp}}}\left(\text{GL}_{n}\right)$ by taking $\phi^{\theta}\left(\tau\right)=\,^{t}\phi\left(t\tau t^{-1}\right)^{-1}$, for all $\tau\in WD_{E}$. Then $\phi^{\theta}$ does not depend (up to conjugation) choices of $t$. A parameter $\phi\in\Phi_{\text{\text{temp}}}\left(\text{GL}_{n}\right)$ is said to be conjugate-dual if $\phi\simeq\phi^{\theta}$. This is equivalent to the existence of a non-degenerate bilinear form $B:\mathbb{C}^{n}\times\mathbb{C}^{n}\rightarrow\mathbb{C}$
satisfying 
\[
B\left(\phi\left(\tau\right)w,\phi\left(t\tau t^{-1}\right)w^{\prime}\right)=B\left(w,w^{\prime}\right),\text{ for all }w,w^{\prime}\in\mathbb{C}^{n}\text{ and }\tau\in WD_{E}.
\]
We denote by $\Phi_{\text{\text{temp}}}^{\theta}\left(\text{GL}_{n}\right)$
the set of conjugate-dual tempered parameters. Let $\epsilon\in\left\{ \pm\right\} $
be a sign. An element $\phi\in\Phi_{\text{\text{temp}}}^{\theta}\left(\text{GL}_{n}\right)$
is conjugate-dual of sign $\epsilon$ if there exists a non-degenerate
bilinear form $B:\mathbb{C}^{n}\times\mathbb{C}^{n}\rightarrow\mathbb{C}$
satisfying the above condition and 
\[
B\left(w,\phi\left(t^{2}\right)w^{\prime}\right)=\epsilon B\left(w^{\prime},w\right),\text{ for all }w,w^{\prime}\in\mathbb{C}^{n}.
\]
This definition does not depend on choices of $t$. A non-degenerate
bilinear form $B$ satisfying the above conditions is called $\epsilon$-conjugate-dual.
We fix such a form $B$ and denote by $\text{Aut}\left(\phi,B\right)$
the group containing $g\in\text{GL}_{n}\left(\mathbb{C}\right)$ preserving
$B$ and commuting with the image of $\phi$. Up to inner automorphisms,
this group does not depend on choices of $t$ nor $B$ and it is a
complex reductive group not connected in general). We set $\mathcal{S}_{\phi}=\text{Aut}\left(\phi,B\right)/\text{Aut}\left(\phi,B\right)^{0}$,
where $\text{Aut}\left(\phi,B\right)^{0}$ is the identity component.
This group is abelian and hence well-defined up to isomorphism. The
notation $\mathcal{S}_{\phi}$ is somewhat imprecise: it does not
show whether we consider $\phi$ as a conjugate-dual parameter of
sign + or $-$ (and a certain parameter can be considered in both
signs). However, in what follows, the context should be clear without
any possible confusions. We denote by $\Phi_{\text{\text{temp}}}^{+}\left(\text{GL}_{n}\right)$
(resp. $\Phi_{\text{\text{temp}}}^{-}\left(\text{GL}_{n}\right)$)
the subset of $\Phi_{\text{\text{temp}}}^{\theta}\left(\text{GL}_{n}\right)$
containing conjugate-dual parameters of sign $+$ (resp. of sign $-$).
We set $\Phi_{\text{\text{temp, irr}}}^{\epsilon}\left(\text{GL}_{n}\right)=\Phi_{\text{\text{temp, irr}}}\left(\text{GL}_{n}\right)\cap\Phi_{\text{\text{temp}}}^{\epsilon}\left(\text{GL}_{n}\right)$,
where $\epsilon\in\left\{ \pm,\theta\right\} $. Then $\Phi_{\text{\text{temp, irr}}}^{\theta}\left(\text{GL}_{n}\right)=\Phi_{\text{\text{temp, irr}}}^{+}\left(\text{GL}_{n}\right)\sqcup\Phi_{\text{\text{temp, irr}}}^{-}\left(\text{GL}_{n}\right)$.
Let 
\[
\Phi_{\text{\text{temp}}}^{\epsilon}=\underset{d\geq0}{\bigsqcup}\Phi_{\text{\text{temp}}}^{\epsilon}\left(\text{GL}_{d}\right),\text{ for }\epsilon\in\left\{ \pm,\theta,\emptyset\right\} .
\]
We have the following properties
\begin{itemize}
\item If $\phi_{1}\in\Phi_{\text{\text{temp}}}^{\epsilon_{1}}$, $\phi_{2}\in\Phi_{\text{\text{temp}}}^{\epsilon_{2}}$,
$\epsilon_{1},\epsilon_{2}\in\left\{ \pm1\right\} $, then $\phi_{1}\otimes\phi_{2}\in\Phi_{\text{\text{temp}}}^{\epsilon_{1}\epsilon_{2}}$
;
\item If $\phi\in\Phi_{\text{\text{temp}}}^{\epsilon}\left(\text{GL}_{n}\right)$,
$\epsilon\in\left\{ \pm1\right\} $, then $\det\phi\in\Phi_{\text{\text{temp}}}^{\epsilon^{n}}$.
\end{itemize}
Let $\phi\in\Phi_{\text{temp}}^{\theta}$. Then there is a unique
decomposition up to choices of index 
\[
\phi\simeq\underset{j\in J}{\bigoplus}l_{j}\phi_{j}\oplus\underset{i\in I}{\bigoplus}l_{i}\left(\phi_{i}\oplus\phi_{i}^{\theta}\right),
\]
where
\begin{itemize}
\item $I$ and $J$ are disjoint finite sets.
\item $l_{i},l_{j}$ are nonzero natural numbers, for all $i\in I$ and
$j\in J$.
\item For all $j\in J$, we have $\phi_{j}\in\Phi_{\text{\text{temp, irr}}}^{\theta}\left(\text{GL}_{n}\right)$.
\item For all $i\in I$, we have $\phi_{i}\in\Phi_{\text{\text{temp, irr}}}\left(\text{GL}_{n}\right)\backslash\Phi_{\text{\text{temp, irr}}}^{\theta}\left(\text{GL}_{n}\right)$.
\item $\phi_{i},\,\phi_{j}$ ($i\in I,\,j\in J$) are two by two distinct.
\end{itemize}
We denote by $J^{+}$ (resp. $J^{-}$) the subset containing $j\in J$
such that $\phi_{j}\in\Phi_{\text{\text{temp, irr}}}^{+}$ (resp.
$\phi_{j}\in\Phi_{\text{\text{temp, irr}}}^{-}$). Then $\phi\in\Phi_{\text{\text{temp}}}^{+}$
(resp. $\phi\in\Phi_{\text{\text{temp}}}^{-}$) if and only if $l_{j}$
is even for any $j\in J^{-}$ (resp. for all $j\in J^{+}$). Let $\epsilon\in\left\{ \pm1\right\} $
and suppose that $\phi\in\Phi_{\text{\text{temp, irr}}}^{\epsilon}$.
Let us fix a non-degenerate bilinear form $B$ which is $\epsilon$-conjugate-dual.
We have the following identification (upto inner automorphisms) 
\[
\text{Aut}\left(\phi,B\right)=\underset{j\in J^{\epsilon}}{\prod}\text{O}\left(l_{j},\mathbb{C}\right)\times\underset{j\in J^{-\epsilon}}{\prod}\text{Sp}\left(l_{j},\mathbb{C}\right)\times\underset{i\in I}{\prod}\text{GL}\left(l_{i},\mathbb{C}\right)
\]
from which we induce an identification $\mathcal{S}_{\phi}=\left\{ \pm1\right\} ^{J^{\epsilon}}$.

\subsection{The local Langlands correspondence for unitary groups}\label{llc9}

Let $\left(V,h\right)$ be a skew-hermitian space
over $E$ of dimension $n$ and $G$ be its unitary group. We set $\Phi_{\text{temp}}\left(G\right)=\Phi_{\text{temp}}^{\left(-1\right)^{d+1}}\left(\text{GL}_{d}\right)$
and $\mathcal{E}^{G}\left(\phi\right)$ the set of characters $\epsilon$ of $\mathcal{S}_{\phi}$ such that $\epsilon\left(z_{\phi}\right)=\mu\left(V,h\right)$. The
local Langlands correspondence, which has been obtained in \cite{Mok15,KMSW14,MR18,CZ21}, gives us 
\begin{description}
\item [{(LLC)}] For any unitary group $G=\text{U}\left(V,h\right)$, there
exists a decomposition into disjoint union 
\[
\text{Temp}\left(G\right)=\underset{\phi\in\Phi_{\text{temp}}\left(G\right)}{\bigsqcup}\Pi^{G}\left(\phi\right)
\]
and bijections 
\[
\begin{array}{ccc}
\mathcal{E}^{G}\left(\phi\right) & \simeq & \Pi^{G}\left(\phi\right)\\
\epsilon & \mapsto & \sigma\left(\phi,\epsilon\right)
\end{array}.
\]
\end{description}
The finite sets $\Pi^{G}\left(\phi\right)$ are called L-packets.
The local Langlands correspondence provides a parametrization of $\text{Temp}\left(G\right)$.
This parametrization are uniquely determined by certain conditions. We introduce three
of them, which are of endoscopic nature and will be important for our use. For $G=\text{U}\left(V,h\right)$,
$\phi\in\Phi_{\text{temp}}\left(G\right)$ and $s\in\mathcal{S}_{\phi}$,
we set
\[
\Theta_{\phi,s}=\underset{\epsilon\in\mathcal{E}^{G}\left(\phi\right)}{\sum}\epsilon\left(s\right)\Theta_{\pi\left(\phi,\epsilon\right)},
\]
where $\Theta_{\pi\left(\phi,\epsilon\right)}$ is the character
of the representation $\pi\left(\phi,s\right)$. The first condition
is 
\begin{description}\label{stab}
\item [{(Stab)}] If $\left(V,h\right)$ is quasi-split, then $\Theta_{\phi,1}$
is stable.
\end{description}
We now consider the remaining two conditions, one involves in classic
endoscopy and the other involves in twisted endoscopy. Let $\left(V_{+},h_{+}\right)$
and $\left(V_{-},h_{-}\right)$ be two skew-hermitian spaces of dimensions
$d_{+}$ and $d_{-}$ with corresponding unitary groups $G_{+}$ and
$G_{-}$. We assume that both $G_{+}$ and $G_{-}$ are quasi-split and $d=d_{+}+d_{-}$.
Let $\phi_{+}\in\Phi_{\text{temp}}\left(G_{+}\right)$ and $\phi_{-}\in\Phi_{\text{temp}}\left(G_{-}\right)$.
Let $\mu_{+},\,\mu_{-}$ be two continuous characters of $W_{E}$,
which can be identified with characters of $E^{\times}$ by local
class field theory.

Assume the restrictions of $\mu_{+}$ and $\mu_{-}$ to $F^{\times}$
coincide with $\omega_{E/F}^{d_{-}}$ and $\omega_{E/F}^{d_{+}}$.
We set $\phi=\mu_{+}\phi_{+}\oplus\mu_{-}\phi_{-}$ to be an element
in $\Phi_{\text{temp}}\left(G\right)$. Let $s$ be an element which
acts by identity on $\mu_{+}\phi_{+}$ and $-1$ on $\mu_{-}\phi_{-}$.
There exists non-degenerate $\left(-1\right)^{d+1}$-conjugate-dual
bilinear forms $B_{+}$ and $B_{-}$ on $\mu_{+}\phi_{+}$ and $\mu_{-}\phi_{-}$.
We set $B=B_{+}\oplus B_{-}$, which is a conjugate-dual bilinear
form on $\phi$ of sign $\left(-1\right)^{d+1}$. We have $s\in\text{Aut}\left(\phi,B\right)$,
hence determines an element $s\in\mathcal{S}_{\phi}$. Moreover, as an
element in $\mathcal{S}_{\phi}$, we can see that $s$ does not depend
on choices of $B_{+}$ and $B_{-}$. The group $G_{+}\times G_{-}$
is an endoscopic group of $G$. Moreover, the pair $\left(\mu_{+},\mu_{-}\right)$
fixes the endoscopic datum. We normalize the transfer factors as in Section \ref{llc5}. We state the second condition.
\begin{description}\label{et}
\item [{(ET)}] There exists a number $\gamma_{\mu_{+},\mu_{-}}^{G}\left(\phi_{+},\phi_{-}\right)$
of modulus $1$ such that $\gamma_{\mu_{+},\mu_{-}}^{G}\left(\phi_{+},\phi_{-}\right)\Theta_{\phi,s}$
is a transfer of $\Theta_{\phi_{+},0}\times\Theta_{\phi_{-},0}$.
\end{description}
Now suppose the restrictions of $\mu_{+}$ and $\mu_{-}$ to $F^{\times}$
coincide with $\omega_{E/F}^{d_{-}}$ and $\omega_{E/F}^{d_{+}+1}$.
Let $\phi=\mu_{+}\phi_{+}\oplus\mu_{-}\phi_{-}\in\Phi_{\text{\text{temp}}}^{\theta}\left(\text{GL}_{d}\right)$.
Likewise, we now define the twisted endoscopic property. Let $U$ be a $d$-dimensional vector space with the associated twisted group $\left(M,\tilde{M}\right)$. We set $\pi=\pi\left(\phi\right)$.
Then $\pi$ is a conjugate-dual representation of $M\left(F\right)$.
Hence, it can be extended to a representation $\tilde{\pi}$ of $\tilde{M}\left(F\right)$.
The group $G_{+}\times G_{-}$ is a twised endoscopic group of $\tilde{\pi}$.
Moreover, the pair $\left(\mu_{+},\mu_{-}\right)$ fixes the twisted
endoscopic datum. Similar to the setting in endoscopic property, we normalize the transfer factors as in Section \ref{llc5}. We state the third condition.
\begin{description}\label{tet}
\item [{(TET)}] There exists a number $c_{\mu_{+},\mu_{-}}\left(\phi_{+},\phi_{-}\right)$
of modulus $1$ such that $c_{\mu_{+},\mu_{-}}\left(\phi_{+},\phi_{-}\right)\Theta_{\tilde{\pi}}$
is a transfer of $\Theta_{\phi_{+},0}\times\Theta_{\phi_{-},0}$.
\end{description}
We admit the existence of \textbf{(LLC)} satisfying conditions \textbf{(Stab),
(ET) and (TET)}.
\subsection{Asai factors} We recall the notion of Asai representations in \cite{GGP12a,GGP23}. Let $E/F$ be a quadratic extension of nonarchimedean fields of characteristic $0$ and $\tau$ be the nontrivial element of $\text{Gal}(E/F)$. Let $M$ be a representation of $WD_E$. Since the representation $M\otimes M^\tau$ is $\tau$-invariant, we have the following decomposition
$$
\text{Ind}^{WD_F}_{WD_E}(M\otimes M^\tau)=\text{As}^+_{E/F}(M)\oplus \text{As}^-_{E/F}(M)
$$
of $WD_F$-modules, where $\text{As}^\pm_{E/F}$ is characterized by the action of an element $s\in W_F\setminus W_E$ (cf. \cite[pg. 26-27]{GGP12a}). Without abuse of notations, we drop the sign $+$ and thus set $\text{As}_{E/F}(M)=\text{As}^+_{E/F}(M)$.

\section{Weil representations and the twisted Gan-Gross-Prasad conjecture}\label{sec3}
\subsection{Weil representation of unitary groups}\label{sec3.1}

Let $F$ be a nonarchimedean local field of characteristic $0$ and
$\left(W,\left\langle ,\right\rangle \right)$ be a nondegenerate symplectic vector space of dimension $2n$ over $F$. We denote by $\text{Sp}\left(W\right)$
the isometry group of $W$. Let $H\left(W\right)=W\oplus F$ be the Heisenberg group with operation 
$$
\left(w_{1},t_{1}\right)\left(w_{2},t_{2}\right)=\left(w_{1}+w_{2},t_{1}+t_{2}+\frac{1}{2}\left\langle w_{1},w_{2}\right\rangle \right).
$$
The center of $H\left(W\right)$ is $Z=\left\{ \left(0,t\right)\mid\ t\in F\right\} $.
We define an action of $\text{Sp}\left(W\right)$ on $H\left(W\right)$ as follows
$$
g\cdot\left(w,t\right)=\left(gw,t\right),\ \text{for }g\in\text{Sp}\left(W\right)\text{ and }\left(w,t\right)\in H\left(W\right).
$$
Let $\psi$ be a nontrivial additive character of $F$. Stone-von Neumann theorem gives us a unique smooth irreducible representation $\left(\rho_{\psi},S\right)$ of $H\left(W\right)$ whose central character is $\psi$. 

Since $g\in\text{Sp}\left(W\right)$ acts trivially on $Z$, the representation $\left(\rho_{\psi}^{g},S\right)$
given by $\rho_{\psi}^{g}\left(h\right)=\rho_{\psi}\left(h^{g}\right)$
has the same central character $\psi$, and thus is isomorphic to $\left(\rho_{\psi},S\right)$.
Hence, for each $g\in\text{Sp}\left(W\right)$, there exists an automorphism
$A\left(g\right):S\rightarrow S$ such that 
$$
A\left(g^{-1}\right)\rho_{\psi}\left(h\right)A\left(g\right)=\rho_{\psi}\left(h^{g}\right).
$$
The above automorphism $A\left(g\right)$ is unique up to a scalar in $\mathbb{C}^{\times}$. Thus, we can define a central extension $\overline{\text{Sp}}_\psi\left(W\right)$ 
$$
1\longrightarrow\mathbb{C}^{\times}\longrightarrow\overline{\text{Sp}}_\psi\left(W\right)\longrightarrow\text{Sp}\left(W\right)\longrightarrow1
$$
such that $A$ can be lifted to a representation $\omega_{\psi}$
of $\overline{\text{Sp}}_\psi\left(W\right)$ via $\omega_{\psi}\left(g,A\left(g\right)\right)=A\left(g\right)$. By \cite[Section 43]{Wei64} and \cite{Moo68}, there exists a unique subgroup $\text{Mp}(W)$ of $\overline{\text{Sp}}\left(W\right)$, which does not depend on choices of $\psi$, such that $\text{Mp}(W)$ is a double cover of $\text{Sp}(W)$. This subgroup is called the metaplectic cover of $\text{Sp}(W)$ and the restriction of the representation $\omega_\psi$ to $\text{Mp}(W)$ is called the Weil representation of $\text{Mp}(W)$. We will revisit their explicit forms in Section \ref{Schrodinger}.

Let $E$ be a quadratic extension of $F$ and $\left(V,\left\langle \cdot,\cdot\right\rangle _{V}\right)$ be a skew-hermitian space over $E$ of dimension $n$. Then $\left(\text{Res}_{E/F}V,\text{Tr}_{E/F}\left\langle \cdot,\cdot\right\rangle _{V}\right)$
is a $2n$-dimensional symplectic space over $F$. We denote by $\text{Sp}\left(\text{Res}_{E/F}V\right)$
the symplectic group associated to the above symplectic space. As
mentioned above, one can define the Weil representation $\omega_{\psi}$
of $\text{Mp}\left(\text{Res}_{E/F}V\right)$. Let $\mu$ be a conjugate-symplectic character
of $E^{\times}$, i.e. $\mu\mid_{F^{\times}}=\omega_{E/F}$, where
$\omega_{E/F}$ is the quadratic character factoring through $F^{\times}/N_{E/F}E^{\times}\overset{\sim}{\longrightarrow}\left\{ \pm1\right\} $.
By \cite[Theorem 3.1]{Kud94} and \cite[Section 1]{HKS96}, the character $\mu$
determines an inclusion $\mu:g\in U\left(V\right)\mapsto (g,\mu(g))\in \text{Mp}\left(\text{Res}_{E/F}V\right)$
splitting $\text{Mp}\left(\text{Res}_{E/F}V\right)\rightarrow\text{Sp}\left(\text{Res}_{E/F}V\right)$.
This gives us the Weil representation $\omega_{V,\psi,\mu}$ of $U\left(V\right)$.

Let $K$ be a field extension of $F$ not containing $E$ and $L=K\otimes_{F}E$.
Let $\left(M,\left\langle ,\right\rangle \right)$ be a skew-hermitian
space relative to $L/K$. Then $\left(\text{Res}_{L/E}M,\text{Tr}_{L/E}\left\langle \right\rangle \right)$ is a skew-hermitian space over $E$ and one has an inclusion $i:\text{Res}_{K/F}\text{U}\left(M\right)\hookrightarrow\text{U}\left(\text{Res}_{L/E}M\right)$ defined over $F$, where $\text{Res}_{K/F}\text{U}\left(M\right)$ is the usual Weil restriction. We recall a functorial property of Weil representations, which has been proved in \cite[Proposition 3.1]{Le25}.
\begin{proposition}\label{3.1}
Denote $\psi_{K}=\psi\circ\text{Tr}_{K/F}$ and $\mu_{L}=\mu\circ\text{Nm}_{L/E}$ so that $\mu_L$ is conjugate-symplectic relative to $L/K$.
Let $\omega_{M,\psi_{K},\mu_{L}}$ be the Weil representation of $\text{U}\left(M\right)$
and $\omega_{\text{Res}_{L/E}M,\psi,\mu}$ be the Weil representation
of $\text{U}\left(\text{Res}_{L/E}M\right)$. Then 
$$
\omega_{M,\psi_{K},\mu_{L}}\cong i^{*}\omega_{\text{Res}_{L/E}M,\psi,\mu},
$$
where $i^{*}\omega_{\text{Res}_{L/E}M,\psi,\mu}$ is the pullpack
representation of $\omega_{\text{Res}_{L/E}M,\psi,\mu}$ via the inclusion
$i:\text{U}\left(M\right)\hookrightarrow\text{U}\left(\text{Res}_{L/E}M\right)$.
\end{proposition}
\subsection{Weil index and Maslov index}\label{secweilindex} 
We refer to \cite{Li08} and \cite{GKT25} as the main reference in this subsection. Let $X$ be an $n$-dimensional $F$-vector space and $b:X\times X \rightarrow F$ be a nondegenerate symmetric bilinear form. We denote by $q_b$ the corresponding quadratic form on $X$. We have the natural isomorphism $\beta_b:x\mapsto b(-,x)$ mapping $X$ onto $X^*$, where $X^*$ is the algebraic dual of $X$. Let $\psi$ be a nontrivial additive character of $F$. Fix a Haar measure $dx$ on $X$, we can choose the measure $dx^*$ on $X^*$ to be dual to $X$. We set $|\beta_b|=\frac{d\beta(x)}{dx}$.

Given the quadratic form $q_b$ and additive character $\psi$, we define the corresponding character of second degree $f_b=\psi \circ \frac{1}{2}q_b$. The normalised convolution $\phi \mapsto |\beta_b|^{1/2}\phi*f_b$ gives us an isometry on $C^\infty_c(X)$ (cf. \cite[Proposition 3.2]{GKT25}). Likewise, we can also consider the character of second degree on $X^*$ defined by $f_b^*(x^*)=f_b(\beta_b^{-1}(x^*))^{-1}$. We have another isometry $\phi \mapsto \hat{\phi} \cdot f_b^*$ mapping $C_c^\infty(X)$ onto $C_c^\infty(X^*)$. We recall \cite[Theorem 3.3]{GKT25}, which gives us the definition of Weil index.
\begin{theorem}
    There exists a constant $\gamma(f_b)\in \mathbb{C}^1$ depending only on $f_b$ such that
    $$
    |\beta_b|^{1/2}\widehat{\phi*f_b}=\gamma(f_b)\phi \mapsto \hat{\phi} \cdot f_b^*,
    $$
    for all $\phi\in C^\infty_c(X)$.
\end{theorem}

We now recall the definition of Maslov index. Let $\ell_1,\ldots,\ell_k$ be $k$ lagrangians of $W$, where $k\geq 3$. We can define a quadratic form $\tau(\ell_1,\ldots,\ell_k)$, which is called a Maslov index. As an equivalent class in the Witt group $W(F)$, it satisfies the following properties
\begin{enumerate}
    \item For any $g\in \text{Sp}(W)$,
    $$
    \tau(g\ell_1,\ldots,g\ell_k)=\tau(\ell_1,\ldots,\ell_k).
    $$
    \item Let $W_1$ and $W_2$ be two symplectic spaces. If $\ell_1,\ldots,\ell_{k}$ are lagrangians of $W_1$ and $\ell_1^\prime,\ldots,\ell_{k}^\prime$ are lagrangians of $W_2$, we have
    $$
    \tau(\ell_1\oplus \ell_1^\prime,\ldots,\ell_k\oplus \ell_k^\prime)=\tau(\ell_1,\ldots,\ell_k)+\tau(\ell_1^\prime,\ldots,\ell_k^\prime).
    $$
    \item $\tau(\ell_1,\ldots,\ell_k)=\tau(\ell_2,\ldots,\ell_k,\ell_1)$ and $\tau(\ell_1,\ldots,\ell_k)=-\tau(\ell_k,\ldots,\ell_1)$.
    \item For any $t=\overline{3,k}$,
    $$
    \tau(\ell_1,\ldots,\ell_k)=\tau(\ell_1,\ldots,\ell_t)+\tau(\ell_{1},\ell_{t},\ldots,\ell_k).
    $$
\end{enumerate}
When $k=3$, Maslov indexes have been defined by Kashiwara (see in \cite{LV80}). When $k>3$, they are defined inductively by the above properties.
\subsection{Schrodinger models of the Weil representation}\label{Schrodinger}
We give a brief introduction about Schrodinger models of the Weil representation. Let $(W,\langle,\rangle))$ is a nondegenerate symplectic space of dimension $2n$ over $F$. We fix a lagrangian $\ell$ of $W$. We denote by $\psi_\ell$ the trivial extension of $\psi$ to $H(\ell)=\ell \times F$. We set $(\rho_\ell,S_\ell)$ to be the smooth induction $\text{Ind}_{H(\ell)}^{H(W)}$. Such representations $\rho_\ell$ are called Schrodinger representations. We can reformulate this construction by the language of densities. Let $\mathcal{H}_\ell$ be the space of measurable functions $f:W \rightarrow \Omega_{1/2}(W/\ell)$ such that
$$
f(v+a)=\psi\left(\frac{\langle v,a \rangle}{2}\right)f(v) \ \ \text{ and }\ \ \int_{W/\ell}ff < \infty.
$$
By Stone-von Neumann theorem, Schrodinger representations that come from different choices of lagrangians of $W$ are equivalent. We now construct cannonical intertwiners between Schrodinger representations corresponding to different lagrangians. Let $\ell_1$ and $\ell_2$ be two langrangians of $W$. There is a unitary intertwiner $\mathcal{F}_{\ell_{1},\ell_2}:\mathcal{H}_{\ell_1}\rightarrow \mathcal{H}_{\ell_2}$ defined by
$$
(\mathcal{F}_{\ell_1,\ell_2})(\phi)=\int_{x\in \ell_2/(\ell_1\cap \ell_2)}\phi(x+y)\psi\left(\frac{\langle x,y \rangle}{2}\right)d\mu_{1,2}^{1/2},
$$
where $\phi\in \mathcal{H}_{\ell_1}$. 

For each $g\in \text{Sp}(W)$, we have an isomorphism $g_*:\mathcal{H}_\ell\simeq \mathcal{H}_{g\ell}$ via $s(\cdot)\mapsto s(g^{-1}(\cdot))$. We define the following operator
$$
M_{\ell}^{\text{Sch}}[g]:=\mathcal{F}_{\ell,g\ell}\circ g_*= g_* \circ \mathcal{F}_{g^{-1}\ell,\ell}.
$$
For any $g,h\in \text{Sp}(W)$, we have
$$
M_{\ell}^{\text{Sch}}[g]\circ M_{\ell}^{\text{Sch}}[h]=\gamma_\psi(\tau(\ell,h\ell,gh\ell))M_{\ell}^{\text{Sch}}[gh].
$$
We set $c_{g,h}(\ell)=\gamma_\psi(\tau(\ell,h\ell,gh\ell))$, which defines a $2$-cocycle on $\text{Sp}(W)$. This is called Maslov cocycle. For any $g\in \text{Sp}(W)$ and $\ell \in \Lambda(W)$ (the set of lagrangians of $W$), we denote
$$
m_g(\ell)=\gamma_\psi(1)^{\frac{\dim W}{2}-\dim g\ell\cap \ell -1}\gamma_\psi(A_{g\ell,\ell}),
$$
where $A_{g\ell,\ell}$ is the perfect pairing on $g\ell \times \ell$. We define the metaplectic group $\text{Mp}(W)$ to be the set of pairs $(g,t)$, where $g\in \text{Sp}(W)$ and $t:\Lambda(W)\rightarrow \mathbb{C}^\times$, satisfying the following properties
\begin{enumerate}
    \item $t(\ell)^2=m_g(\ell)^2$.
    \item For any $\ell,\ell^\prime\in \Lambda(W)$, we have $t(\ell^\prime)=\gamma_\psi(\tau(\ell,g\ell,g\ell^\prime,\ell^\prime))t(\ell)$.
    \item $(g,t)\cdot (g^\prime,t^\prime)=(gg^\prime,tt^\prime\cdot c_{g,g^\prime})$.
\end{enumerate}

We now explicate the Weil representation $\omega_\psi$. For $(g,t)\in \text{Mp}(W)$, we set 
$$
\omega_\psi(g,t)=t(\ell)\cdot M_{\ell}^{\text{Sch}}[g]=t(\ell)\cdot \mathcal{F}_{\ell,g\ell}\circ g_*.
$$
This is call the Schrodinger model of the Weil representation, which depends on choices of $\ell$. We recall \cite[Theorem 9.12]{GKT25}, which provides an explicit formula for the action of $\text{Mp}(W)$ on $\Omega_{1/2}(W/\ell)$. Noting that here we have reformulated the original theorem by the language of densities.
\begin{theorem}
    The action of $\text{Mp}(W)$ is explicitly described as follows: for any $\phi\in \Omega_{1/2}(W/\ell)$ and $y\in W/\ell$,
    \begin{itemize}
        \item 
        $\omega_\psi\left(
    \begin{pmatrix}
    a &  \\
     & a^{*,-1}
\end{pmatrix}, 1
    \right)\phi(y)=\gamma_\psi(\det(a))^{-1}|\det(a)|^{\frac{1}{2}}\phi(a^*y)$;
        \item 
    $\omega_\psi\left(
    \begin{pmatrix}
    1 & b \\
     & 1
\end{pmatrix}, 1
    \right)\phi(y)=\psi(\langle y,-by\rangle)\phi(y)$;
    \item
    $\omega_\psi\left(w_i, 1\right)\phi(y)=\gamma_\psi(1)^{i}\mathcal{F}_i(\phi)(y)$;
    \item
    $\omega_\psi\left(1,z\right)\phi(y)=z\phi(y)$;
    \end{itemize}
    where each element of $\text{Sp}(W)$ is realized as a $2\times 2$ matrix with respect to the polarization and $\mathcal{F}_i$ is the partial Fourier transform defined in \cite[(2.22)]{GKT25}. The representation of $\text{Mp}(W)$ on $\Omega_{1/2}(W/\ell)$ is uniquely characterised by the above actions.
\end{theorem}

\subsection{A local character expansion of the Weil representation}\label{sec3.2}
In this section, we study the behavior of the Weil representation $\omega_{V,\psi,\mu}$ of $U(V)$ near the identity element. We denote by $\mathfrak{u}(V)$ the Lie algebra of $U(V)$. Let  
$$
\begin{matrix}
\Phi: & V & \longrightarrow & \mathfrak{u}(V)^*  \\
 & v & \mapsto & (X\mapsto \langle v,Xv \rangle_V)  
\end{matrix}
$$
be the moment map for the Hamiltonian space $(U(V),V)$. We first consider a local character expansion of the Weil representation near the identity element.
\begin{proposition}\label{identity}
    Let $\omega$ be a neighborhood of $0$ in $\mathfrak{u}(V)(F)$. When $\omega$ is sufficiently small, then for any smooth function $f\in C^\infty(\omega)$, we have
    $$
    \text{Tr}(\omega_{V,\psi,\mu}(f\circ \log))=\int_{V}(\hat{f}\circ\Phi)(v)dv.
    $$
\end{proposition}
\begin{proof}
    Let $\ell$ be a langrangian of the symplectic space $\text{Res}_{E/F}V$. Denote
    $$
    U(V)^\ell=\{g\in U(V):\ g\ell\cap \ell =\{0\}\text{ and }\det(g-1)\neq 0\}
    $$
    and 
    $$
    \omega^\ell=\{X\in \omega:\ \exp(X)\ell\cap \ell =\{0\}\text{ and }\det(X)\neq 0\}.
    $$ 
    We fix $\phi\in \Omega_{1/2}(V/\ell)$. For each $g\in U(V)^\ell$, observe
    $$
    (\mathcal{F}_{\ell,g\ell}(\phi))(y)=\int_{x\in V/g\ell}\phi(x)\psi\left(\frac{Q_{g\ell,\ell}(x,y)}{2}\right)\mu_{g\ell,\ell}^{1/2},   
    $$
    where $Q_{g\ell,\ell}$ is the quadratic form defined in \cite[Section 10.2]{Tho08}. 
    
    We denote by $\mu: U(V) \rightarrow \overline{\text{Sp}}(\text{Res}_{E/F}V)$ the embedding of $U(V)$ into the metaplectic cover. We choose $\omega$ sufficiently small so that $\mu(\exp{X})=(\exp X,\gamma_\psi(1)^{n}\gamma_\psi(A_{(\exp{X})\ell,\ell})\cdot M_\ell^{\text{Sch}}[\exp X])$, for any $X\in \omega$. For any smooth function $f$ on $\omega$, we have 
    $$
    (\omega_{V,\psi,\mu}(f\circ \log)\phi)(y)=\int_{\mathfrak{u}(V)(F)}f(X) \gamma_\psi(1)^{n}\gamma_\psi(A_{(\exp{X})\ell,\ell})\int_{V/\ell}\phi(x)\psi\left(\frac{Q_{(\exp X)\ell,\ell}(x,y)}{2}\right)dxdX
    $$
    $$
    =\int_{V/\ell}\phi(x)\int_{\mathfrak{u}(V)(F)}f(X) \gamma_\psi(1)^{n}\gamma_\psi(A_{(\exp{X})\ell,\ell})\psi\left(\frac{Q_{(\exp X)\ell,\ell}(x,y)}{2}\right)dXdx.
    $$
    This gives us
    $$
    \text{tr}(\omega_{V,\psi,\mu}(f\circ \log))=\int_{V/\ell}\int_{\mathfrak{u}(V)(F)}f(X) \gamma_\psi(1)^{n}\gamma_\psi(A_{(\exp{X})\ell,\ell})\psi\left(\frac{Q_{(\exp X)\ell,\ell}(x,x)}{2}\right)dXdx.
    $$
    Let $h$ be a compactly-supported smooth function on $V/\ell$ whose measure is positive and $h(0)=1$. We set $h_s(\cdot)=h(s\cdot)$, for $s\in F$. Then 
    $$
    \text{tr}(\omega_{V,\psi,\mu}(f\circ \log))=\lim_{s\rightarrow 0}\text{tr}(h_s\cdot \omega_{V,\psi,\mu}(f\circ \log)).
    $$
    Observe
    $$
    \text{tr}(h_s\cdot \omega_{V,\psi,\mu}(f\circ \log))=\int_{V/\ell}h_s(x)\int_{\mathfrak{u}(V)(F)}f(X) \gamma_\psi(1)^{n}\gamma_\psi(A_{(\exp{X})\ell,\ell})\psi\left(\frac{Q_{(\exp X)\ell,\ell}(x,x)}{2}\right)dXdx
    $$
    $$
    =\int_{\mathfrak{u}(V)(F)}f(X)\gamma_\psi(1)^{n}\gamma_\psi(A_{(\exp{X})\ell,\ell}) \int_{V/\ell}h_s(x)\psi\left(\frac{Q_{(\exp X)\ell,\ell}(x,x)}{2}\right)dxdX.
    $$
    By \cite[Proposition 10.2]{Tho08}, $Q_{(\exp X)\ell,\ell}$ is equivalent to $\tau(\Gamma_{\exp X},\Gamma_1,\ell\oplus \ell)$ in the Witt group. Moreover, when $X$ is sufficiently small, we have 
    $$
    \gamma_\psi(1)^{n}\gamma_\psi(A_{(\exp{X})\ell,\ell})\psi\left(\frac{\tau(\Gamma_{\exp X},\Gamma_1,\ell\oplus \ell)(v,v)}{2}\right)=\psi\left(\frac{\tau(\Gamma_{-\exp X},\Gamma_{\exp X},\Gamma_1)(v,v)}{2}\right)=\psi(\langle v,Xv\rangle),
    $$
    for any $x\in V$. Noting that the above equality follows from the fact that 
    $$
    \gamma_\psi(A_{(\exp{X})\ell,\ell})\gamma_\psi\left(\tau(\Gamma_{\exp X},\Gamma_1,\ell\oplus \ell)\right)=\gamma_\psi(A_{(-\exp{X})\ell,\ell})\gamma_\psi\left(\tau(\Gamma_{-\exp X},\Gamma_1,\ell\oplus \ell)\right)\gamma_\psi(\tau(\Gamma_{-\exp X},\Gamma_{\exp X},\Gamma_1))
    $$
    and $\gamma_\psi(1)^n\gamma_\psi(A_{(-\exp{X})\ell,\ell})\gamma_\psi\left(\tau(\Gamma_{-\exp X},\Gamma_1,\ell\oplus \ell)\right)=1$ when $X$ is sufficiently small. Thus, by shrinking $\omega$ further, we obtain
    $$
     \text{tr}(h_s\cdot \omega_{V,\psi,\mu}(f\circ \log))=\int_{\mathfrak{u}(V)(F)}f(X)\int_{V}h^\prime_s(v)\cdot \psi(\langle v,Xv \rangle)dvdX
    $$
    $$
    =\int_{V}h^\prime_s(v)\int_{\mathfrak{u}(V)(F)}f(X) \psi(\langle v,Xv \rangle)dXdv,
    $$
    where $h^\prime_s$ is the function on $V$ induced from $h_s$ under the equivalence between $Q_{(\exp X)\ell,\ell}$ and $\tau(\Gamma_{\exp X},\Gamma_1,\ell\oplus \ell)$. Therefore
    $$
    \text{tr}(\omega_{V,\psi,\mu}(f\circ \log))=\lim_{s\rightarrow0}\text{tr}(h_s\cdot\omega_{V,\psi,\mu}(f\circ \log))
    $$
    $$
    =\int_{V}\int_{\mathfrak{u}(V)(F)}f(X) \psi(\langle v,Xv \rangle)dXdv=\int_V (\hat{f}\circ \Phi)(v)dv
    $$
    as desired.
\end{proof}
\subsection{A local trace formula for the twisted Gan-Gross-Prasad conjecture}

Let $F$ be a nonarchimedean local field of characteristic $0$. Let $E$ and $K$ be two quadratic field extension over $F$. We set $L=E\otimes_F K$. Let $V$ be a non-degenerate $n$-dimensional skew-hermitian space over $E$ and $V_K=V\otimes_F K$, which is naturally a skew-hermitian space with respect to $L/K$. Noting that the skew-hermitian space $V_K$ does not depend on choices of $V$ (cf. \cite[Lemma 8.1]{GGP23}). We denote $U\left(V\right)$ by $H_{V}$ and $\text{Res}_{K/F}\ U\left(V_K\right)$ by $G$.

Let $\psi$ be a nontrivial additive character of $F$ and $\mu$ be a conjugate-symplectic character of $E^{\times}$. We denote by $\omega_{V,\psi,\mu}$ the Weil representation associated to $\left(V,\psi,\mu\right)$.
Let $\pi$ be an irreducible generic representation of $G\left(F\right)$. We consider the problem of determining 
$$
m_{V}\left(\pi\right)=\dim\text{Hom}_{H_{V}}\left(\pi,\omega_{V,\psi,\mu}\right).
$$
We call a triple $\left(G,H_{V},\omega_{V,\psi,\mu}\right)$ a twisted Gan-Gross-Prasad triple. From now to the end of section \ref{sec5}, we prove the following theorem by induction. 
\begin{theorem}\label{smo}
    Let $\phi$ be a tempered L-parameter of $G(F)$ and $\Pi_{\phi}$ be the L-packet corresponding to $\phi$ in the local Langlands correspondence of $G\left(F\right)$. Then 
    $$
    \sum_{V}\underset{\pi\in\Pi_{\phi}}{\sum}\,m_{V}\left(\pi\right)=1.
    $$
\end{theorem}
In \cite{GGP23}, they have shown the above theorem for $n=1,2$. We now consider $n\geq 3$. \textbf{The proof under
the induction hypothesis only finishes at the end of section \ref{sec5}.} 

For $f\in{\mathcal{C}_{\text{scusp}}}\left(G\left(F\right)\right)$, let us define a kernel function $K_V\left(f,\cdot\right)$ on $Z_{G}\left(F\right)H\left(F\right)\backslash G\left(F\right)$
by 
\[
K_V\left(f,g\right)=\text{Trace}\left(\omega_{V,\psi,\mu}\left(^{g}f\right)\right)
=
\underset{i}{\sum}\int_{H_V\left(F\right)}f\left(g^{-1}hg\right)\left\langle \phi_{i},\omega_{V,\psi,\mu}\left(h\right)\phi_{i}\right\rangle dh,
\]
where $g\in Z_{G}\left(F\right)H_V\left(F\right)\backslash G\left(F\right)$
and $\left\{ \phi_{i}\right\} _{i\in I}$ is an orthonormal basis of $\omega_{V,\psi,\mu}$. We define the following linear form 
\[
J_V\left(f\right)=\int_{Z_{G}\left(F\right)H\left(F\right)\backslash G\left(F\right)}K_V\left(f,g\right)dg.
\]
We set 
\[
J_{V,\text{spec}}\left(f\right)=\int_{{\mathcal{X}}\left(G\right)}D\left(\pi\right)\hat{\theta}_{f}\left(\pi\right)m_V\left(\bar{\pi}\right)d\pi,
\]
for all $f\in{\mathcal{C}}_{\text{scusp}}\left(Z_G(F)\backslash G\left(F\right)\right)$.
We state the following spectral expansion of $J_V(f)$, whose proof follows from \cite[Theorem 5.1]{Le25} verbatim.
\begin{theorem}\label{spectral}
For any $f\in{\mathcal{C}}_{\text{scusp}}\left(G\left(F\right)\right)$, the linear forms $K_V\left(f,g\right)$ and $J_V\left(f\right)$ are absolutely convergent. Under the induction hypothesis, we have the following spectral expansion 
\[
J_V\left(f\right)=J_{V,\text{spec}}\left(f\right),
\]
where ${\mathcal{X}}\left(G\right)$ consists of virtual representations of the form $i_{M}^{G}\left(\sigma\right)$ where $M$ is a Levi subgroup of $G$ and $\sigma$ is an elliptic representation of $M$.
\end{theorem}

\section{Geometric expansion and a reduction to the Lie algebra}\label{sec4}
We define a linear form $m_{V,\text{geom}}$ on the space of quasi-characters $\text{QC}(G(F))$. We would like to show that $J_{V}(f)=m_{V,\text{geom}}(\theta_f)$ for any $f\in \mathcal{C}_{\text{scusp}}(G(F))$. In this section, we prove the two linear forms coincide when $f$ is supported outside the origin, and thus deduce the comparison to their infinitesimal variants.

\subsection{The linear forms $m_{V,\text{geom}}$ and $m^{\text{Lie}}_{V,\text{geom}}$}\label{sec4.1}
We first define the geometric support for the desired linear form. We denote by $\mathcal{T}_{\text{ell}}(H_V)$ the set containing representatives of $H_V(F)$-conjugacy classes of elliptic maximal tori in $H$. Let $T\in \mathcal{T}_{\text{ell}}(H_V)$. Since $T$ is elliptic, it follows that $T(F)$ is isomorphic to $\prod_{i}U_{E_i/F_i}(1)$, where $F_i$ is a field extension of $F$ not containing $E$ and $E_i=E\otimes_F F_i$. We set 
$$
\gamma_\psi(T)=\prod_i \gamma_\psi(2Nm_{E_i/F_i}),
$$
where $\gamma_\psi(2Nm_{E_i/F_i})$ is the Weil index associated to $\psi$ and the quadratic form $2Nm_{E_i/F_i}$.

For any quasi-character $\theta \in \text{QC}(G(F))$, we define
$$
m_{V,\text{geom}}(\theta)=\frac{1}{2}c_{\theta}\left(1\right)+\mu\left(\det V\right)
\underset{T\in{\mathcal{T}}_{\text{ell}}\left(H\right)}{\sum}\frac{\gamma_\psi(T)}{\left|W\left(H,T\right)\right|}
\underset{s\rightarrow0^{+}}{\lim}\int_{T\left(F\right)}D^{G}\left(x\right)^{1/2}c_{\theta}\left(x\right)\frac{\mu\left(\det\left(1-x^{-1}\right)\right)}{\left|\det\left(1-x\right)\right|_{E}^{1/2-s}}dx.
$$
For any (virtual) tempered representation $\pi$ of $G(F)$, we set the geometric multiplicity 
$$
m_{V,\text{geom}}(\pi)=m_{V,\text{geom}}(\theta_\pi).
$$
Let $x\in G(F)_{\text{ell}}\cap H_V(F)$. Since $(G_x,H_x,\omega_{V,\psi,\mu}|_{H_x})$ is a product of twisted Gan-Gross-Prasad triples, we can define a linear form $m_{V,x,\text{geom}}$ on $\text{QC}(G_x(F))$ as the product of linear forms $m_{-,\text{geom}}$ of corresponding twisted Gan-Gross-Prasad triples.

Likewise, we can define a Lie algebra variant of $m_{V,\text{geom}}(\theta)$. We denote by $\mathcal{T}_{\text{ell}}(\mathfrak{h}_V)$ the set containing representatives of Lie algebras of $H_V(F)$-conjugacy classes of elliptic maximal tori in $H_V$. When $T\in \mathcal{T}_{\text{ell}}(H_V)$ and $\mathfrak{t}$ is the Lie algebra of $T$, we set $\gamma_\psi(\mathfrak{t})=\gamma_\psi(T)$. For quasi-characters $\theta \in \text{QC}_c(\mathfrak{g}(F))$, we define
$$
m^{\text{Lie}}_{V,\text{geom}}(\theta)=\frac{1}{2}c_{\theta}\left(0\right)+\mu\left(\det V\right)
\underset{\mathfrak{t}\in{\mathcal{T}}_{\text{ell}}\left(\mathfrak{h}_V\right)}{\sum}\frac{\gamma_\psi(\mathfrak{t})}{\left|W\left(H_V,\mathfrak{t}\right)\right|}
\underset{s\rightarrow0^{+}}{\lim}\int_{\mathfrak{t}\left(F\right)}D^{G}\left(X\right)^{1/2}c_{\theta}\left(X\right)\frac{\mu\left(\det\left(X\right)\right)}{\left|\det\left(X\right)\right|_{E}^{1/2-s}}dx.
$$

We now show that the above definitions are well-defined.
\begin{proposition}\label{4.1}
Let $T\in \mathcal{T}_{\text{ell}}(H_V)$. We denote by $\mathfrak{t}$ the Lie algebra of $T$.
    \begin{enumerate}
        \item For any $\theta\in \text{QC}_c(G(F))$, the following integral is absolutely convergent if $\text{Re}(s)>0$
        $$
        \int_{T\left(F\right)}D^{G}\left(x\right)^{1/2}c_{\theta}\left(x\right)\frac{\mu\left(\det\left(1-x^{-1}\right)\right)}{\left|\det\left(1-x\right)\right|_{E}^{1/2-s}}dx
        $$
        and the following limit
        $$
        \underset{s\rightarrow0^{+}}{\lim}\int_{T\left(F\right)}D^{G}\left(x\right)^{1/2}c_{\theta}\left(x\right)\frac{\mu\left(\det\left(1-x^{-1}\right)\right)}{\left|\det\left(1-x\right)\right|_{E}^{1/2-s}}dx
        $$
        exists.
        \item Let $x\in G(F)_{\text{ell}}\cap H_V(F)$ and $\Omega_x\subseteq G_x(F)$ be a $G$-good open neighborhood of $x$ and set $\Omega=\Omega_x^G$. Then, if $\Omega_x$ is sufficiently small, we have
        $$
        m_{V,\text{geom}}(\theta)=\sum_{x^\prime \sim x}m_{V,x^\prime,\text{geom}}(\theta_{x,\Omega_x}),
        $$
        for any $\theta \in \text{QC}_c(\Omega)$. Noting that here we are taking the sum over $H_V(F)$-conjugacy classes of $x$.
        \item For any $\theta\in \text{QC}_c(\mathfrak{g}(F))$, the following integral is absolutely convergent if $\text{Re}(s)>0$
        $$
        \int_{\mathfrak{t}\left(F\right)}D^{G}\left(X\right)^{1/2}c_{\theta}\left(X\right)\frac{\mu\left(\det\left(X\right)\right)}{\left|\det\left(X\right)\right|_{E}^{1/2-s}}dX
        $$
        and the following limit
        $$
        \underset{s\rightarrow0^{+}}{\lim}\int_{\mathfrak{t}\left(F\right)}D^{G}\left(X\right)^{1/2}c_{\theta}\left(X\right)\frac{\mu\left(\det\left(X\right)\right)}{\left|\det\left(X\right)\right|_{E}^{1/2-s}}dX
        $$
        exists. Moreover, we have
        $$
        m_{V,\text{geom}}^{\text{Lie}}(\theta_\lambda)=|\lambda|^{\delta(G)/2}m_{V_\lambda,\text{geom}}^{\text{Lie}}(\theta),
        $$
        for any $\theta\in \text{QC}_c(\mathfrak{g}(F))$ and $\lambda \in F^\times$. Here $\theta_\lambda(X)=\theta(\lambda^{-1}X)$ for any $X\in \mathfrak{g}_{\text{reg}}(F)$, and $V_\lambda$ is the skew-hermitian space corresponding to the skew-hermitian form $\lambda\langle\cdot,\cdot\rangle_V$.
        \item Let $\omega\subseteq \mathfrak{g}(F)$ be a $G$-excellent open neighborhood of $0$ and set $\Omega=\exp(\omega)$. Then
        $$
        m_{V,\text{geom}}(\theta)=m_{V,\text{geom}}^{\text{Lie}}(\theta_{\omega}),
        $$
        for any $\theta \in \text{QC}_c(\omega)$.
    \end{enumerate}
\end{proposition}
\begin{proof}
    The proof follows from \cite[Proposition 11.2.1]{BP20}.
\end{proof}
\subsection{Statement of the main theorem and rank one case}\label{sec4.2}
We now state our first main theorem, whose proof takes place in the remaining of section \ref{sec4} and \ref{sec5} via induction. \textbf{The proof only finishes at the end of section \ref{sec5.6}.}
\begin{theorem}\label{maintheorem1}
For any $f\in \mathcal{C}_{scusp}(G(F))$, we have 
$$
J_V\left(f\right)=m_{V,\text{geom}}\left(\theta_{f}\right).
$$
Therefore, for any tempered representation $\pi$ of $G(F)$, 
$$
m_V\left(\pi\right)=m_{V,\text{geom}}\left(\theta_{\pi}\right).
$$
\end{theorem}
We consider the case when $\dim V=1$. In this case, the trace formula we consider is 
\[
J_V\left(f\right)=\text{Trace}\left(\omega_{V,\psi,\mu}\left(f\right)\right),
\]
where $f\in{\mathcal{C}}_{\text{scusp}}\left(G\left(F\right)\right)=C^\infty\left(L^{1}\right)$. Here $L^1$ is the subgroup containing elements $a\in L^\times$ such that $\text{Nm}_{L/K}(a)=1$.
\begin{proposition}\label{4.3}
    When $\dim V=1$, we have
    $$J_{V}(f)=\frac{1}{2}f(1)+\mu(\det V)\gamma_{\psi}(2Nm_{E/F})\lim_{s\rightarrow 0^+}\int_{E^1}f(x)\mu(1-x^{-1})|1-x|_E^{s-1/2}dx.$$
\end{proposition}
\begin{proof}
Let $\chi$ be a continuous character of $L^1$. When considering $\chi\times {}^\sigma\chi^{-1}$ as a character of $WD_L^{ab}\simeq L^\times$, we can see that $As_{L/E}(\chi)$ is its restriction to $E^\times$. Let $\delta$ be a nonzero $\text{Tr}_{E/F}$-zero element in $E$. Applying \cite[Lemme A.1 and (10) in pg. 1360]{BP15} for the character $\chi \times {}^\sigma\chi^{-1}\times \mu^{-1}$, we have 
$$
\epsilon\left(\frac{1}{2},As_{L/E}(\chi)\times \mu^{-1},\psi_E^\delta\right)=2\omega_{E/F}(-1)\gamma_{\psi}(2Nm_{E/F})\lim_{s\rightarrow 0^+}\int_{E^1}(\chi \times {}^\sigma\chi^{-1}\times \mu^{-1})(\delta^{-1}(1-x))|1-x|_E^{s-1/2}dx
$$
$$
=2\omega_{E/F}(-1)\mu(\delta)\gamma_{\psi}(2Nm_{E/F})\lim_{s\rightarrow 0^+}\int_{E^1}\chi(x)\mu(1-x^{-1})|1-x|_E^{s-1/2}dx.
$$
Since $\epsilon\left(\frac{1}{2},As_{L/E}(\chi)\times \mu^{-1},\psi_E^\delta\right)=\chi(-1)\omega_{E/F}(-1)\mu(\delta)\epsilon\left(\frac{1}{2},As_{L/E}(\chi)\times \mu^{-1},\psi_E\right)$, it follows that
$$
\epsilon\left(\frac{1}{2},As_{L/E}(\chi)\times \mu^{-1},\psi_E\right)=2\chi(-1)\lim_{s\rightarrow 0^+}\int_{E^1}(\chi \times {}^\sigma\chi^{-1}\times \mu^{-1})(\delta^{-1}(1-x))|1-x|_E^{s-1/2}dx.
$$
The spectral expansion in the case $n=1$ gives us 
$$
J_V(f)= \sum_{\chi\in \hat{L}^1}m_V(\chi)\int_{L^1}f(x)\chi(x)^{-1}dx.
$$
By \cite[Theorem 3.1]{GGP23}, we have
$$
m_V(\chi)=\frac{1+\mu(\det V)\chi(-1)\epsilon\left(\frac{1}{2},As_{L/E}(\chi)\times \mu^{-1},\psi_E\right)}{2}.
$$
From the above discussion, the right hand side is equal to
$$
\frac{1}{2}+\mu(\det V)\gamma_{\psi}(2Nm_{E/F})\lim_{s\rightarrow 0^+}\int_{E^1}\chi(x)\mu(1-x^{-1})|1-x|_E^{s-1/2}dx.
$$
Applying the Fourier inverse formula for $f$, we obtain
$$
J_V(f)=\sum_{\chi\in \hat{L}^1}m_V(\chi)\int_{L^1}f(x)\chi(x)^{-1}dx
$$
$$
=\frac{1}{2}f(1)+\mu(\det V)\gamma_{\psi}(2Nm_{E/F})\lim_{s\rightarrow 0^+}\int_{E^1}f(x)\mu(1-x^{-1})|1-x|_E^{s-1/2}dx.
$$
We have finished our proof for Proposition \ref{4.3}.
\end{proof}

We now consider the case $n\geq2$. We first prove a compatibility between the geometric multiplicity and parabolic induction. We denote by $E^\prime$ be the third quadratic subfield of $L$ when $K\neq E$ and $E^\prime = F$ when $K=E$. Let $P=LU$ be a parabolic subgroup of $G$. We can write the Levi subgroup $L$ as a product of $\prod_{i=1}^dL_i\times \tilde{G}$, where $L_i$ is a general linear group over $L$, for $i=\overline{1,d}$, and $\tilde{G}$ is a unitary group over $K$ of smaller rank. 

We recall the notion of $V$-relevant flag in \cite{GGP23}. A flag $\mathcal{L}=\{V_1,\ldots,V_d,V_{d+1}\}$ of $V_K$, where $V_i$ is a skew-hermitian space relative to $L/E^\prime$ if $i=\overline{1,d}$ and relative to $L/K$ if $i=d+1$, is call $V$-relevant if 
$$
\det(V)=\prod_i \det(\text{Res}_{L/E}(V_i)).
$$
Assume $P$ stabilizes $\mathcal{L}$. We define the linear form on $\text{QC}(L(F))=\otimes_{i=1}^d\text{QC}(L_i(F))\otimes \text{QC}(\tilde{G}(F))$ associated to $\mathcal{L}$ by
$$
m^{\mathcal{L}}_{V,\text{geom}}(\otimes_{i=1}^d\theta_i\otimes \tilde{\theta})=\left(\prod_{i=1}^d m^{L_i}_{V_i,\text{geom}}(\theta_i)\right)m_{V_{d+1},\text{geom}}^{\tilde{L}}(\tilde{\theta}),
$$
where $\theta_i\in \text{QC}(L_i(F))$ for $i=\overline{1,d}$ and $\tilde{\theta}\in\text{QC}(\tilde{G}(F))$. Noting that here $m^{L_i}_{V_i,\text{geom}}$ and $m_{V_{d+1},\text{geom}}^{\tilde{L}}$ are linear forms corresponding to the twisted Gan-Gross-Prasad triples $(L_i,U(V_i),\omega_{V_i,\psi,\mu})$ and $(\tilde{L},U(V_{d+1}),\omega_{V_{d+1},\psi,\mu})$. The following lemma gives us the compatibility of $m_{V,\text{geom}}$ with parabolic induction.

\begin{lemma}\label{4.4}
Let $\theta^{L}\in\text{QC}\left(L\left(F\right)\right)$ and $\theta=i_{L}^{G}\left(\theta^{L}\right)$. We have  
$$
m_{V\text{geom}}\left(\theta\right)=\sum_{\mathcal{L}}\,m_{V,\text{geom}}^{\mathcal{L}}\left(\theta^{L}\right),
$$
where $\mathcal{L}$ runs over the set of $V$-relevant flags stabilized by $P$.
\end{lemma}

\begin{proof}
    The proof can be adapted from \cite[Lemme 17.2.1]{BP16} without any difficulties.
\end{proof}

By \cite[Theorem 4.8, Theorem 10.1]{GGP23}, Lemma \ref{4.4} and the induction hypothesis, it follows that 
$$
m_V\left(\pi\right)=m_{V,\text{geom}}\left(\pi\right),
$$
for any $\pi\in R_{\text{ind}}\left(G\right)$. By Theorem \ref{spectral}, for any strongly cuspidal function $f$ on $G(F)$, 
$$
J_V\left(f\right)=J_{V,\text{spec}}\left(f\right)=\int_{{\mathcal{X}}\left(G\right)}D\left(\pi\right)\hat{\theta}_{f}\left(\pi\right)m_V\left(\bar{\pi}\right)d\pi
$$
\[
=\int_{{\mathcal{X}}_{\text{ind}}\left(G\right)}D\left(\pi\right)\hat{\theta}_{f}\left(\pi\right)m_V\left(\bar{\pi}\right)d\pi+\underset{\pi\in{\mathcal{X}}_{\text{ell}}\left(G\right)}{\sum}D\left(\pi\right)m_V\left(\bar{\pi}\right)\int_{\Gamma\left(G\right)_{\text{ell}}}D^{G}\left(x\right)\theta_{f}\left(x\right)\theta_{\pi}\left(x\right)dx
\]
\[
=m_{V,\text{geom}}\left(\theta_{f}\right)+\underset{\pi\in{\mathcal{X}}_{\text{ell}}\left(G\right)}{\sum}D\left(\pi\right)\left(m_V\left(\bar{\pi}\right)-m_{V,\text{geom}}\left(\bar{\pi}\right)\right)\int_{\Gamma\left(G\right)_{\text{ell}}}D^{G}\left(x\right)\theta_{f}\left(x\right)\theta_{\pi}\left(x\right)dx.
\]
We set  
\[
J_{V,\text{qc}}\left(\theta\right)=m_{V,\text{geom}}\left(\theta\right)+\underset{\pi\in{\mathcal{X}}_{\text{ell}}\left(G\right)}{\sum}D\left(\pi\right)\left(m_V\left(\bar{\pi}\right)-m_{V,\text{geom}}\left(\bar{\pi}\right)\right)\int_{\Gamma\left(G\right)_{\text{ell}}}D^{G}\left(x\right)\theta\left(x\right)\theta_{\pi}\left(x\right)dx,
\]
for $\theta\in\text{QC}\left(G\left(F\right)\right)$. Since $\text{Supp}\left(m_{V,\text{geom}}\right)\subseteq\Gamma\left(G\right)_{\text{ell}}$,
we can see that $\text{Supp}\left(J_{V,\text{qc}}\right)\subseteq\Gamma\left(G\right)_{\text{ell}}$.
Moreover, from the above computation, by substituting $\theta$ to
be $\theta_{\bar{\pi}}$ and the orthogonality relations of Arthur for elliptic representations, we can see that the statement 
\[
J_{V,\text{qc}}\left(\theta\right)=m_{V,\text{geom}}\left(\theta\right),\text{ for any }\theta\in\text{QC}\left(G\left(F\right)\right)
\]
implies 
\[
m_V\left(\pi\right)=m_{V,\text{geom}}\left(\pi\right),\text{ for any }\pi\in{\mathcal{X}}_{\text{ell}}\left(G\right),
\]
which is to say $J_V(f)=m_{V,\text{geom}}(\theta_f)$. Thus, from now it suffices to show that 
$$
J_{V,\text{qc}}(\theta)=m_{V,\text{geom}}(\theta),
$$
for any quasi-character $\theta$ on $G(F)$.

\subsection{Descent to non-central elliptic semisimple elements}\label{sec4.3}
In this section, we show that the two linear forms $J_{V,\text{qc}}$ and $m_{V,\text{geom}}$ coincide when supported outside central elements of $H_V(F)$.
\begin{proposition}\label{elliptic}
Let $\theta\in QC\left(G\left(F\right)\right)$ and assume that $Z_{H_V}\left(F\right)\cap\text{Supp}\left(\theta\right)=\emptyset$.
Then 
\[
J_{V,\text{qc}}\left(\theta\right)=m_{V,\text{geom}}(\theta).
\]
\end{proposition}

\begin{proof}
Observe that $J_{V,\text{qc}}$ is supported in $\Gamma_{\text{ell}}\left(G\right)$. By a process of partition of unity, it suffices to prove the equality
for $\theta\in\text{QC}\left(\Omega\right)$,
where $\Omega$ is a completely stably $G\left(F\right)$-invariant
open subset of $G\left(F\right)$ of the form $\Omega_{x}^{G}$
for some noncentral elements $x\in G\left(F\right)_{\text{ell}}$ and $\Omega_{x}\subseteq G_{x}\left(F\right)$ is a $G$-good open neighborhood of $x$. We can assume that $\Omega_{x}$ is relatively compact modulo conjugation.

We first consider the case when $x$ is not $G\left(F\right)$-conjugate to any element of $H_{V}\left(F\right)$. Since the set $\Gamma_G\left(H_V\right)$ containing $G(F)$-conjugacy classes in $H_V(F)$ is closed in $\Gamma\left(G\right)$, when $\Omega_{x}$ is sufficiently small, we have $\Omega_{x}^{G}\cap\Gamma\left(H\right)=\emptyset$. In this case, it is easy to see that $\Omega_{x}^{G}$ has no contribution in both $J_{\text{qc},V}$ and $m_{V,\text{geom}}$.

There remains to consider the case when $x$ is $G\left(F\right)$-conjugate to some elements of $H_{V}\left(F\right)$. We may take $x\in H_{V}\left(F\right)$. In this case, we have  
\[
G_{x}=G_{1}\times\ldots\times G_{m}\ \ \ \ \text{and}\ \ \ \ H_{V,x}=H_{1}\times\ldots\times H_{m},
\]
where $H_{1},\ldots,H_{m}$ are certain unitary groups and $G_{i}=\text{Res}_{K/F}H_{i,K}$. We can choose $\Omega_x$ such that $\Omega_{x}\cap H_{x}\left(F\right)$ is a $H$-good
open neighborhood of $x$ and 
\[
\Omega\cap H(F)=\bigsqcup_{i=1}^{2^{m-1}}\left(g_{i}^{-1}(\Omega_{x}\cap H_{x}\left(F\right)\right)g_{i})^{H},
\]
for some $g_{i}\in G\left(F\right)$. We set $x_i=g_i^{-1}xg_i$ and $\Omega_{x_i}=g_i^{-1}\Omega_xg_i$.

By shrinking $\Omega_{x}$ further, we assume that $\Omega_{x}=\Omega_{1}\times\ldots\times\Omega_{m}$, where $\Omega_{i}\subseteq G_{i}\left(F\right)$ is open and completely $G_{i}\left(F\right)$-invariant. Since $\text{QC}\left(\Omega_{x}\right)=\underset{i}{\bigotimes}\,\text{QC}\left(\Omega_{i}\right)$,
it suffices to consider $\theta_{x,\Omega_{x}}=\otimes_i\theta_{i}$,
where $\theta_{i}$ lies in $\text{QC}\left(\Omega_{i}\right)$. For each $i=\overline{1,m}$, we choose $f_{x,i}\in \mathcal{C}_{\text{scusp}}(\Omega_i)$ such that $\theta_{f_{x,i}}=\theta_i$. We denote $f_x=\otimes_if_{x,i}$. Let $f=\otimes_if_i=\widetilde{f_x}\in{\mathcal{C}}_{\text{scusp}}\left(\Omega\right)$ be a lift defined in Proposition \ref{2.8}. We have  
\[
J_{V,\text{qc}}(\theta_f)=J_{V}\left(f\right)=\int_{Z_{G}\left(F\right)H_V\left(F\right)\backslash G\left(F\right)}\underset{i}{\sum}\int_{ H_V\left(F\right)}f\left(g^{-1}hg\right)\left\langle \phi_{i},\omega_{V,\psi,\mu}\left(h\right)\phi_{i}\right\rangle dhdg
\]
\[
=\sum_{j=1}^{2^{m-1}}\int_{Z_{G}\left(F\right)H_V\left(F\right)\backslash G\left(F\right)}\underset{i}{\sum}\int_{H_{V,x}\left(F\right)\backslash H_V\left(F\right)}\int_{ H_{V,x}\left(F\right)}\left(^{g_jhg}f\right)_{x,\Omega_{x}}\left(h_{x}\right)\left\langle \phi_{i},\omega_{V,\psi,\mu}\left(h_{x}\right)\phi_{i}\right\rangle dh_{x}dhdg.
\]
Assume one moment that the exterior double integral and the sum above
are absolutely convergent. Then each summand of $J_{V}\left(f\right)$ is equal to
\begin{equation}\label{ssdescent}
\int_{G_{x_j}\left(F\right)\backslash G\left(F\right)}\int_{Z_{G_{x_j}}\left(F\right)H_{V,x_j}\left(F\right)\backslash G_{x_j}\left(F\right)}\underset{i}{\sum}\int_{ H_{V,x}\left(F\right)}\left(^{g_{x_j}g}f\right)_{x_j,\Omega_{x_j}}\left(h_{x_j}\right)\left\langle \phi_{i},\omega_{V,\psi,\mu}\left(h_{x_j}\right)\phi_{i}\right\rangle dh_{x_j}dg_{x_j}dg.   
\end{equation}
We introduce a function $\alpha$ on $G_{x}\left(F\right)\backslash G\left(F\right)$
as in \cite[Proposition 5.7.1]{BP20}. Up to translating $g$ by an
element in $G_{x}\left(F\right)$, we may assume that $\left(^{g}f\right){}_{x,\Omega_{x}}=\alpha\left(g\right)f_{x}$.
Thus 
\[
\int_{Z_{G}\left(F\right)H_{V,x}\left(F\right)\backslash G_{x}\left(F\right)}\underset{i}{\sum}\int_{ H_{V,x}\left(F\right)}\left(^{g}f\right)_{x,\Omega_{x}}\left(g_{x}^{-1}h_{x}g_{x}\right)\left\langle \phi_{i},\omega_{V,\psi,\mu}\left(h_{x}\right)\phi_{i}\right\rangle dh_{x}dg_{x}
\]
\[
=\alpha\left(g\right)\int_{Z_{G}\left(F\right)H_{V,x}\left(F\right)\backslash G_{x}\left(F\right)}\underset{i}{\sum}\int_{ H_{V,x}\left(F\right)}f_x\left(g_{x}^{-1}h_{x}g_{x}\right)\left\langle \phi_{i},\omega_{V,\psi,\mu}\left(h_{x}\right)\phi_{i}\right\rangle dh_{x}dg_{x}
=\alpha\left(g\right)\prod_{i=1}^mJ^{H_{i}}\left(f_x^{i}\right),
\]
where $J^{H_{i}}$ is the linear form in style of $J_{V}$ for the
twisted Gan-Gross-Prasad triple $\left(G_{i},H_{i},\omega_{H_{i},\psi,\mu}\right)$.

We set $g_j^{-1}H_xg_j=H_{1,j}\times\ldots \times H_{m,j}$ and $G_{i,j}=\text{Res}_{K/F}(H_{i,j})_K$. Since the function $\alpha$ is compactly supported,
the exterior double integral of (\ref{ssdescent}) is absolutely convergent. Moreover, as $\int_{G_{x}\left(F\right)\backslash G\left(F\right)}\alpha\left(g\right)dg=1$, it follows that  
\[
J_{V}\left(f\right)=\sum_{i=1}^{2^{m-1}}\prod_{i=1}^mJ^{H_{i,j}}\left(^{g_j}f_{x,i}\right).
\]
By the induction hypothesis, we have
$$
J_{V}\left(f\right)=\sum_{i=1}^{2^{m-1}}\prod_{i=1}^mJ^{H_{i,j}}\left(^{g_j}f_{x,i}\right)=\sum_{i=1}^{2^{m-1}}\prod_{i=1}^m m_{H_{i,j},\text{geom}}(^{g_j}\theta_{f_{x,i}}),
$$
where $m_{H_{i,j},\text{geom}}$ is the linear form $m_{V,\text{geom}}$ for the twisted GGP triple $(G_{i,j},H_{i,j},\omega_{H_{i,j},\psi,\mu})$. Proposition \ref{4.1}(2) and a direct computation give us 
$$
J_{V,\text{qc}}(\theta_f)=J_V(f)=m_{V,x,\text{geom}}(\theta_{f_x})=m_{V,\text{geom}}(\theta_f).
$$
\end{proof}

\subsection{Descent to non-identity central elements}\label{sec4.4}
Let $s$ be a non-identity element in $Z_{H_V}(F)\simeq E^1$. In this section, we show that the two linear forms $J_{V,\text{qc}}$ and $m_{V,\text{geom}}$ coincide near $s$. We first examine the local character expansion of the Weil representation near $s$.
\begin{proposition}\label{nonidentity}
    For any sufficiently small neighborhood $\omega$ of $0$ in $\mathfrak{u}(V)$ and any smooth function $f\in C^\infty(\omega)$, we have
    $$
    \text{tr}(\omega_{V,\psi,\mu}(f\circ \log)\circ \omega_{V,\psi,\mu}(s))=\frac{\mu(\det V)\mu(1-s^{-1})^{n}}{|s-1|^{n/2}_E}\sum_{\mathfrak{t}\in \mathcal{T}_{\text{ell}}(\mathfrak{h}_V)}\frac{\gamma_\psi(\mathfrak{t})}{|W(H_V,\mathfrak{t})|}
    $$
    $$
    \int_{\mathfrak{t}(F)}D^{H_V}(X)\int_{T(F)\backslash H_V(F)}{}^hf(X)d\Dot{h}dX.
    $$
\end{proposition}
\begin{proof}
    We first consider character expansions near non-identity elements of Weil representations of $E^1$ and $E^\times$. When $\dim V=1$, by Proposition \ref{4.3}, for $\omega$ sufficiently small, we have
    $$
     \text{tr}(\omega_{V,\psi,\mu}(f\circ \log)\circ \omega_{V,\psi,\mu}(s))=\frac{\mu(\det V)\mu(1-s^{-1})\gamma_\psi(2Nm_{E/F})}{|s-1|^{1/2}_E}
    \int_{E^0}f(X)dX.
    $$
    Let $\Omega_\mu$ be the Weil representation of $\text{GL}_1(E)\simeq E^\times$. In this setting, since $\Omega_\mu=\bigoplus_{\chi \in \widehat{E^\times}}\chi$, it follows that 
    $$
    \text{tr}(\Omega_\mu(F))=F(1),
    $$
    for any $F\in C^\infty_c(E^\times)$. Thus, when $F$ is supported near $s$, we have $\text{tr}(\Omega_\mu(F))$ vanishes.

    We now consider the general case. Similar to Proposition \ref{identity}, let $\ell$ be a langrangian of the symplectic space $\text{Res}_{E/F}V$. Using the splitting defined in \cite[Section 11.4.4]{GKT25}, we can choose $\omega$ sufficiently small so that $\mu((\exp{X})s)=((\exp X)s,\mu(s)^n\gamma_\psi(1)^{n}\gamma_\psi(A_{(\exp{X})\ell,\ell})\cdot M_\ell^{\text{Sch}}[\exp X])$, for any $X\in \omega$. As in the proof of Proposition \ref{identity}, for a smooth function $f$ on $\omega$, we have
    $$
    \text{tr}(\omega_{V,\psi,\mu}(f\circ \log)\circ\omega_{V,\psi,\mu})=\mu(s)^{n}\gamma_\psi(1)^{n}\int_{V/\ell}\int_{\mathfrak{h}_V(F)}f(X) \gamma_\psi(A_{(\exp{X})\ell,\ell})\psi\left(\frac{Q_{(\exp X)s\ell,\ell}(x,x)}{2}\right)dXdx.
    $$
    Let $h$ be a compactly-supported smooth function on $V/\ell$ whose measure is positive and $h(0)=1$. We set $h_k(\cdot)=h(k\cdot)$, for $k\in F$. Then 
    $$
    \text{tr}(\omega_{V,\psi,\mu}(f\circ \log)\circ\omega_{V,\psi,\mu}(s))=\lim_{k\rightarrow 0}\text{tr}(h_k\cdot\omega_{V,\psi,\mu}(f\circ \log)\circ\omega_{V,\psi,\mu}).
    $$
    $$
    =\mu(s)^{n}\gamma_\psi(1)^{n}\lim_{k\rightarrow 0} \int_{\mathfrak{h}_V(F)}f(X) \gamma_\psi(A_{(\exp{X})\ell,\ell})\int_{V/\ell}h_k(x)\psi\left(\frac{Q_{(\exp X)s\ell,\ell}(x,x)}{2}\right)dxdX.
    $$
    By Weyl integration formula, we have
    $$
    \text{tr}(h_t\cdot\omega_{V,\psi,\mu}(f\circ \log)\circ\omega_{V,\psi,\mu}(s))=\mu(s)^{n}\gamma_\psi(1)^{n}\sum_{T\in \mathcal{T}(H_V)}|W(H_V,T)|^{-1}
    $$
    $$\int_{\mathfrak{t}(F)}f_T(X) \gamma_\psi(A_{(\exp{X})\ell,\ell})\int_{V/\ell}h_k(x)\psi\left(\frac{Q_{(\exp X)s\ell,\ell}(x,x)}{2}\right)dxdX,
    $$
    where $\mathcal{T}(H_V)$ is the set containing representatives of $H_V(F)$-conjugacy classes of maximal tori in $H_V$, and $\mathfrak{t}$ is the Lie algebra of $T$ and $f_T(X)=D^{H_V}(X)\int_{T(F)\backslash H_V(F)}f(h^{-1}Xh)d\Dot{h}$ as a function on $\mathfrak{t}(F)$. Therefore
    $$
    \text{tr}(\omega_{V,\psi,\mu}(f\circ \log)\circ\omega_{V,\psi,\mu}(s))=\sum_{T\in \mathcal{T}(H_V)}|W(H_V,T)|^{-1}\text{tr}_T(\omega_{V,\psi,\mu}(f_T\circ \log)\circ\omega_{V,\psi,\mu}(s)).
    $$
    If $T$ is not elliptic, then $T$ contains at least one copy of $\text{GL}_1$, thus $\text{tr}_T(\omega_{V,\psi,\mu}(f_T\circ \log)\circ\omega_{V,\psi,\mu}(s))$ vanishes. When $T$ is elliptic, we have
    $$
    \text{tr}_T(\omega_{V,\psi,\mu}(f_T\circ \log)\circ\omega_{V,\psi,\mu}(s))=\frac{\mu(\det V)\mu(1-s^{-1})^{n}\gamma_\psi(\mathfrak{t})}{|s-1|^{n/2}_E}
    \int_{\mathfrak{t}(F)}f_T(X)dX,
    $$
    which is to say
    $$
    \text{tr}(\omega_{V,\psi,\mu}(f\circ \log)\circ \omega_{V,\psi,\mu}(s))=\frac{\mu(\det V)\mu(1-s^{-1})^{n}}{|s-1|^{n/2}_E}\sum_{\mathfrak{t}\in \mathcal{T}_{\text{ell}}(\mathfrak{h}_V)}\frac{\gamma_\psi(\mathfrak{t})}{|W(H_V,\mathfrak{t})|}
    $$
    $$
    \int_{\mathfrak{t}(F)}D^{H_V}(X)\int_{T(F)\backslash H_V(F)}{}^hf(X)d\Dot{h}dX.
    $$
\end{proof}
Let $\omega$ be a $G$-excellent neighborhood of $0$ in $\mathfrak{g}(F)$ and $\Omega=\exp(\omega)s$. Let $\theta$ be a quasi-character of $G(F)$ supported in $\Omega$ and $\theta_{s,\omega}$ be a quasi-character on $\omega$ such that $D^G(X)^{1/2}\theta_{s,\omega}(X)=D^G(x)^{1/2}\theta(s\cdot\exp(X))$, for all $X\in \omega$. When $\omega$ is sufficiently small, we have 
$$
m_{V,\text{geom}}(\theta)=\frac{1}{2}c_{\theta}\left(1\right)+\mu\left(\det V\right)
\underset{T\in{\mathcal{T}}_{\text{ell}}\left(H_V\right)}{\sum}\frac{\gamma_\psi(T)}{\left|W\left(H_V,T\right)\right|}
\underset{s\rightarrow0^{+}}{\lim}\int_{T\left(F\right)}D^{G}\left(x\right)^{1/2}c_{\theta}\left(x\right)\frac{\mu\left(\det\left(1-x^{-1}\right)\right)}{\left|\det\left(1-x\right)\right|_{E}^{1/2-s}}dx
$$
$$
=\frac{\mu\left(\det V\right)\mu(1-s^{-1})^n}{\left|s-1\right|_{E}^{n/2}}
\underset{\mathfrak{t}\in{\mathcal{T}}_{\text{ell}}\left(\mathfrak{h}_V\right)}{\sum}\frac{\gamma_\psi(\mathfrak{t})}{\left|W\left(H_V,\mathfrak{t}\right)\right|}
\int_{\mathfrak{t}\left(F\right)}D^{G}\left(X\right)^{1/2}c_{\theta_{s,\omega}}\left(X\right)dx.
$$
Thus, there remains to prove the following proposition.
\begin{proposition}\label{central}
    We keep the same notations as above. When $\omega$ is sufficiently small, for any $\theta\in \text{QC}(\Omega)$, we have 
    $$
    J_{V,\text{qc}}(\theta)=\frac{\mu\left(\det V\right)\mu(1-s^{-1})^n}{\left|s-1\right|_{E}^{n/2}}
\underset{\mathfrak{t}\in{\mathcal{T}}_{\text{ell}}\left(\mathfrak{h}_V\right)}{\sum}\frac{\gamma_\psi(\mathfrak{t})}{\left|W\left(H_V,\mathfrak{t}\right)\right|}
\int_{\mathfrak{t}\left(F\right)}D^{G}\left(X\right)^{1/2}c_{\theta_{s,\omega}}\left(X\right)dx.
    $$
\end{proposition}
\begin{proof}
    Let $f_{s,\omega}\in \mathcal{C}_{\text{scusp}}(\omega)$ such that $\theta_{s,\omega}=\theta_{f_{s,\omega}}$. Let $f\in\mathcal{C}_{scusp}(\Omega)$ be a lift of $f_{s,\omega}$ in the sense of Proposition \ref{2.8}. Then $\theta_{f}=\theta$. Observe
    $$
    J_{V,\text{qc}}(\theta)=J_V(f)=\int_{Z_G(F)H_V(F)\backslash G(F)}
    \sum_i\int_{H_V(F)}{}^{x}f(h)\langle \phi_i,\omega_{V,\psi,\mu}(h)\phi_i\rangle dhdx
    $$
    $$
    =\int_{Z_G(F)H_V(F)\backslash G(F)}
    \sum_i\int_{\mathfrak{h}_V(F)}{}^{x}f_{s,\omega}(X)\langle \phi_i,\omega_{V,\psi,\mu}(s\exp(X))\phi_i\rangle dXdx.
    $$
    We fix a sequence $(\kappa_N)_{N\geq 1}$ of functions $\kappa_N:Z_G(F)H_V(F)\backslash G(F)\rightarrow \{0,1\}$ satisfying the two following conditions:
    \begin{enumerate}
        \item There exists $c_1,c_2>0$ such that for all $x\in Z_G(F)H_V(F)\backslash G(F)$ and $N\geq 1$, we have
        $$
        \sigma_{Z_GH_V\backslash G}(x)\leq c_1N \Rightarrow \kappa_N(x)=1
        $$
        and
        $$
        \kappa_N(x)=1 \Rightarrow \sigma_{Z_GH_V\backslash G}(x)\leq c_2N
        $$
        \item There exists an open-compact subgroup $K^\prime \subseteq G(F)$ such that the function $\kappa_N$ is right-invariant by $K^\prime$ for all $N\geq 1$.
    \end{enumerate}
    We set 
    $$
    J_{V,N}(f_{s,\omega})=\int_{Z_G(F)H_V(F)\backslash G(F)}\kappa_N(x)
    \sum_i\int_{\mathfrak{h}_V(F)}{}^{x}f_{s,\omega}(X)\langle \phi_i,\omega_{V,\psi,\mu}(s\exp(X))\phi_i\rangle dXdx,
    $$
    for $N\geq 1$. Then
    $$
    J_V(f)=\lim_{N\rightarrow \infty}J_{V,N}(f_{s,\omega}).
    $$
    When $\omega$ is sufficiently small, by Proposition \ref{nonidentity}, we have
    $$
    \sum_i\int_{\mathfrak{h}_V(F)}{}^{x}f_{s,\omega}(X)\langle \phi_i,\omega_{V,\psi,\mu}(s\exp(X))\phi_i\rangle dX=\frac{\mu\left(\det V\right)\mu(1-s^{-1})^n}{\left|s-1\right|_{E}^{n/2}}\sum_{\mathfrak{t}\in \mathcal{T}_{\text{ell}}(\mathfrak{h}_V)}\frac{\gamma_\psi(\mathfrak{t})}{|W(H_V,\mathfrak{t})|}
    $$
    $$
    \int_{\mathfrak{t}(F)}D^{H_V}(X)\int_{T(F)\backslash H_V(F)}{}^{hx}f_{s,\omega}(X)d\Dot{h}dX.
    $$
    For each maximal torus $T$ of $H_V$, we set $T^G=\text{Cent}_{G}(T)$ and
    $$
    \kappa_{N,T}(x)=\int_{Z_G(F)T(F)\backslash T^G(F)}\kappa_{N}(ax)da.
    $$
    This gives us
    $$
    J_{V,N}(f_{s,\omega})=\frac{\mu\left(\det V\right)\mu(1-s^{-1})^n}{\left|s-1\right|_{E}^{n/2}}\sum_{\mathfrak{t}\in \mathcal{T}_{\text{ell}}(\mathfrak{h}_V)}\frac{\gamma_\psi(\mathfrak{t})}{|W(H_V,\mathfrak{t})|}\int_{\mathfrak{t}(F)}D^{H_V}(X)
    $$
    $$
    \int_{T^G(F)\backslash G(F)}f_{s,\omega}(x^{-1}Xx)\kappa_{N,T}(x)dxdX.
    $$
    By \cite[Theorem 4.1.1]{BP18} and a descent-to-Lie-algebra statement, we have
    $$
    J_V(f)=\lim_{N\rightarrow \infty}J_{V,N}(f_{s,\omega})=\frac{\mu\left(\det V\right)\mu(1-s^{-1})^n}{\left|s-1\right|_{E}^{n/2}}\sum_{\mathfrak{t}\in \mathcal{T}_{\text{ell}}(\mathfrak{h}_V)}\frac{\gamma_\psi(\mathfrak{t})}{|W(H_V,\mathfrak{t})|}\int_{\mathfrak{t}(F)}D^{H_V}(X)
    \theta_{f_{s,\omega}}(X)dX
    $$
    as desired.
\end{proof}

\section{The spectral expansion of $J_V^{\text{Lie}}$ and strong multiplicity one property}\label{sec5} 
In the previous section, we have showed that
$$
J_{V,\text{qc}}(\theta)=m_{V,\text{geom}}(\theta)+\sum_{\mathcal{O}\in \text{Nil}(\mathfrak{g})}c_{\mathcal{O}}\cdot c_{\theta,\mathcal{O}}(1),
$$
for any $\theta \in \text{QC}_c(G(F))$. The main result in this section is to show $c_{\mathcal{O}}=0$, for all nilpotent orbit $\mathcal{O}$, and thus finish the proof of Theorem \ref{maintheorem1} and Theorem \ref{maintheorem}(i).
\subsection{An infinitesimal trace formula $J^{\text{Lie}}_V$}\label{sec5.1}

Let $h_V:V\times V\rightarrow E$ be the skew-hermitian form of $V$. We define the following linear form on ${\mathcal{C}}_{\text{scusp}}\left(\mathfrak{g}(F)\right)$ as follows 
$$
J_V^{\text{Lie}}\left(f\right)=\int_{Z_{G}\left(F\right)H_V\left(F\right)\backslash G\left(F\right)}\int_{V(F)}\int_{\mathfrak{h}_V^{\perp}\left(F\right)}
\hat{f}\left(x^{-1}Xx+\Phi\left(x^{-1}\cdot v\right)\right)dXdvdx,
$$
where $\Phi$ is the moment map of the $H_V$-symplectic space $\text{Res}_{E/F}V$ given by 
\[
\begin{array}{ccccc}
\Phi & : & V & \longrightarrow & \mathfrak{h}_V^{*}\\
 &  & v & \mapsto & \left(X\mapsto h_V(v,Xv) \right).
\end{array}
\]
We can think about $J_V^{\text{Lie}}$ as an infinitesimal variant of the trace formula $J_V$.
\begin{proposition}
For any $f\in{\mathcal{C}}_{\text{scusp}}\left(\mathfrak{g}(F)\right)$, the linear form $J_V^{\text{Lie}}\left(f\right)$ is convergent.
\end{proposition}
\begin{proof}
    It suffices to show that for all $d>0$, we have   
    $$
    \int_{V(F)}\int_{\mathfrak{h}_V^{\perp}\left(F\right)}
    \hat{f}\left(x^{-1}Xx+\Phi\left(x^{-1}\cdot v\right)\right)dXdv 
    \ll \Xi^{H_V\backslash G}(x)^2\sigma_{H_V\backslash G}(x)^{-d},
    $$
    for any $f\in \mathcal{C}_{\text{scusp}}(\mathfrak{g}(F))$ and $x\in Z_G(F)H_V(F)\backslash G(F)$. Let $\omega \subseteq \mathfrak{g}(F)$ be a $G$-excellent neighborhood of $0$ such that $\omega_{H_V}=\omega \cap \mathfrak{h}_V(F)$ satisfies Proposition \ref{identity}. Thus, for any $f\in \mathcal{C}_{\text{scusp}}(\omega)$, it follows that
    $$
    \int_{V(F)}\int_{\mathfrak{h}_V^{\perp}\left(F\right)}
    \hat{f}\left(x^{-1}Xx+\Phi\left(x^{-1}\cdot v\right)\right)dXdv = K_V(f\circ \log,x).
    $$
    By \cite[Theorem 5.2]{Le25}, for any $d>0$, we have 
    $$
    K_V(f\circ \log,x) \ll \Xi^{H_V\backslash G}(x)^2\sigma_{H_V\backslash G}(x)^{-d}.
    $$
    This gives us 
    \begin{equation}\label{estimate}
     \int_{V(F)}\int_{\mathfrak{h}_V^{\perp}\left(F\right)}
    \hat{f}\left(x^{-1}Xx+\Phi\left(x^{-1}\cdot v\right)\right)dXdv \ll \Xi^{H_V\backslash G}(x)^2\sigma_{H_V\backslash G}(x)^{-d},   
    \end{equation}
    for any $f\in \mathcal{C}_{\text{scusp}}(\omega)$ and $x\in Z_G(F)H_V(F)\backslash G(F)$. Let $\lambda \in F^{\times 2}$ and we set $f_\lambda(X)=f(\lambda^{-1}X)$ for any $X\in \mathfrak{g}(F)$. Since 
    $$
    \int_{V(F)}\int_{\mathfrak{h}_V^{\perp}\left(F\right)}
    \hat{f}_\lambda\left(x^{-1}Xx+\Phi\left(x^{-1}\cdot v\right)\right)dXdv
    $$
    $$
    =|\lambda|^{n^2-n}\int_{V(F)}\int_{\mathfrak{h}_V^{\perp}\left(F\right)}
    \hat{f}\left(x^{-1}Xx+\Phi\left(x^{-1}\cdot v\right)\right)dXdv,
    $$
    the inequality (\ref{estimate}) holds for all $f\in \mathcal{C}_{\text{scusp}}(\mathfrak{g}(F))$. Therefore, the linear form $J_V^{\text{Lie}}$ is convergent.
\end{proof}

\subsection{The space $\mathfrak{h}_V^{\perp}\oplus \Phi(V)$ and characteristic polynomials}\label{sec5.2}

We set $\Sigma_V =\mathfrak{h}_V^{\perp}\oplus \Phi(V)$
as a subvariety of $\mathfrak{g}$. In this subsection, we study the sets of $H_V$-conjugacy classes and $G$-conjugacy classes of $\Sigma_V$. 

To be more precise, we expect the $H_V$-action is free on an open subvariety of $\Sigma_V$ characterized by their characteristic polynomials. We have the following natural map via the inclusion and $\Phi$, which is also denoted by $\Phi$
\[
\Phi:\mathfrak{h}_V^{\perp}\oplus V\longrightarrow\mathfrak{g}.
\]
Let $X=X_{V}+\Phi(v)\in\Sigma_V$, where $X_{V}\in\mathfrak{h}_V^{\perp}$ and $v\in V$. Let $P_{X_{V}}$ be the characteristic polynomial of $X_{V}$ acting on $V_{\bar{F}}$.
Then $P_{X_{V}}$ is an element of $\bar{E}\left[T\right]$. We denote by $D$ the $\bar{E}$-linear endomorphism on $\bar{E}\left[T\right]$ given by $D\left(T^{i+1}\right)=T^{i}$, for $i\geq0$ and $D\left(1\right)=0$.
\begin{proposition}\label{5.2}
We have 
\[
P_{X}\left(T\right)=P_{X_{V}}\left(T\right)-\left[\sum_{i=0}^{n-1} h_V(v,X_{V}^iv)\cdot D^{i+1}\left(P_{X_{V}}\left(T\right)\right)\right].
\]
\end{proposition}

\begin{proof}
The computation is direct but tedious.
\end{proof}
From the above computation, we can deduce the following corollary.
\begin{corollary}\label{5.3}
The $H_V$-invariant polynomial functions on $\Sigma_V$
\[
\left(X_{V},v\right)\in\mathfrak{h}_V^{\perp}\oplus V\mapsto h_V(v,X_{V}^{j}v) 
\]
extend to $G$-invariant polynomials functions on $\mathfrak{g}$ defined over $F$.
\end{corollary}

In particular, the polynomial function
$$
\left(X_{V},v\right)\in\mathfrak{h}_V^{\perp}\oplus V\mapsto \det(h_V(X_V^{i}v,X_{V}^{j}v))_{0\leq i,j \leq n-1}\in \bar{F}
$$
extends to a $G$-invariant polynomial function on $\mathfrak{g}$ defined over $F$. We denote such extension by $Q_0$. We set $d^G(X)=\det(1-\text{Ad}(X))_{|\mathfrak{g}/\mathfrak{g}_X}$ for $X\in \mathfrak{g}_{\text{reg}}$. Let $Q=Q_0d^G$ and $\Sigma_V^{\prime}$ be the nonvanishing locus of $Q$ in $\Sigma_V$. The subvariety $\Sigma_V^{\prime}$ is characterized by the following
property.
\begin{proposition}\label{5.4}
For any $X=X_{V}+\Phi(v)\in\Sigma_V$, we have $X\in \Sigma_V^\prime$ if and only if $X\in \mathfrak{g}_{\text{reg}}$ and the set 
\[
\left\{ v,X_{V}v,\ldots,X_{V}^{n-1}v\right\} 
\]
generates $V_{\bar{F}}$ as an $\bar{E}$-module.
\end{proposition}
\begin{proof}
    Let $X=X_V+\Phi(v)$, where $X_V \in \mathfrak{h}_V^\perp$ and $v\in V$. The above proposition follows from the facts that $Q_0(X)\neq 0$ if and only if the set
    $$
    \{v,X_Vv,\ldots,X_V^{n-1}v\}
    $$
    generates $V_{\bar{F}}$ and $d^G(X)\neq 0$ if and only if $X\in \mathfrak{g}_{\text{reg}}$.
\end{proof}
The above proposition gives us information about conjugacy classes
in $\Sigma^{\prime}$.
\begin{proposition}\label{5.5}
The action by conjugation of $H_V$ on $\Sigma^{\prime}$ is free and two elements in $\Sigma_V^{\prime}$ whose image via $\Phi$ are $G$-conjugate
if and only if they are $H_V$-conjugate.
\end{proposition}
\begin{proof}
    Let $X=X_V+\Phi(v)$ and $X^\prime=X_V^\prime+\Phi(v^\prime)$ be two elements in $\Sigma_V^\prime$ such that $P_X=P_{X^\prime}$. By Proposition \ref{5.2}, it follows that $P_{X_V}=P_{X_V^\prime}$ and 
    $$
    h_V(v,X_V^iv)=h_V(v^\prime,(X_V^{\prime})^iv^\prime),
    $$
    for any $i=\overline{0,n-1}$. By definition of $\Sigma_V^\prime$, $\{v,X_Vv,\ldots,X_V^{n-1}v\}$ and $\{v^\prime,X^\prime_Vv^\prime,\ldots,X_V^{\prime,n-1}v^\prime\}$ are two basis for $V$. Let $g$ be the unique $\bar{E}$-linear automorphism of $V_{\bar{F}}$ mapping $X^i_Vv\mapsto X^{\prime,i}_Vv^\prime$, for any $i=\overline{0,n-1}$. Since $h_V(v,X_V^iv)=h_V(v^\prime,(X_V^{\prime})^iv^\prime)$, it follows that $g\in H_V$. Moreover, we can see that $gX_Vg^\prime=X_V^\prime$. Therefore, we have $gXg^{-1}=X^\prime$. 
    
    Conversely, if $X^\prime=gXg^{-1}$ for some element $g\in H_V$, then $P_X=P_{X^\prime}$. Thus, a similar to above gives us $g$ is the unique element in $H_V$ satisfying this property.
\end{proof}
\begin{corollary}\label{5.6}
We have 
\[
\sigma_{G}\left(t\right)\ll\sigma_{H\backslash G}\left(t\right)\sigma_{\Sigma^{\prime}}\left(X\right),
\]
for any $X\in\Sigma^{\prime}$ and $t\in G_{X}$.
\end{corollary}
\begin{proof}
    The proof of \cite[Proposition 10.5.2]{BP20} works verbatim.
\end{proof}

\subsection{The quotient $\Sigma_V^{\prime}\left(F\right)/H_V\left(F\right)$}\label{sec5.3}

We prove a genericity property for the Borel subalgebras
intersecting $\Sigma_V^\prime$.
\begin{proposition}\label{5.7}
Let $X=X_V+\Phi(v)\in\Sigma^{\prime}$, where $X_V\in \mathfrak{h}_V$ and $v\in V$. Let $\mathfrak{b}=\mathfrak{t}\oplus\mathfrak{u}$
be a Borel algebra of $\mathfrak{g}$ defined over $\bar{F}$ containing $X$. Then 
\[
\mathfrak{g}=\mathfrak{h}_V^\perp\oplus d\Phi_{v}(V)\oplus\mathfrak{u}.
\]
\end{proposition}
\begin{proof}
    Since $\dim(\mathfrak{h}_V^\perp)+\dim(d\Phi_{v}(V))+\dim(\mathfrak{u})=\dim(\mathfrak{g})$ and $\mathfrak{h}_V^\perp\cap d\Phi_{v}(V)=0$, it suffices to show 
    $$
    (\mathfrak{h}_V^\perp\oplus d\Phi_{v}(V)) \cap \mathfrak{u}=0,
    $$
    whose proof can be adapted from \cite[Proposition 10.6.1]{BP20} without any difficulties.
\end{proof}
We denote by $\mathfrak{g}^{\prime}$ be the nonvanishing locus of $Q$ in $\mathfrak{g}$. Let $\mathfrak{g}^{\prime}/G$ be the geometric quotient of $\mathfrak{g}^{\prime}$ by $G$-adjoint action. By Proposition \ref{5.5}, the map $\Sigma^\prime \rightarrow \mathfrak{g}^\prime/G$ factors through $\Sigma_V^{\prime}/H_V$, hence gives us the following morphism
\[
\pi:\Sigma_V^{\prime}/H_V\longrightarrow\mathfrak{g}^{\prime}/G.
\]
We study the $F$-analytic counterpart 
$$
\pi_{F}:\Sigma_V^{\prime}(F)/H_V(F)\longrightarrow\mathfrak{g}^{\prime}(F)/G(F)
$$
of the above map. Since $H_V(F)$ acts freely on $\Sigma_V^\prime(F)$, we set $\mu_{\Sigma_V^\prime/H_V}$ to be the quotient measure associated to $\mu_{\Sigma_V}$ and $\mu_{H_V}$. This measure is characterized by the following equality 
\[
\int_{\Sigma_V\left(F\right)}\phi\left(X\right)d\mu_{\Sigma}\left(X\right)=\int_{\Sigma_V^{\prime}\left(F\right)/H_V\left(F\right)}\int_{H_V\left(F\right)}\phi\left(h^{-1}Xh\right)dhd\mu_{\Sigma_V^{\prime}/H_V}\left(X\right)
\]
for all $\phi\in C_{c}\left(\Sigma\left(F\right)\right)$. We denote
by $dX$ the measure on $\mathfrak{g}^{\prime}\left(F\right)/G\left(F\right)$
inherited from one on $\mathfrak{g}_{\text{reg}}\left(F\right)/G\left(F\right)=\Gamma_{\text{reg}}(\mathfrak{g})$.
\begin{proposition}\label{5.8}
\begin{enumerate}
\item $\pi$ is an isomorphism of algebraic varieties and $\pi_{F}$ is
an open embedding of $F$-analytic spaces.
\item $\pi_{F}$ sends the measure $\mu_{\Sigma_V^{\prime}/H_V}\left(X\right)$
to $D^{G}\left(X\right)^{1/2}dX$.
\item The natural projection $p:\Sigma_V^{\prime}\rightarrow\Sigma_V^{\prime}/H_V$
has the norm descent property.
\end{enumerate}
\end{proposition}
\begin{proof}
    \begin{enumerate}
        \item By Proposition \ref{5.5}, it follows that $\pi$ and $\pi_F$ are injective. Moreover, by Proposition \ref{5.2}, we have $\pi$ is surjective and thus bijective. To prove $\pi_F$ is an open embedding, it suffices to show that $\pi$ is a local isomorphism. Let $X=X_V+\Phi(v)\in \Sigma_V^\prime$, where $X_V\in \mathfrak{h}_V^\perp$ and $v\in V$. We need to prove $d\pi_X$ is an isomorphism. Observe
        $$
        T_X(\Sigma^\prime_V/H_V)=(\mathfrak{h}_V^\perp+d\Phi_v(V))/\text{ad}(X)(\mathfrak{h}_V)
        $$
        and 
        $$
        T_X(\mathfrak{g}^\prime/G)=\mathfrak{g}/\text{ad}(X)(\mathfrak{g}).
        $$
        The differential $d\pi_X$ is the natural inclusion of $(\mathfrak{h}_V^\perp+d\Phi_v(V))/\text{ad}(X)(\mathfrak{h}_V)$ in $\mathfrak{g}/\text{ad}(X)(\mathfrak{g})$. We choose a Borel subalgebra $\mathfrak{b}$ of $\mathfrak{g}$ containing $X$. Let $\mathfrak{u}$ be its nilpotent radical. By Proposition \ref{5.7}, it follows that $\mathfrak{g}=\mathfrak{h}_V^\perp\oplus d\Phi_{v}(V)\oplus\mathfrak{u}$. Since $\mathfrak{u}=\text{ad}(X)(\mathfrak{b})\subset \text{ad}(X)(\mathfrak{g})$, we can see that $d\pi_X$ is surjective. Moreover, since 
        $$
        \text{ad}(X)(\mathfrak{g})=\text{ad}(X)(\mathfrak{h}_V)+\text{ad}(X)(\mathfrak{b})=\text{ad}(X)(\mathfrak{h}_V)+\mathfrak{u}
        $$
        and $(\mathfrak{h}_V^\perp\oplus d\Phi_v(V))\cap \mathfrak{u}=0$, it follows that $(\mathfrak{h}_V^\perp\oplus d\Phi_v(V))\cap \text{ad}(X)(\mathfrak{g})\subseteq \text{ad}(X)(\mathfrak{h}_V)$. Therefore, $d\pi_X$ is injective, which is to say $\pi$ is a local isomorphism.

        \item Let $X\in \Sigma_V^\prime(F)$. We denote $\mathfrak{g}_X=\ker (\text{ad}(X))$ and $\mathfrak{g}^X=\text{im} (\text{ad}(X))$. We have the following isomorphism
        $$
        d\pi_{F,X}:(\mathfrak{h}_V^\perp(F)+d\Phi_v(V)(F))/\text{ad}(X)(\mathfrak{h}_V)  \overset{\sim}{\longrightarrow}   \mathfrak{g}(F)/\mathfrak{g}^X(F).
        $$
        Let $F(X)\in \mathbb{R}^+$ such that
        $$
        (d\pi_{F,X})_*(\mu_{\mathfrak{h}_V}^\perp\otimes\mu_V/\text{ad}(X)_*\mu_{\mathfrak{h}})=F(X)\mu_{\mathfrak{g}}/\mu_{\mathfrak{g}^X}.
        $$
        It suffices to show $F(X)=D^G(X)^{1/2}$. We choose a Borel subalgebra $\mathfrak{b}$ of $\mathfrak{g}$ containing $X$ and let $\mathfrak{u}$ be its nilpotent radical. We choose the measure $\mu_{\mathfrak{u}}$ on $\mathfrak{u}$ such that
        $$
        \mu_{\mathfrak{g}}=\mu_{\mathfrak{h}_V}^\perp\otimes\mu_V\otimes \mu_{\mathfrak{u}}.
        $$
        This gives us
        $$
        \mu_{\mathfrak{g}}=(\mu_{\mathfrak{h}_V}^\perp\otimes\mu_V)^\perp\otimes \mu_{\mathfrak{u}}^\perp=(\mu_{\mathfrak{h}_V}^\perp\otimes\mu_V)^\perp\otimes \mu_{\mathfrak{g}_X}\otimes \mu_{\mathfrak{u}}.
        $$
        Let $T\in \text{End}(\mathfrak{g})$ which is equal to $\text{ad}(X)$ on $(\mathfrak{h}_V\cap (d\Phi_v(V))^\perp)\oplus \mathfrak{u}$ and $\text{Id}$ on $\mathfrak{g}_X$. Then we have $D^G(X)=|\det(T)|$. Observe
        $$
        D^G(X)\mu_{\mathfrak{g}}=T_*\mu_{\mathfrak{g}}=\text{ad}(X)_*(\mu_{\mathfrak{h}_V}^\perp\otimes\mu_V)^\perp\otimes \mu_{\mathfrak{g}_X} \otimes \text{ad}(X)_*\mu_{\mathfrak{u}}
        $$
        $$
        =D^G(X)^{1/2}(\text{ad}(X)_*(\mu_{\mathfrak{h}_V}^\perp\otimes\mu_V)^\perp\otimes \mu_{\mathfrak{g}_X} \otimes \mu_{\mathfrak{u}}).
        $$
        Since $\mu_{\mathfrak{g}}=\mu_{\mathfrak{g_X}}\otimes \mu_{\mathfrak{g}^X}$, we obtain 
        $$
        \mu_{\mathfrak{g}^X}=D^G(X)^{1/2}(\text{ad}(X)_*(\mu_{\mathfrak{h}_V}^\perp\otimes\mu_V)^\perp\otimes \mu_{\mathfrak{u}}).
        $$
        Therefore, we deduce that $F(X)=D^G(X)^{1/2}$ as desired.
        \item Part (iii) follows from the proof of \cite[Proposition 10.7.1(iii)]{BP20} without any difficulties.
    \end{enumerate}
\end{proof}

\subsection{A spectral expansion of $J_V^{\text{Lie}}$}\label{sec5.4}

Let $\Gamma\left(\Sigma_V\right)$ be the subset of $\Gamma\left(\mathfrak{g}\right)$
consisting of the conjugacy classes of the semisimple parts of the
elements in $\Sigma_V\left(F\right)$. For choices of measure of $\Gamma(\Sigma_V)$, we take the restriction of the measure on $\Gamma\left(\mathfrak{g}\right)$. Let ${\mathcal{T}}\left(G\right)$ be a set of representatives of $G\left(F\right)$-conjugacy classes of maximal tori in $G$. For each $T\in {\mathcal{T}}\left(G\right)$, we denote by $\mathfrak{t}\left(F\right)_{\Sigma_V}$ the subset of elements $X\in\mathfrak{t}\left(F\right)$ whose conjugacy class belongs to $\Gamma\left(\Sigma_V\right)$. We have 
\[
\int_{\Gamma\left(\Sigma_V\right)}\phi\left(X\right)dX=\underset{T\in{\mathcal{T}}\left(G\right)}{\sum}\left|W\left(G,T\right)\right|^{-1}\int_{\mathfrak{t}\left(F\right)_{\Sigma_V}}\phi\left(X\right)dX
\]
for all $\phi\in C_{c}^{\infty}\left(\Gamma\left(\Sigma_V\right)\right)$.
The following theorem gives us a spectral expansion of $J_V^{\text{Lie}}$.
\begin{theorem}\label{Liespectral}
For any $f\in{\mathcal{C}}_{\text{scusp}}\left(\mathfrak{g}\left(F\right)\right)$, we have 
\[
J_V^{\text{Lie}}\left(f\right)=\int_{\Gamma\left(\Sigma_V\right)}D^{G}\left(X\right)^{1/2}\hat{\theta}_{f}\left(X\right)dX.
\]
\end{theorem}
\begin{proof}
    Once we have Proposition \ref{5.8}, the proof follows one of \cite[Theorem 10.8.1]{BP20} verbatim.
\end{proof}
\subsection{Comparison near the identity element}\label{sec5.5}
Let $\omega$ be a $G$-excellent neighborhood near $0$ in $\mathfrak{g}(F)$. We set $\Omega=\exp(\omega)$. Recall that for any quasi-character $\theta\in QC\left(\mathfrak{g}\left(F\right)\right)$
and $\lambda\in F^{\times}$, we denote by $\theta_{\lambda}$ the
quasi-character given by $\theta_{\lambda}=\theta\left(\lambda^{-1}X\right)$ for each $X\in\mathfrak{g}_{\text{reg}}\left(F\right)$.
\begin{proposition}\label{5.10}
Assume the induction hypothesis. Then
\begin{enumerate}
\item If $\omega$ is sufficiently small, then for any $f\in \mathcal{C}_{\text{scusp}}(\Omega)$, we have 
    \[
    J_V\left(f\right)=J_V^{\text{Lie}}\left(f_\omega\right).
    \]
\item There exists a unique continuous linear form $J^{\text{Lie}}_{V,\text{qc}}$ such that 
$$
J^{\text{Lie}}_{V}(f)=J^{\text{Lie}}_{V,\text{qc}}(\theta_f),
$$
for all $f\in \mathcal{C}_{\text{scusp}}(\mathfrak{g}(F))$. Moreover, we have
$$
J^{\text{Lie}}_{V,\text{qc}}(\theta_\lambda)=|\lambda|^{\delta(G)/2}J^{\text{Lie}}_{V_\lambda,\text{qc}}(\theta),
$$
for any $\theta\in \text{QC}_{c}(\mathfrak{g}(F))$ and $\lambda\in F^\times$.
\item Let $\theta\in \text{QC}_c(\mathfrak{g}(F))$ supported outside $0$. Then
$$
J^{\text{Lie}}_{V,\text{qc}}(\theta)=m_{V,\text{geom}}^{\text{Lie}}(\theta).
$$
\end{enumerate}
\end{proposition}

\begin{proof}
\begin{enumerate}
    \item Let $\omega_{\mathfrak{h}_V}=\omega\cap\mathfrak{h}_V\left(F\right)\subseteq\mathfrak{h}_V\left(F\right)$.
We can see that $\omega_{\mathfrak{h}_V}$ is an $H_V$-excellent open
neighborhood of $0$. We have 
\[
\int_{H_V\left(F\right)}f\left(h\right)\left\langle \phi_{i},\omega_{V,\psi,\mu}\left(h\right)\phi_{i}\right\rangle dh
=\int_{\omega_{\mathfrak{h}_V}}j^{H_V}\left(X\right)f\left(e^{X}\right)\left\langle \phi_{i},\omega_{V,\psi,\mu}\left(\exp X\right)\phi_{i}\right\rangle dX
\]
\[
=\int_{\mathfrak{h}_V\left(F\right)}f_{\omega}\left(X\right)\left\langle \phi_{i},\omega_{V,\psi,\mu}\left(\exp X\right)\phi_{i}\right\rangle dh,
\]
for all $f\in{\mathcal{C}}_{\text{scusp}}\left(\Omega\right)$. This gives us 
$$
J_V\left(f\right)=\int_{Z_G(F)H_V(F)\backslash G(F)}
\sum_i \int_{\mathfrak{h}_V(F)}f_\omega(x^{-1}Xx)\langle \phi_i,\omega_{\psi,\mu}(\exp(X))\phi_i\rangle dXdx.
$$
By Proposition \ref{identity}, when $\omega$ is sufficiently small, the inner sum of the above formula is equal to
$$
\int_{V(F)} \widehat{^x f_\omega|_{\mathfrak{h}_V(F)}}(\Phi(v))dv,
$$
where $\Phi$ is the moment map of $H_V$-space $\text{Res}_{E/F}V$. By the Fourier inverse formula, it follows that 
$$
\int_{V(F)} \widehat{^x f_{|\mathfrak{h}_V(F)}}(\Phi(v))dv=\int_{V(F)}\int_{\mathfrak{h}_V^\perp(F)} \hat{f_\omega}(x^{-1}Xx+\Phi(x^{-1}\cdot v))dXdv,
$$
which is to say $J_V(f)=J^{\text{Lie}}_V(f_\omega)$ as desired.

\item We set
$$
J^{\text{Lie}}_{V,\text{qc}}(\theta)=\int_{\Gamma(\Sigma_V)}D^G(X)^{1/2}\hat{\theta}(X)dX,
$$
for $\theta \in \text{QC}_c(\mathfrak{g}(F))$. By Theorem \ref{Liespectral}, we have
$$
J^{\text{Lie}}_{V}(f)=J^{\text{Lie}}_{V,\text{qc}}(\theta_f),
$$
for any $f\in \mathcal{C}_{\text{scusp}}(\mathfrak{g}(F))$, which gives us the statement of the existence. The uniqueness follows from the surjectivity of the map $f \mapsto \theta_f$. Using the formula in Theorem \ref{Liespectral}, we obtain  
$$
J^{\text{Lie}}_{V,\text{qc}}(\theta_\lambda)=|\lambda|^{\delta(G)/2}J^{\text{Lie}}_{V_\lambda,\text{qc}}(\theta),
$$
for all $\theta \in \text{QC}_c(\mathfrak{g}(F))$ and $\lambda \in F^\times$.

\item By Proposition \ref{elliptic} and Proposition \ref{central}, together with part (i) and part (ii), there exists an open neighborhood $\omega$ of $0$ in $\mathfrak{g}(F)$ such that
$$
J^{\text{Lie}}_{V,\text{qc}}(\theta)=m_{V,\text{geom}}^{\text{Lie}}(\theta),
$$
for any $\theta\in \text{QC}_c(\omega)$ supported outside $0$. Moreover, by the homogeneity properties in part (ii) and Proposition \ref{4.1}(iii), we can extend the above statement to any $\theta\in \text{QC}_c(\mathfrak{g}(F))$ supported outside $0$.
\end{enumerate}
\end{proof}
\subsection{End of the proof of Theorem \ref{maintheorem1}}\label{sec5.6}

We give the first approximation for $J_{V,\text{qc}}^{\text{Lie}}$.
\begin{proposition}\label{prop6.11}
Assume the induction hypothesis. There exists a constant $c_V\in\mathbb{C}$
such that 
\begin{enumerate}
    \item When $n$ is even,
    $$
    J_{V,\text{qc}}^{\text{Lie}}\left(\theta\right)=c_V\cdot c_{\theta}\left(0\right)+m_{V,\text{geom}}^{\text{Lie}}\left(\theta\right),
    $$
    for all $\theta\in QC_{c}\left(\mathfrak{g}\left(F\right)\right)$.
    \item When $n$ is odd,
    $$
    J_{V,\text{qc}}^{\text{Lie}}\left(\theta\right)=c_V\cdot (c_{\theta,\mathcal{O}_1}\left(0\right)-c_{\theta,\mathcal{O}_2}\left(0\right))+m_{V,\text{geom}}^{\text{Lie}}\left(\theta\right),
    $$
    for all $\theta\in QC_{c}\left(\mathfrak{g}\left(F\right)\right)$. Here $\mathcal{O}_1$ and $\mathcal{O}_2$ are two regular nilpotent orbits in $\mathfrak{g}$.
\end{enumerate}
\end{proposition}

\begin{proof}
By Proposition \ref{elliptic} and \ref{central}, for any $\theta\in \text{QC}_c(G(F))$, we have 
$$
J_{V,\text{qc}}^{\text{Lie}}(\theta)=m_{V,\text{geom}}^{\text{Lie}}(\theta)+\sum_{\mathcal{O}\in\text{Nil}(\mathfrak{g})}c_{V,\mathcal{O}}\cdot c_{\theta,\mathcal{O}}(0).
$$
By Proposition \ref{5.10}(ii) and Proposition \ref{4.1}(iii), by substituting $\theta_\lambda$ to $\theta$, for some $\lambda\in F^\times$, we obtain
$$
|\lambda|^{\delta(G)/2}J_{V_\lambda,\text{qc}}^{\text{Lie}}(\theta)=|\lambda|^{\delta(G)/2}m_{V_\lambda,\text{geom}}^{\text{Lie}}(\theta)+\sum_{\mathcal{O}\in\text{Nil}(\mathfrak{g})}|\lambda|^{\dim(\mathcal{O})/2}\cdot c_{V,\mathcal{O}}\cdot c_{\theta,\mathcal{O}_\lambda}(0).
$$
Since $\dim(\mathcal{O})\leq \delta(G)$, we can see that $c_{V,\mathcal{O}}=0$ unless $\mathcal{O}$ is a regular nilpotent orbit. When $n$ is even, $\mathfrak{g}$ only has one regular nilpotent orbit, thus we are able to deduce (i). When $n$ is odd, we can choose $\lambda \in F^\times \backslash N(E^\times)$ to exchange the two regular nilpotent orbits in $\mathfrak{g}$. Let $V^\prime$ be the remaining $n$-dimensional skew-hermitian space. In this case, we have
$$
c_{V,\mathcal{O}_1}=c_{V^\prime,\mathcal{O}_2}\ \ \ \ \text{ and }\ \ \ \ c_{V,\mathcal{O}_2}=c_{V^\prime,\mathcal{O}_1}.
$$
Let $M$ be a tempered $L$-parameter for $G(F)$ such that $M$ is of the form
$$
M=M_1+\ldots +M_n,
$$
where $M_i$ is one-dimensional and conjugate self-dual of parity $(-1)^{n-1}$ for all $i=\overline{1,n}$. We set $\theta_M=\sum_{\pi\in \Pi_M}\theta_\pi$. By the main theorem in \cite{CG22}, we have
$$
\sum_{V}J_{V,\text{qc}}(\theta_M)=\sum_{V}\sum_{\pi\in \Pi_M}m_{V}(\pi)=1.
$$
This gives us
$$
\sum_{V}m_{V,\text{geom}}(\theta_M)+\sum_{V}\sum_{\mathcal{O}\in \text{Nil}_{\text{reg}}(\mathfrak{g})}c_{V,\mathcal{O}}\cdot c_{\theta_M,\mathcal{O}}(1)=1.
$$
Since $\sum_{V}m_{V,\text{geom}}(\theta_M)=c_{\theta_{M}}(1)=1$ (see Section \ref{sec5.7} for an explanation), it follows that
$$
\sum_{V}\sum_{\mathcal{O}\in \text{Nil}_{\text{reg}}(\mathfrak{g})}c_{V,\mathcal{O}}\cdot c_{\theta_M,\mathcal{O}}(1)=0.
$$
As the LHS is equal to $2(c_{V,\mathcal{O}_1}+c_{V,\mathcal{O}_2})\cdot c_{\theta_M}(1)$, we have $c_{V,\mathcal{O}_1}=-c_{V,\mathcal{O}_2}$, which gives us the statement in (ii).
\end{proof}
We now finish our proof for Theorem \ref{maintheorem1}.
\begin{proof}
There remains to show the coefficient $c_V$ is zero. Fix a Borel subgroup $B\subset G$ and a maximal torus $T_{\text{qd}}\subset B$ defined over $F$. Denote by $\Gamma_{\text{qd}}\left(\mathfrak{g}\right)$
the subset of $\Gamma\left(\mathfrak{g}\right)$ consisting of the
conjugacy classes that meet $\mathfrak{t}_{\text{qd}}\left(F\right)$.
We recall the subset $\Gamma\left(\Sigma_V\right)$ of $\Gamma\left(\mathfrak{g}\right)$ consisting of semisimple parts of representatives of $G(F)$-conjugacy classes of $\mathfrak{h}_V^\perp(F)+\Phi(V(F))$ defined in subsection \ref{sec5.4}. 

Assume that $B$ is a good Borel subgroup and $\mathfrak{b}\cap\Sigma_V\neq \emptyset$. By Proposition \ref{5.7}, we have 
\[
\mathfrak{g}=\mathfrak{h}_V^\perp\oplus d\Phi(V)\oplus\mathfrak{u},
\]
where $\mathfrak{u}$ is the nilpotent radical of $\mathfrak{b}$. Moreover, since $\mathfrak{g}=\mathfrak{h}_V^\perp\oplus(\mathfrak{h}_V\cap \mathfrak{b})\oplus \mathfrak{u}$, it follows that the restriction of the natural projection $\mathfrak{b}\rightarrow \mathfrak{t}_{\text{qd}}$ to $\Sigma_V \cap \mathfrak{b}$ is $\mathfrak{h}_V\cap \mathfrak{b}$.

We can identify $\mathfrak{t}_{\text{qd}}(F)$ with $E^n$ when $E=K$ or $L^{[(n-1)/2]}\times (L^0)^{n-2[(n-1)/2]}$ when $E\neq K$. Here $L^0$ is the subset $\ker(\text{Tr}_{L/K})$ of $L$. Let $\theta_{0}\in C_{c}^{\infty}\left(\mathfrak{t}_{\text{qd,reg}}\left(F\right)\right)$
be $W\left(G,T_{\text{qd}}\right)$-invariant and such that 
\[
\int_{\mathfrak{t}_{\text{qd}}\left(F\right)}D^{G}\left(X\right)^{1/2}\theta_{0}\left(X\right)dX\neq0.
\]
By the above identification of $\mathfrak{t}_{\text{qd}}(F)$, we can assume $\theta_0=\theta_1^{\otimes n}$ when $E=K$ or $\theta_0=\theta_1^{\otimes [(n-1)/2]}\otimes \theta_2^{\otimes (n-2[(n-1)/2])}$ when $E\neq K$. We extend $\theta_{0}$ to a smooth invariant function on $\mathfrak{g}_{\text{reg}}$ which is zero outside $\mathfrak{t}_{\text{qd,reg}}\left(F\right)^{G}$.
Then $\theta_{0}$ is a compactly supported quasi-character. Let $\theta=\hat{\theta}_{0}$.
Since 
\[
\theta=\int_{\Gamma\left(\mathfrak{g}\right)}D^{G}\left(X\right)^{1/2}\theta_{0}\left(X\right)\hat{j}\left(X,\cdot\right)dX
\]
and 
\[
D^{G}\left(Y\right)^{1/2}\hat{j}\left(X_{\text{qd}},Y\right)=\begin{cases}
\underset{w\in W\left(G,T_{\text{qd}}\right)}{\sum}\psi\left(B\left(X_{\text{qd}},wY\right)\right) & \text{if }Y\in\mathfrak{t}_{\text{qd,reg}}\left(F\right)\\
0 & \text{otherwise},
\end{cases}
\]
we have $\text{Supp}\left(\theta\right)\subseteq\Gamma_{\text{qd}}\left(\mathfrak{g}\right)$.
This gives us 
\[
c_{\theta}\left(0\right)=\int_{\Gamma\left(\mathfrak{g}\right)}D^{G}\left(X\right)^{1/2}\theta_{0}\left(X\right)c_{\hat{j}\left(X,\cdot\right)}\left(0\right)dX=\int_{\Gamma_{\text{qd}}\left(\mathfrak{g}\right)}D^{G}\left(X\right)^{1/2}\theta_{0}\left(X\right)dX.
\]
We need to show
$$
J^{\text{Lie}}_{V,\text{qc}}(\theta)=m_{V,\text{geom}}^{\text{Lie}}(\theta),
$$
i.e.
\begin{equation}\label{keyeqn}
    \int_{\Gamma_{\text{qd}}(\Sigma)}D^G(X)^{1/2}\theta_0(X)dX
=\frac{1}{2}\int_{\Gamma_{\text{qd}}\left(\mathfrak{g}\right)}D^{G}\left(X\right)^{1/2}\theta_{0}\left(X\right)dX
\end{equation}
$$
+\frac{\mu(\det V)\gamma_\psi(2Nm_{E/F})^n}{|W(G,T_{\text{qd}})|}
\lim_{s\rightarrow 0^+}\int_{\mathfrak{t}_{\text{qd}}(F)\cap \mathfrak{h}_V(F)}D^G(X)^{1/2}\hat{\theta}_0(X)\frac{\mu(\det X)}{|\det X|^{1/2-s}}dX.
$$
Let $\delta$ be a $\text{Tr}_{E/F}$-$0$ element in $E$. We pick a $\text{Tr}_{K/F}$-$0$ element $\tau$ in $K$ such that $\omega_{E/F}(N_{K/F}(\tau))=1$. By Theorem \ref{Liespectral} when $n=1$ and Proposition \ref{4.3}, together with a descent-to-Lie-algebra statement, we have
\begin{equation}\label{theta1}
    \int_{F+\delta\cdot \text{N}_{E/F}(E)}\theta_1(X)dX=\frac{1}{2}\int_{E}\theta_{1}\left(X\right)dX
+\mu(\delta)\gamma_\psi(2Nm_{E/F})
\lim_{s\rightarrow 0^+}\int_{\delta \cdot F}\hat{\theta}_1(X)\frac{\mu(X)}{|X|_E^{1/2-s}}dX
\end{equation}
and
\begin{equation}\label{theta2}
    \int_{\delta \tau F+\delta\cdot \text{N}_{E/F}(E)}\theta_2(X)dX=\frac{1}{2}\int_{\delta K}\theta_{2}\left(X\right)dX
+\mu(\delta)\gamma_\psi(2Nm_{E/F})
\lim_{s\rightarrow 0^+}\int_{\delta \cdot F}\hat{\theta}_2(X)\frac{\mu(X)}{|X|_E^{1/2-s}}dX.
\end{equation}
Under the identification of $\mathfrak{t}_{\text{qd}}(F)$, observe 
$$
\Gamma_{\text{qd}}(\Sigma_V)=\bigsqcup_{\stackrel{\delta_i\in E^0/N_{E/F}(E)}{\prod_i \mu(\delta_i)=\mu(\det V)}}\prod_i (\delta_i\tau F+\delta_i\cdot Nm_{E/F}(E)).
$$ 
Therefore, by combining the equations (\ref{theta1}) and (\ref{theta2}), we obtain the equation (\ref{keyeqn}) as desired. Moreover, since $c_{\theta}\left(0\right)\neq0$, we
have $c=0$, thus finish the proof of Theorem \ref{maintheorem1}.
\end{proof}

\subsection{Proof of Theorem \ref{maintheorem}(i)}\label{sec5.7}
We now prove part 1 of Theorem \ref{maintheorem}. Let $\varphi$ be a tempered L-parameter of $G(F)$. We want to show that
$$
\sum_{V}\sum_{\pi \in \Pi_{\varphi}}m_V(\pi)=1,
$$
where the first sum runs over the two skew-hermitian spaces over $E$ of dimension $n$. We fix a skew-hermitian space $V$. By Theorem \ref{maintheorem1}, we have
$$
\sum_{\pi \in \Pi_{\varphi}}m_V(\pi)=\frac{1}{2}c_\varphi(1)+\mu\left(\det V\right)
\underset{T\in{\mathcal{T}}_{\text{ell}}\left(H_V\right)}{\sum}\frac{\gamma_\psi(T)}{\left|W\left(H_V,T\right)\right|}
\underset{s\rightarrow0^{+}}{\lim}\int_{T\left(F\right)}D^{G}\left(x\right)^{1/2}c_{\varphi}\left(x\right)\frac{\mu\left(\det\left(1-x^{-1}\right)\right)}{\left|\det\left(1-x\right)\right|_{E}^{1/2-s}}dx,
$$
where $c_\varphi=\sum_{\pi\in \Pi_\varphi}c_{\theta_\pi}$. We denote by $\mathcal{T}_{\text{ell}}^{\text{stab}}(H_V)$ the set of representatives of stable $H_V$-conjugacy classes of elliptic maximal tori in $H_V$. Let $p_{V,\text{stab}}:\mathcal{T}_{\text{ell}}(H_V)\rightarrow \mathcal{T}_{\text{ell}}^{\text{stab}}(H_V)$ be the natural projection map. Since $\theta_\varphi$ is stably invariant (see \textbf{(Stab)} in section \ref{stab}), it follows that 
$$
\sum_{\pi \in \Pi_{\varphi}}m_V(\pi)=\frac{1}{2}c_\varphi(1)+\mu\left(\det V\right)
\underset{T\in{\mathcal{T}}^{\text{stab}}_{\text{ell}}\left(H_V\right)}{\sum}\frac{\gamma_\psi(T)|p_{V,\text{stab}}^{-1}(T)|}{\left|W\left(H_V,T\right)\right|}
$$
$$\underset{s\rightarrow0^{+}}{\lim}\int_{T\left(F\right)}D^{G}\left(x\right)^{1/2}c_{\varphi}\left(x\right)\frac{\mu\left(\det\left(1-x^{-1}\right)\right)}{\left|\det\left(1-x\right)\right|_{E}^{1/2-s}}dx.
$$
Let $V^\prime$ is the other skew-hermitian space over $E$ of dimension $n$. Since ${\mathcal{T}}^{\text{stab}}_{\text{ell}}\left(H_V\right)$ is bijective to ${\mathcal{T}}^{\text{stab}}_{\text{ell}}\left(H_{V^\prime}\right)$, and $|p_{V,\text{stab}}^{-1}(T)|=|p_{V^\prime,\text{stab}}^{-1}(T^\prime)|$ whenever $T\in {\mathcal{T}}^{\text{stab}}_{\text{ell}}\left(H_V\right)$ matches $T^\prime\in {\mathcal{T}}^{\text{stab}}_{\text{ell}}\left(H_{V^\prime}\right)$, it follows that  
$$
\sum_{V}\sum_{\pi \in \Pi_{\varphi}}m_V(\pi)=c_\varphi(1)=1,
$$
where the last equality follows from the genericity of $\varphi$. We have finished our proof for Theorem \ref{maintheorem}(i).

\section{Geometric expansion of the twisted local trace formula}\label{sec8}

Recall the twisted trace formula $\tilde{J}_{\chi}$ formulated in \cite{Le25}. We define a linear form $\epsilon_{\text{geom}}$ on the space of quasi-characters $\text{QC}(\tilde{M}(F))$. The main objective of this section is to show that $\tilde{J}_{\chi}(\tilde{f})=\epsilon_{\text{geom}}(\theta_{\tilde{f}})$ for any $\tilde{f}\in \mathcal{C}_{\text{scusp}}(Z_M(F)\backslash \tilde{M}(F),\chi)$.

\subsection{The twisted trace formula $\tilde{J}_{\chi}$}\label{sec7.1}
Let $F$ be a non-archimedean local field of characteristic $0$ and
$E$ and $K$ be quadratic extensions of $F$. We set $L=K\otimes_{F}E$ and $M=\text{Res}_{L/F}\text{GL}_{n}$. Let
$\theta_{n}:\left(g,h\right)\mapsto\left(J_{n}{}^{t}\bar{h}^{-1}J_{n}^{-1},J_{n}{}^{t}\bar{g}^{-1}J_{n}^{-1}\right)$ when $K=E$, or $g\mapsto J_{n}{}^{t}\bar{g}^{-1}J_{n}^{-1}$ when
$K\neq E$, be an involution on $M$. Here $\bar{\cdot}$ is the action
of the nontrivial element in $\text{Gal}\left(E/F\right)$ and 
\[
J_{n}=\left(\begin{array}{cccc}
0 & \cdots & 0 & -1\\
\vdots & 0 & 1 & 0\\
0 & \ddots & 0 & \vdots\\
\left(-1\right)^{n} & 0 & \cdots & 0
\end{array}\right).
\]
The restriction of $\theta_{n}$ to the subgroup $N=\text{Res}_{E/F}\text{GL}_{n}$ deduces an involution on $N$. We set $\tilde{M}=M\theta_{n}$ and $\tilde{N}=N\theta_{n}$.
Let $V$ be an $n$-dimensional vector space over $E$ and $\mathcal{S}\left(V\right)$
be the space of Schwarz functions on $V$. We define the Weil representation
$\omega_{\mu}$ of $N\left(F\right)$ realized on $\mathcal{S}\left(V\right)$
by 
\[
\left(\omega_{\mu}\left(g\right)\phi\right)\left(v\right)=\left|\det g\right|_{E}^{\frac{1}{2}}\mu\left(\det g\right)\phi\left(vg\right),
\]
for any $g\in N\left(F\right)$ and $\phi\in\mathcal{S}\left(V\right)$.
We give an extension of $\omega_{\mu}$ to $\tilde{N}\left(F\right)$ by taking
$\tilde{\omega}_{\psi,\mu}\left(\theta_{n}\right)\phi=\widehat{\overline{\phi}\left(\cdot J_{n}\right)}$,
where $\bar{\phi}\left(v\right)=\phi\left(\bar{v}\right)$ and $\hat{\phi}$
is the Fourier transform of $\phi$ with respect to $\psi_{E}=\psi\circ\text{Tr}_{E/F}$.
We fix a central character $\chi$ of $M\left(F\right)$ such that
$\chi$ is invariant under the action of $\theta_{n}$. For simplicity, we also denote by $\chi$ its restriction to $Z_{N}\left(F\right)$. Let $\omega_{\mu,\chi}$ be the $\chi$-isotypic summand of $\omega_{\mu}$. We denote by $\tilde{\omega}_{\psi,\mu,\chi}$
its extension to $\tilde{N}\left(F\right)$. Let $\left\{ \phi_{i}\right\} _{i\in I}$
be an orthonormal basis for $\tilde{\omega}_{\psi,\mu,\chi}$. For
any $m\in M\left(F\right)$, we set 
\[
K_{\chi}\left(\tilde{f},x\right)=\underset{i}{\sum}\int_{Z_{N}\left(F\right)\backslash\tilde{N}\left(F\right)}\tilde{f}\left(m^{-1}\tilde{n}m\right)\left\langle \phi_i,\tilde{\omega}_{\psi,\mu,\chi}\left(\tilde{n}\right)\phi_{i}\right\rangle d\tilde{n},
\]
where $\tilde{f}\in\mathcal{C}_{\text{scusp}}\left(Z_{M}\left(F\right)\backslash \tilde{M}\left(F\right),\chi\right)$. The above kernel function is locally constant and invariant under $N\left(F\right)Z_{M}\left(F\right)$.
We define the following linear form 
\[
\tilde{J}_{\chi}\left(\tilde{f}\right)=\int_{N\left(F\right)Z_{M}\left(F\right)\backslash M\left(F\right)}K_\chi\left(\tilde{f},m\right)dm,
\]
for $\tilde{f}\in\mathcal{C}_{\text{scusp}}\left(Z_{M}\left(F\right)\backslash \tilde{M}\left(F\right),\chi\right)$. Similar to Theorem \ref{spectral}, the integrals defining $K_\chi$ and $\tilde{J}_{\chi}$ are absolutely convergent.

Let $\left(\pi,\tilde{\pi},E_{\pi}\right)\in\text{Temp}\left(\tilde{M}\left(F\right)\right)$.
For $e,e^{\prime}\in E_{\pi}$ and $\phi,\phi^{\prime}\in\omega_{\mu,\chi}$,
we define 
\[
\mathcal{L}_{\pi}\left(e\otimes\phi,e^{\prime}\otimes\phi^{\prime}\right)=\int_{N\left(F\right)}\left\langle \pi\left(g\right)e,e^{\prime}\right\rangle \left\langle \phi,\omega_{\mu,\chi}\left(g\right)\phi^{\prime}\right\rangle dg.
\]
By using some estimates in \cite[Appendix D.1]{Xue16}, the above expression is absolutely convergent. We have 
\[
\mathcal{L}_{\pi}\left(\pi\left(g\right)e\otimes\phi,e^{\prime}\otimes\omega_{\mu,\chi}\left(g\right)\phi^{\prime}\right)=\mathcal{L}_{\pi}\left(e\otimes\phi,e^{\prime}\otimes\phi^{\prime}\right),
\]
for any $g\in N\left(F\right)$. As in \cite[Remark 1]{GGP23}, observe $\dim\text{Hom}_{N}\left(\pi,\omega_{\mu}\right)=1$. Let $e$ be a nonzero trace-$0$ element in $E$. We set 
\[
\epsilon_{\psi}\left(\widetilde{\pi}\right)=\omega_\pi\left(e\right)^{n}\omega_{K/F}\left(-1\right)^{\frac{n\left(n-1\right)}{2}}\epsilon\left(\frac{1}{2},\text{As}_{L/E}\left(\pi\right)\times\mu^{-1},\psi_{E}\right),
\]
where $\omega_\pi$ is the central character of $\pi$. We have the following intertwining relation.
\begin{proposition}\label{intertwining}
For any $e,e^{\prime}\in E_{\pi}$ and $\phi,\phi^{\prime}\in\omega_{\mu,\chi}$
and $\tilde{y}\in\tilde{N}\left(F\right)$, we have 
\[
\mathcal{L}_{\pi}\left(\widetilde{\pi}\left(\tilde{y}\right)e\otimes\phi,e^{\prime}\otimes\tilde{\omega}_{\psi,\mu,\chi}\left(\tilde{y}\right)\phi^{\prime}\right)=\epsilon_{\psi}\left(\widetilde{\pi}\right)\mathcal{L}_{\pi}\left(e\otimes\phi,e^{\prime}\otimes\phi^{\prime}\right).
\]
\end{proposition}

\begin{proof}
When $E=K$, the statement has been proved in \cite[Proposition 7.2]{Le25}. In the case $K\neq E$, the proof follows verbatim up to replacing the local functional equation argument for Rankin-Selberg integrals in \cite{JPSS83} with its analog for Asai Rankin-Selberg integrals in \cite{Fli93,Kab04}.
\end{proof}

Let $\tilde{f}\in \mathcal{C}_\text{scusp}(Z_M(F)\backslash \tilde{M}(F),\chi)$.
Define 
$$
J_{\chi,\text{ spec}}\left(\tilde{f}\right)=\underset{\tilde{L}\in\mathcal{L}(\tilde{M}_{\min})}{\sum}|\widetilde{W}^L||\widetilde{W}^M|^{-1}(-1)^{a_{\tilde{L}}-a_{\tilde{G}}}\int_{E_{\text{ell}}(Z_M(F)\backslash\tilde{L}(F),\chi^{-1})}\hat{\theta}_{\tilde{f}}\left(\tilde{\pi}\right)\epsilon_{\psi}\left(\tilde{\pi}^{\vee}\right)d\tilde{\pi}.
$$
We state the following theorem, whose proof follows from \cite[Theorem 8.1]{Le25} verbatim.
\begin{theorem}\label{spectraltwisted}
For any $\tilde{f}\in \mathcal{C}_\text{scusp}(Z_M(F)\backslash \tilde{M}(F),\chi)$, we have
$$
\tilde{J}_{\chi}\left(\tilde{f}\right)=J_{\chi,\text{spec}}\left(\tilde{f}\right).
$$
\end{theorem}

\subsection{The linear form $\epsilon_{\text{geom}}$}

Similar to section \ref{sec4.1}, we define a linear form $\epsilon_{\text{geom}}$ on the space $\text{QC}(\tilde{M}(F))$. Let $\mathcal{T}_{\text{ell}}(\tilde{N})$ be the set containing representatives of $N(F)$-conjugacy classes of elliptic twisted maximal tori in $\tilde{N}$. Let $(T,\tilde{T})\in \mathcal{T}_{\text{ell}}(\tilde{N})$. We denote by $\theta$ the corresponding involution of $T$. Then $T_\theta(F)$ is isomorphic to $\prod_i U_{E_i/F_i}(1)$, where $F_i$ is a field extension of $F$ not containing $E$ and $E_i=EF_i$. As in section \ref{sec4.1}, we set
$$
\gamma_\psi(\tilde{T})=\prod_i \gamma_\psi(2Nm_{E_i/F_i}).
$$
Let $\tilde{x}\in \tilde{T}(F)$. We set $x={}^t(\tilde{x}^\sigma)^{-1}\tilde{x}$, where $\sigma$ is the nontrivial $F$-automorphism of $E$. For a quasi-character $\tilde{\theta}\in \text{QC}(\tilde{M}(F))$, we define
$$
\epsilon_{\text{geom}}(\tilde{\theta})=\sum_{\tilde{T}\in \mathcal{T}_{\text{ell}}(\tilde{N})}\frac{\gamma_\psi(\tilde{T})}{|W(N,\tilde{T})|}\lim_{s\rightarrow 0^+}\int_{\tilde{T}(F)/\theta}D^{\tilde{M}}(\tilde{x})^{1/2}c_{\tilde{\theta}}(\tilde{x})\frac{\mu(\det(1-x^{-1}))}{|\det(1-x)|_E^{1/2-s}}d\tilde{x}.
$$
A similar argument to Proposition \ref{4.1} shows that the linear form $\epsilon_{\text{geom}}(\tilde{\theta})$ is absolutely convergent. For any virtual tempered representation $\tilde{\pi}$ of $\tilde{M}(F)$, we set $\epsilon_{\text{geom}}(\tilde{\pi})=\epsilon_{\text{geom}}(\theta_{\tilde{\pi}})$. 

In the remaining of this section, we prove the following theorem by induction.
\begin{theorem}\label{geometrictwisted}
    For any $\tilde{f} \in \mathcal{C}_{\text{scusp}}(Z_M(F)\backslash\tilde{M}(F),\chi)$, we have
    $$
    \tilde{J}_{\chi}(\tilde{f})=\epsilon_{\text{geom}}(\theta_{\tilde{f}}).
    $$
    Therefore, for any tempered representation $\tilde{\pi}$ of $\tilde{M}(F)$,
    $$
    \epsilon_{\psi}(\tilde{\pi})=\epsilon_{\text{geom}}(\tilde{\pi}).
    $$
\end{theorem}

\subsection{The case $K=E$}\label{sec8.2}

When $K=E$, we can use the main result in \cite{Le25} to prove Theorem \ref{geometrictwisted}. In this case, we have $M=\text{Res}_{E/F}\text{GL}_{n,E}\times\text{GL}_{n,E}$ and $N=\text{Res}_{E/F}\text{GL}_{n,E}$. Let $\widetilde{\pi \times {}^\sigma\pi^\vee}$ be a tempered representation of $\tilde{M}(F)$. By \cite[Theorem 1.2]{Le25} and Theorem \ref{maintheorem1}, it follows that
$$
\epsilon_\psi(\widetilde{\pi \times {}^\sigma\pi^\vee})=\mu(\det V)m_V(\pi)+\mu(\det V^\prime)m_{V^\prime}(\pi)
$$
$$
=\underset{T\in{\mathcal{T}}_{\text{ell}}\left(H\right)}{\sum}\frac{\gamma_\psi(T)}{\left|W\left(H,T\right)\right|}
\underset{s\rightarrow0^{+}}{\lim}\int_{T\left(F\right)}D^{G}\left(x\right)^{1/2}c_{\theta}\left(x\right)\frac{\mu\left(\det\left(1-x^{-1}\right)\right)}{\left|\det\left(1-x\right)\right|_{E}^{1/2-s}}dx.
$$
By using the matching in section \ref{llc2} and Theorem \ref{basechangeGL}, we obtain
$$
\epsilon_\psi(\widetilde{\pi \times {}^\sigma\pi^\vee})=\sum_{\tilde{T}\in \mathcal{T}_{\text{ell}}(\tilde{N})}\frac{\gamma_\psi(\tilde{T})}{|W(N,\tilde{T})|}\lim_{s\rightarrow 0^+}\int_{\tilde{T}(F)/\theta}D^{\tilde{M}}(\tilde{x})^{1/2}c_{\theta_{\widetilde{\pi \times {}^\sigma\pi^\vee}}}(\tilde{x})\frac{\mu(\det(1-x^{-1}))}{|\det(1-x)|_E^{1/2-s}}d\tilde{x}
$$
$$
=\epsilon_{\text{geom}}(\widetilde{\pi \times {}^\sigma\pi^\vee}).
$$
Applying Theorem \ref{spectraltwisted}, for any $\tilde{f}\in \mathcal{C}_{\text{scusp}}(Z_M(F)\backslash\tilde{M}(F),\chi)$, we have
$$
\tilde{J}_{\chi}(\tilde{f})=\underset{\tilde{L}\in\mathcal{L}(\tilde{L}_{\min})}{\sum}|\widetilde{W}^L||\widetilde{W}^G|^{-1}(-1)^{a_{\tilde{L}}-a_{\tilde{G}}}\int_{E_{\text{ell}}(Z_M(F)\backslash\tilde{L}(F),\chi^{-1}\times{}^\sigma\chi)}\hat{\theta}_{\tilde{f}}\left(\tilde{\pi}\right)\epsilon_{\psi}\left(\tilde{\pi}^{\vee}\right)d\tilde{\pi}
$$
$$
=\underset{\tilde{L}\in\mathcal{L}(\tilde{L}_{\min})}{\sum}|\widetilde{W}^L||\widetilde{W}^G|^{-1}(-1)^{a_{\tilde{L}}-a_{\tilde{G}}}\int_{E_{\text{ell}}(Z_M(F)\backslash\tilde{L}(F),\chi^{-1}\times{}^\sigma\chi)}\hat{\theta}_{\tilde{f}}\left(\tilde{\pi}\right)\epsilon_{\text{geom}}\left(\tilde{\pi}^{\vee}\right)d\tilde{\pi}=\epsilon_{\text{geom}}(\theta_{\tilde{f}}),
$$
which gives us a proof of Theorem \ref{geometrictwisted} when $K=E$.

\subsection{Linearization of $\tilde{J}_{\chi}$} From now, we consider the case when $K\neq E$. We first prove a compatibility between the geometric twisted multiplicity $\epsilon_{\text{geom}}$ and parabolic induction. Let $\tau$ be the nontrivial $F$-automorphism of $K$ and $E^\prime$ be the third quadratic subfield of $L$. Let $\tilde{L}$ be a twisted Levi subgroup of $\tilde{M}$. There exists a decomposition
$$
V_K=V_u\oplus \ldots V_{1}\oplus V_0 \oplus V_{-1}\oplus \ldots \oplus V_{-u}
$$
such that $\tilde{L}$ is the subset containing $\tilde{x}\in \tilde{M}$ which satisfies $\tilde{x}(V_i)=V_{-i}^\sigma$, for any $i=\overline{1,u}$. For $i=\overline{0,u}$, we set
$$
w_i=\left(\begin{array}{cccc}
0 & \cdots & 0 & -1\\
\vdots & 0 & 1 & 0\\
0 & \ddots & 0 & \vdots\\
\left(-1\right)^{\dim(V_i)} & 0 & \cdots & 0
\end{array}\right).
$$
Similar to the case $K=E$, for each $i=\overline{1,u}$, we denote $M_i=\text{Res}_{L/F}\text{GL}(V_i)\times \text{GL}(V_{-i})$ and $\tilde{M_i}=M_i\theta_i$, where $\theta_i: (g_i,g_{-i})\mapsto (w_i{}^tg_{-i}^{\sigma\tau,-1}w_i^{-1},w_i{}^tg_{i}^{\sigma\tau,-1}w_i^{-1})$ is an involution on $M_i$. Likewise, as in the case $K\neq E$, we set $M_0=\text{Res}_{L/F}\text{GL}(V_0)$ and $\tilde{M_0}=M_0\theta_0$, where $\theta_0:g\mapsto w_0{}^tg^{\sigma,-1}w_0$ is an involution on $M_0$. Then $\tilde{L}=\prod_{i=0}^{u}\tilde{M_i}$. We denote by $\epsilon_{\text{geom}}^{\tilde{M_i}}$ the variant of the linear form $\epsilon_{\text{geom}}$ when replacing $\tilde{M}$ by $\tilde{M_i}$.

Let $\theta^{\tilde{L}}=\otimes_{i=0}^u\,\theta^{\tilde{M_i}}$ be a quasi-character of $\tilde{L}$, where $\theta^{\tilde{M_i}}\in \text{QC}(\tilde{M_i})$ for any $i=\overline{0,u}$. We denote $\tilde{\theta}=\text{Ind}^{\tilde{M}}_{\tilde{L}}(\theta^{\tilde{L}})$ in the sense of \cite[Section 1.12]{Wal12a}. The following proposition gives us a compatibility between geometric twisted multiplicity and parabolic induction.
\begin{proposition}\label{twistedparabolic}
    Using the above notations, we have
    $$
    \epsilon_{\text{geom}}(\tilde{\theta})=\prod_{i=0}^{u}\epsilon_{\text{geom}}^{\tilde{M_i}}(\theta^{\tilde{M_i}}).
    $$
\end{proposition}
\begin{proof}
    The proof can be adapted from \cite[Proposition 6.2.1]{BP14} without any difficulties.
\end{proof}
Applying \cite[Proposition 10.4]{GGP23}, Proposition \ref{twistedparabolic} and the induction hypothesis, we obtain
$$
\epsilon_\psi(\tilde{\pi})=\epsilon_{\text{geom}}(\tilde{\pi}),
$$
for any $\tilde{\pi}\in E_{\text{ind}}(\tilde{M})$. By Theorem \ref{spectraltwisted}, for any strongly cuspidal function $\tilde{f}\in \mathcal{C}_{\text{scusp}}(Z_M(F)\backslash \tilde{M}(F),\chi)$, we have
$$
\tilde{J}_{\chi}(\tilde{f})=J_{\chi,\text{spec}}(\tilde{f})=\underset{\tilde{L}\in\mathcal{L}(\tilde{L}_{\min})}{\sum}|\widetilde{W}^L||\widetilde{W}^G|^{-1}(-1)^{a_{\tilde{L}}-a_{\tilde{G}}}\int_{E_{\text{ell}}(Z_M(F)\backslash\tilde{L}(F),\chi^{-1})}\hat{\theta}_{\tilde{f}}\left(\tilde{\pi}\right)\epsilon_{\psi}\left(\tilde{\pi}^{\vee}\right)d\tilde{\pi}
$$
$$
=\epsilon_{\text{geom}}(\theta_{\tilde{f}})+\sum_{\tilde{\pi}\in E_{\text{ell}}(Z_M(F)\backslash \tilde{M}(F),\chi^{-1})}(\epsilon_\psi(\tilde{\pi}^\vee)-\epsilon_{\text{geom}}(\tilde{\pi}^\vee))\int_{\Gamma_{\text{ell}}(Z_M\backslash \tilde{M})}D^{\tilde{M}}(\tilde{x})\theta_{\tilde{f}}(\tilde{x})\theta_{\tilde{\pi}}(\tilde{x})d\tilde{x}.
$$
For $\tilde{\theta}\in \text{QC}(Z_M(F)\backslash\tilde{M}(F),\chi)$, we set
$$
J_{\chi,\text{qc}}(\tilde{\theta})=\epsilon_{\text{geom}}(\tilde{\theta})+\sum_{\tilde{\pi}\in E_{\text{ell}}(Z_M(F)\backslash \tilde{M}(F),\chi^{-1})}(\epsilon_\psi(\tilde{\pi}^\vee)-\epsilon_{\text{geom}}(\tilde{\pi}^\vee))\int_{\Gamma_{\text{ell}}(Z_M\backslash \tilde{M})}D^{\tilde{M}}(\tilde{x})\tilde{\theta}(\tilde{x})\theta_{\tilde{\pi}}(\tilde{x})d\tilde{x}.
$$
Since $\text{Supp}(\epsilon_{\text{geom}})\subseteq \Gamma(\tilde{M})_{\text{ell}}$, it follows that $\text{Supp}(J_{\chi,\text{qc}})\subseteq \Gamma(\tilde{M})_{\text{ell}}$. Moreover, by substituting $\tilde{\theta}=\theta_{\tilde{\pi}}$, we can see that the statement
$$
J_{\chi,\text{qc}}(\tilde{\theta})=\epsilon_{\text{geom}}(\tilde{\theta}),\text{ for any }\tilde{\theta}\in \text{QC}(Z_M(F)\backslash \tilde{M}(F),\chi)
$$
implies
$$
\epsilon_\psi(\tilde{\pi})=\epsilon_{\text{geom}}(\tilde{\pi}),\text{ for any }\tilde{\pi}\in E_{\text{ell}}(Z_M(F)\backslash \tilde{M}(F),\chi),
$$
i.e. $\tilde{J}_{\chi}(\tilde{f})=\epsilon_{\text{geom}}(\theta_{\tilde{f}})$, for any $\tilde{f}\in \mathcal{C}_{\text{scusp}}(Z_M(F)\backslash \tilde{M}(F),\chi)$. Therefore, from now it suffices to prove the following theorem.
\begin{theorem}\label{twistedgeometriclinear}
    For any $\tilde{\theta}\in \text{QC}(Z_M(F)\backslash \tilde{M}(F),\chi)$, we have
    $$
    J_{\chi,\text{qc}}(\tilde{\theta})=\epsilon_{\text{geom}}(\tilde{\theta}).
    $$
\end{theorem}
\subsection{Comparision at non-identity locus}
In this subsection, we prove the two linear forms in Theorem \ref{twistedgeometriclinear} agree outside central locus.
\begin{proposition}\label{elliptictwisted}
Let $\tilde{\theta}\in \text{QC}(Z_M(F)\backslash \tilde{M}(F),\chi)$ and assume that $Z_M(F)\cap \text{Supp}(\tilde{\theta})=\emptyset$. Then
\begin{equation} \label{ellipticeqn}
    J_{\chi,\text{qc}}(\tilde{\theta})=\epsilon_{\text{geom}}(\tilde{\theta}).
\end{equation}
\end{proposition}

\begin{proof}
    Similar to Proposition \ref{elliptic}, it suffices to prove the equality for $\tilde{\theta}\in \text{QC}(\Omega)$, where $\Omega$ is a completely stably $M(F)$-invariant open subset of $\tilde{M}(F)$ of the form $\Omega_{\tilde{x}}^{M}$ for some noncentral elliptic element $\tilde{x}$ and $M$-good open neighborhood $\Omega_{\tilde{x}}\subseteq M_{\tilde{x}}(F)$. We can assume $\Omega_{\tilde{x}}$ is relatively compect modulo conjugation.

    If $\tilde{x}$ is not $M(F)$-conjugate to any element of $\tilde{N}(F)$, then we can shrink $\Omega_{\tilde{x}}$ so that $\Omega \cap \tilde{N}(F)$ is empty. In this case, it is easy to see that both sides of (\ref{ellipticeqn}) are equal to zero.

    We may now assume that $\tilde{x}\in \tilde{N}(F)$. In this case, we have
    $$
    M_{\tilde{x}}=M_1\times \ldots M_u \text{ and }N_{\tilde{x}}=N_1\times \ldots \times N_u,
    $$
    where $N_i=\text{Res}_{F_i/F}\,\text{U}_{E_i/F_i}(V_i)$ is a certain unitary group and $M_i=\text{Res}_{K/F}N_{i,K}$ is also a unitary group, for all $i=\overline{1,u}$. We can choose $\Omega_{\tilde{x}}=\Omega_1\times \ldots \times \Omega_u$ such that $\Omega_{\tilde{x}}\cap N_{\tilde{x}}(F)$ is an $N$-good open neighborhood of $\tilde{x}$ and $\Omega_i\subseteq M_i(F)$ is open and completely $M_i(F)$-invariant. Assume $\theta_{\tilde{x},\Omega_{\tilde{x}}}=\otimes_{i}\theta_i$, where $\theta_i\in \text{QC}(\Omega_i)$. For each $i=\overline{1,u}$, we choose $f_{\tilde{x},i}\in \mathcal{C}_{\text{scusp}}(\Omega_i)$ such that $\theta_{f_{\tilde{x},i}}=\theta_i$. We set $f_{\tilde{x}}=\otimes_i f_{\tilde{x},i}$. Using Proposition \ref{2.8}, we set $\tilde{f}\in \mathcal{C}_{\text{scusp}}(\Omega)$ to be a lift of $f_{\tilde{x}}$. Observe
    $$
    J_{\chi,\text{qc}}(\tilde{\theta})=\tilde{J}_{\chi}(\tilde{f})
    =\int_{Z_M(F)N(F)\backslash M(F)}\underset{i}{\sum}\int_{Z_{N}\left(F\right)\backslash\tilde{N}\left(F\right)}\tilde{f}\left(m^{-1}\tilde{n}m\right)\left\langle \phi,\tilde{\omega}_{\psi,\mu,\chi}\left(\tilde{n}\right)\phi_{i}\right\rangle d\tilde{n}dm
    $$
    $$
    =\int_{N\left(F\right)Z_M\left(F\right)\backslash M\left(F\right)}\underset{i}{\sum}\underset{N_{\tilde{x}}\left(F\right)\backslash N\left(F\right)}{\int}\ \underset{Z_N\left(F\right)\backslash N_{\tilde{x}}\left(F\right)\tilde{x}}{\int}\left(^{m}\tilde{f}\right)_{\tilde{x},\tilde{\Omega}_{\tilde{x}}}\left(n_{\tilde{x}}^{-1}\tilde{n}n_{\tilde{x}}\right)\left\langle \phi_{i},\tilde{\omega}_{\psi,\mu}\left(\tilde{n}\right)\phi_{i}\right\rangle d\tilde{n}dn_{\tilde{x}}dm
    $$
    $$
    =\underset{M_{\tilde{x}}\left(F\right)\backslash M\left(F\right)}{\int}\ \underset{Z_M\left(F\right)N_{\tilde{x}}\left(F\right)\backslash M_{\tilde{x}}\left(F\right)}{\int}\underset{i}{\sum}\underset{Z_N\left(F\right)\backslash N_{\tilde{x}}\left(F\right)\tilde{x}}{\int}\left(^{m}\tilde{f}\right)_{\tilde{x},\tilde{\Omega}_{\tilde{x}}}\left(n_{\tilde{x}}^{-1}\tilde{n}n_{\tilde{x}}\right)\left\langle \phi_{i},\tilde{\omega}_{\psi,\mu}\left(\tilde{n}\right)\phi_{i}\right\rangle d\tilde{n}dm_{\tilde{x}}dm
    $$
    $$
    =\underset{M_{\tilde{x}}\left(F\right)\backslash M\left(F\right)}{\int}\alpha\left(m\right)\underset{Z_M\left(F\right)N_{\tilde{x}}\left(F\right)\backslash M_{\tilde{x}}\left(F\right)}{\int}\underset{i}{\sum}\underset{Z_N\left(F\right)\backslash N_{\tilde{x}}\left(F\right)\tilde{x}}{\int}\tilde{f}_{\tilde{x}}\left(m_{\tilde{x}}^{-1}\tilde{n}m_{\tilde{x}}\right)\left\langle \phi_{i},\tilde{\omega}_{\psi,\mu}\left(\tilde{n}\right)\phi_{i}\right\rangle d\tilde{n}dm_{\tilde{x}}dm
    $$
    $$
    =\underset{Z_M\left(F\right)N_{\tilde{x}}\left(F\right)\backslash M_{\tilde{x}}\left(F\right)}{\int}\underset{i}{\sum}\underset{Z_N\left(F\right)\backslash N_{\tilde{x}}\left(F\right)\tilde{x}}{\int}f_{\tilde{x}}\left(m_{\tilde{x}}^{-1}\tilde{n}m_{\tilde{x}}\right)\left\langle \phi_{i},\tilde{\omega}_{\psi,\mu}\left(\tilde{n}\right)\phi_{i}\right\rangle d\tilde{n}dm_{\tilde{x}}.
    $$
    For each $i=\overline{1,u}$, we set $V_{i,K}=V_i\otimes_F K$. We define twisted groups $\tilde{M_i}$ and $\tilde{N_i}$, which are variants of $\tilde{M}$ and $\tilde{N}$, corresponding to the spaces $V_{i,K}$ and $V_i$. We denote by $\epsilon_{\text{geom}}^{\tilde{M_i}}$ the variant of the linear form $\epsilon_{\text{geom}}$ on $\text{QC}(\tilde{M}_i(F))$. Let $\tilde{\theta_i}\in \text{QC}(\tilde{M}_i(F))$ such that $(\tilde{\theta_i})_{\tilde{x},\Omega_{i}}=\theta_i$. Similar to Proposition \ref{4.1}(2), we have
    $$
    \epsilon_{\text{geom}}(\tilde{\theta})=\prod_{i=1}^u \epsilon_{\text{geom}}^{\tilde{M_i}}(\tilde{\theta_i}).
    $$
    Using the induction hypothesis, it follows that $\epsilon_{\text{geom}}^{\tilde{M_i}}(\tilde{\theta_i})=J_{\chi,\text{qc}}^{\tilde{M_i}}(\tilde{\theta_i})$. We define a lift $\tilde{f_i}\in \mathcal{C}_{\text{scusp}}(\tilde{M_i}(F))$ of $f_i$ in the sense of Proposition \ref{2.8}. A similar argument to above gives us
    $$
    J_{\chi,\text{qc}}^{\tilde{M_i}}(\tilde{\theta_i})=\underset{Z_M\left(F\right)N_i\left(F\right)\backslash M_i\left(F\right)}{\int}\underset{j}{\sum}\underset{Z_N\left(F\right)\backslash N_{i}\left(F\right)\tilde{x}}{\int}f_{\tilde{x},i}\left(m_i^{-1}\tilde{n}_im_{i}\right)\left\langle \phi_{i,j},\tilde{\omega}_{\psi\circ\text{Tr}_{F_i/F},\mu\circ \text{Nm}_{E_i/E}}\left(\tilde{n}_i\right)\phi_{i,j}\right\rangle d\tilde{n}_idm_{i},
    $$
    where $\{\phi_{i,j}\}_j$ is an orthonormal basis of the Weil representation of $\text{Res}_{E_i/F}\text{GL}(V_i)$. Thus
    $$
     \prod_{i=1}^u J_{\chi,\text{qc}}^{\tilde{M_i}}(\tilde{\theta_i}) = \underset{Z_M\left(F\right)N_{\tilde{x}}\left(F\right)\backslash M_{\tilde{x}}\left(F\right)}{\int}\underset{i}{\sum}\underset{Z_N\left(F\right)\backslash N_{\tilde{x}}\left(F\right)\tilde{x}}{\int}f_{\tilde{x}}\left(m_{\tilde{x}}^{-1}\tilde{n}m_{\tilde{x}}\right)\left\langle \phi_{i},\tilde{\omega}_{\psi,\mu}\left(\tilde{n}\right)\phi_{i}\right\rangle d\tilde{n}dm_{\tilde{x}},
    $$
    which is to say $J_{\chi,\text{qc}}(\tilde{\theta})=\epsilon_{\text{geom}}(\tilde{\theta})$.
\end{proof}

\subsection{Descent to Lie algebra and homogeneity}
We prove the following proposition.
\begin{proposition}\label{descentLietwisted}
There exists $c_\chi\in \mathbb{C}$ such that for any $\tilde{\theta}\in \text{QC}(Z_M(F)\backslash\tilde{M}(F),\chi)$, we have 
$$
J_{\chi,\text{qc}}(\tilde{\theta})=c_\chi\cdot c_{\tilde{\theta}}(\theta_n)+\epsilon_{\text{geom}}(\tilde{\theta})
$$
when $n$ is even or
$$
J_{\chi,\text{qc}}(\tilde{\theta})=c_\chi\cdot (c_{\tilde{\theta},\mathcal{O}_1}(\theta_n)-c_{\tilde{\theta},\mathcal{O}_2}(\theta_n))+\epsilon_{\text{geom}}(\tilde{\theta})
$$
when $n$ is odd.
\end{proposition}

\begin{proof}
By Proposition \ref{elliptictwisted}, it follows that
\begin{equation}\label{eqn8.4}
    J_{\chi,\text{qc}}(\tilde{\theta})=\epsilon_{\text{geom}}(\tilde{\theta})+\sum_{\mathcal{O}\in \text{Nil}(M_{\theta_n})}c_{\chi,\mathcal{O}}\cdot c_{\tilde{\theta},\mathcal{O}}(\theta_n),
\end{equation}
for any $\tilde{\theta}\in \text{QC}(Z_M(F)\backslash \tilde{M}(F),\chi)$. Thus, it suffices to prove the statement for a small $M$-good neighborhood $\Omega=Z_M(F)\exp(\omega)\theta_n\subseteq \tilde{M}(F)$ of $\theta_n$. Here $\omega \subseteq \mathfrak{m}_{\theta_n,0}(F)$ is an $M$-excellent neighborhood of $0$. 

By taking $\omega$ small enough, we can assume that $\omega_N=\omega \cap \mathfrak{n}_{\theta_n,0}(F)$ is $N$-excellent and $\tilde{\omega}_{\psi,\mu}$
has a local character expansion on $\exp(\omega_N)\theta_{n}$.
Let $\tilde{f}\in \mathcal{C}_{\text{scusp}}(Z_M(F)\backslash \Omega,\chi)$ such that $\theta_{\tilde{f}}=\tilde{\theta}$. Denote $f_\omega(X)=f(\exp(X)\theta_n)$, for $X\in \omega$. For any $\lambda \in \mathcal{O}_F$, we set $\tilde{f}_\lambda(z\exp(X)\theta_n)=\tilde{f}(z\exp(\lambda^{-1}X)\theta_n)$, where $z\in Z_M(F)$ and $X\in \omega$. Let $\tilde{\theta}_\lambda=\theta_{\tilde{f}_\lambda}$. By \cite[Theorem 4.1]{Kon02},
the wavefront set of $\tilde{\omega}_{\psi,\mu}$ contains all minimal nilpotent
orbits of $\mathfrak{n}_{\theta_{n},0}\left(F\right)$ with opposite leading coefficients (if there are more than one minimal nilpotent orbit). For any $f\in C^\infty(\omega)$, we have
$$
\sum_{i}\int_{\mathfrak{n}_{\theta_n,0}(F)}f(X)\langle \phi_i,\tilde{\omega}_{\psi,\mu,\chi}(\exp(X)\theta_n)\phi_i\rangle dX
$$
$$
=c_{0}\int_{\mathfrak{n}_{\theta_{n},0}\left(F\right)}f\left(X\right)dX+c_{1}\cdot \left(\sum_{\mathcal{O}_i\in\text{Nil}_{\min}(\mathfrak{n}_{\theta_n,0})}(-1)^{i}\int_{\mathcal{O}}\widehat{f\mid_{\mathfrak{n}_{\theta_{n},0}\left(F\right)}}\left(X\right)d_{\mathcal{O}}X\right),
$$
noting that here we need to pin a minimal nilpotent orbit $\mathcal{O}_1$. We denote by $\Sigma$ the orthogonal complement of $\mathfrak{n}_{\theta_n,0}$ inside $\mathfrak{m}_{\theta_n,0}$. For each $\mathcal{O}\in \text{Nil}_{\min}(\mathfrak{n}_{\theta_n,0})$, we pick an element $N_\mathcal{O}\in \mathcal{O}$. Observe
$$
J_{\chi,\text{qc}}(\tilde{\theta})=\tilde{J}_{\chi}(\tilde{f})=\int_{Z_M(F)N_{\theta_n}(F)\backslash M(F)}\sum_{i}\int_{\mathfrak{n}_{\theta_n,0}(F)}{}^mf_\omega(X)\langle\phi_i,\tilde{\omega}_{\psi,\mu,\chi}(\exp(X)\theta_n)\phi_i\rangle dXdm
$$
$$
=c_0\cdot \int_{Z_M(F)N_{\theta_n}(F)\backslash M(F)}\int_{\mathfrak{n}_{\theta_n,0}(F)}{}^mf_\omega(X)dXdm
$$
$$
+c_1\cdot \left(\sum_{\mathcal{O}\in \text{Nil}_{\min}(\mathfrak{n}_{\theta_n,0})}D^{\tilde{N}}(S_{\mathcal{O}})^{1/2}\int_{Z_{M}\left(F\right)\left(N_{\theta_{n}}\right)_{S_{\mathcal{O}}}\left(F\right)\backslash M\left(F\right)}\int_{\Sigma\left(F\right)+S_{\mathcal{O}}}{}^m\hat{f}_\omega\left(X\right)d\mu_{\Sigma}Xdm\right).
$$
Substituting $\tilde{f}_\lambda$ to the above formula, we have
\begin{equation}\label{eqn8.2}
    J_{\chi,\text{qc}}(\tilde{\theta}_\lambda)=|\lambda|^{n^2-1}c_0\cdot \int_{Z_M(F)N_{\theta_n}(F)\backslash M(F)}\int_{\mathfrak{n}_{\theta_n,0}(F)}{}^mf_\omega(X)dXdm
\end{equation}
$$
+|\lambda|^{n^2-n}c_1\cdot \left( \sum_{\mathcal{O}\in \text{Nil}_{\min}(\mathfrak{n}_{\theta_n,0})}D^{\tilde{N}}(S_{\mathcal{O}})^{1/2}\int_{Z_{M}\left(F\right)\left(N_{\theta_{n}}\right)_{S_{\mathcal{O}}}\left(F\right)\backslash M\left(F\right)}\int_{\Sigma\left(F\right)+S_{\mathcal{O}}}{}^m\hat{f}_\omega\left(X\right)d\mu_{\Sigma}Xdm\right).
$$
We now consider $\epsilon_{\text{geom}}(\tilde{\theta})$. We set $\theta_\omega=\theta_{f_{\omega}}$ and $\theta_{\omega,\lambda}(X)=\theta_\omega(\lambda^{-1}X)$. Observe
$$
\epsilon_{\text{geom}}(\tilde{\theta})=\sum_{\tilde{T}\in \mathcal{T}_{\text{ell}}(\tilde{N})}\frac{\gamma_\psi(\tilde{T})}{|W(N,\tilde{T})|}\lim_{s\rightarrow 0^+}\int_{\tilde{T}(F)/\theta}D^{\tilde{M}}(\tilde{x})^{1/2}c_{\tilde{\theta}}(\tilde{x})\frac{\mu(\det(1-x^{-1}))}{|\det(1-x)|_E^{1/2-s}}d\tilde{x}
$$
$$
=\sum_{\tilde{T}\in \mathcal{T}_{\text{ell}}(\tilde{N})}\frac{\gamma_\psi(\tilde{T})}{|W(N,\tilde{T})|}\lim_{s\rightarrow 0^+}\int_{\mathfrak{t}_{\theta_n,0}(F)}D^{M_{\theta_n}}(X)^{1/2}c_{\theta_\omega}(X)\left(\int_{U(1)}\frac{\chi(z)\mu(\det(1-z^{-1}\exp(-X)))}{|\det(1-z^{-1}\exp(-X))|_E^{1/2-s}}dz\right)dX.
$$
We can shrink $\omega$ so that there exists $\epsilon >0$ sufficently small satisfying for any $X\in \omega$, we have 
$$
\int_{U(1)}\frac{\chi(z)\mu(\det(1-z^{-1}\exp(-X)))}{|\det(1-z^{-1}\exp(-X))|_E^{1/2-s}}dz
=\int_{U(1)_{<\epsilon}}\frac{\mu(\det((1-z^{-1})+z^{-1}X))}{|\det(1-z^{-1}+z^{-1}X)|_E^{1/2-s}}dz
$$
$$
+\int_{U(1)_{>\epsilon}}\frac{\chi(z)\mu(\det(1-z^{-1}))}{|\det(1-z^{-1})|_E^{1/2-s}}dz
$$
$$
=\int_{U(1)_{<\epsilon}}\frac{\mu(\det((1-z^{-1})(1-X)+X))}{|\det((1-z^{-1})(1-X)+X)|_E^{1/2-s}}dz+\int_{U(1)_{>\epsilon}}\frac{\chi(z)\mu(\det(1-z^{-1}))}{|\det(1-z^{-1})|_E^{1/2-s}}dz
$$
$$
=\text{Vol}(\text{U}(1)_{<\epsilon})\cdot \frac{\mu(\det(X))}{|\det(X)|_E^{1/2-s}}+\int_{U(1)_{>\epsilon}}\frac{\chi(z)\mu(\det(1-z^{-1}))}{|\det(1-z^{-1})|_E^{1/2-s}}dz.
$$
Then
$$
\epsilon_{\text{geom}}(\tilde{\theta})=\text{Vol}(\text{U}(1)_{<\epsilon})\cdot\sum_{\tilde{T}\in \mathcal{T}_{\text{ell}}(\tilde{N})}\frac{\gamma_\psi(\tilde{T})}{|W(N,\tilde{T})|}\lim_{s\rightarrow 0^+}\int_{\mathfrak{t}_{\theta_n,0}(F)}D^{M_{\theta_n}}(X)^{1/2}c_{\theta_\omega}(X)\frac{\mu(\det(X))}{|\det(X)|_E^{1/2-s}}dX
$$
$$
+\left(\int_{U(1)_{>\epsilon}}\frac{\chi(z)\mu(\det(1-z^{-1}))}{|\det(1-z^{-1})|_E^{1/2-s}}dz\right)\cdot \sum_{\tilde{T}\in \mathcal{T}_{\text{ell}}(\tilde{N})}\frac{\gamma_\psi(\tilde{T})}{|W(N,\tilde{T})|}\lim_{s\rightarrow 0^+}\int_{\mathfrak{t}_{\theta_n,0}(F)}D^{M_{\theta_n}}(X)^{1/2}c_{\theta_\omega}(X)dX.
$$
Substituting $\tilde{\theta}_\lambda$ to the above formula, we have
$$
\epsilon_{\text{geom}}(\tilde{\theta}_\lambda)=|\lambda|^{n^2-n}\text{Vol}(\text{U}(1)_{<\epsilon})\cdot\sum_{\tilde{T}\in \mathcal{T}_{\text{ell}}(\tilde{N})}\frac{\gamma_\psi(\tilde{T})}{|W(N,\tilde{T})|}\lim_{s\rightarrow 0^+}\int_{\mathfrak{t}_{\theta_n,0}(F)}D^{M_{\theta_n}}(X)^{1/2}c_{\theta_\omega}(X)\frac{\omega_{E/F}(-1)^n\mu(\det(X))}{|\det(X)|_E^{1/2-s}}dX
$$
\begin{equation}\label{eqn8.3}
    +|\lambda|^{n^2-1}\cdot\left(\int_{U(1)_{>\epsilon}}\frac{\chi(z)\mu(\det(1-z^{-1}))}{|\det(1-z^{-1})|_E^{1/2-s}}dz\right)\cdot \sum_{\tilde{T}\in \mathcal{T}_{\text{ell}}(\tilde{N})}\frac{\gamma_\psi(\tilde{T})}{|W(N,\tilde{T})|}\lim_{s\rightarrow 0^+}\int_{\mathfrak{t}_{\theta_n,0}(F)}D^{M_{\theta_n}}(X)^{1/2}c_{\theta_\omega}(X)dX.
\end{equation}
By equations (\ref{eqn8.4}), (\ref{eqn8.2}) and (\ref{eqn8.3}), together with the fact that $c_{\tilde{\theta}_\lambda,\mathcal{O}}(\theta_n)=|\lambda|^{\frac{\dim\mathcal{O}}{2}}c_{\tilde{\theta},\mathcal{O}_\lambda}(\theta_n)$, for any $\mathcal{O}\in \text{Nil}(\mathfrak{n}_{\theta_n,0})$, we obtain
$$
J_{\chi,\text{qc}}(\tilde{\theta})=\epsilon_{\text{geom}}(\tilde{\theta})+\sum_{\mathcal{O}\in \text{Nil}_{\text{reg}}(M_{\theta_n})}c_{\chi,\mathcal{O}}\cdot c_{\tilde{\theta},\mathcal{O}}(\theta_n),
$$
for any $\tilde{\theta}\in \text{QC}(Z_M(F)\backslash \Omega,\chi)$. When $n$ is even, there is only one regular nilpotent orbit, so we establish the desired answer. When $n$ is odd, we choose $\lambda \in \mathcal{O}_F\setminus \text{Nm}_{E/F}(E)$ to exchange orbits in $\text{Nil}_{\text{reg}}(M_{\theta_n})$. This gives us $c_{\chi,\mathcal{O}_1}=-c_{\chi,\mathcal{O}_1}$. Therefore, in this case, there exists $c_\chi\in \mathbb{C}$ such that 
$$
J_{\chi,\text{qc}}(\tilde{\theta})=c_\chi\cdot (c_{\tilde{\theta},\mathcal{O}_1}(\theta_n)-c_{\tilde{\theta},\mathcal{O}_2}(\theta_n))+\epsilon_{\text{geom}}(\tilde{\theta}),
$$
for any $\tilde{\theta}\in \text{QC}(Z_M(F)\backslash \Omega,\chi)$.
\end{proof}

\subsection{End of the proof of Theorem \ref{twistedgeometriclinear}}
We now finish our proof for Theorem \ref{twistedgeometriclinear} when $K\neq E$. 
\begin{proof}
By Proposition \ref{descentLietwisted}, for any irreducible elliptic tempered representation $\tilde{\pi}$ of $\tilde{M}(F)$ whose central character is $\chi$, we have
\begin{equation}\label{eqn8.5}
    \epsilon_{\psi}(\tilde{\pi})=c_\chi\cdot c_{\theta_{\tilde{\pi}}}(\theta_n)+\epsilon_{\text{geom}}(\tilde{\pi})
\end{equation}
when $n$ is even or
\begin{equation}\label{eqn8.55}
    \epsilon_{\psi}(\tilde{\pi})=c_\chi\cdot (c_{\theta_{\tilde{\pi}},\mathcal{O}_1}(\theta_n)-c_{\theta_{\tilde{\pi}},\mathcal{O}_2}(\theta_n))+\epsilon_{\text{geom}}(\tilde{\pi})
\end{equation}
when $n$ is odd. Let $M$ be a tempered $L$-parameter for $G(F)$ such that $M$ is of the form
$$
M=M_1+\ldots +M_n,
$$
where $M_i$ is one-dimensional and conjugate self-dual of parity $(-1)^{n-1}$ for all $i=\overline{1,n}$. Moreover, we can choose $M$ so that the central character of $M$ is $\chi$ and $M_i\neq M_j$, for any $i\neq j$. In this case, $M$ is an elliptic $L$-parameter. We set $\theta_M=\sum_{\pi \in \Pi_M}\theta_\pi$. Let $\Pi$ be the representation of $M(F)$ corresponding to $M$ and $\tilde{\Pi}$ be its extension to $\tilde{M}(F)$. By the twisted endoscopic relation mentioned in section \ref{llc7}, it follows that $D^{\tilde{M}}(\tilde{x})^{1/2}\theta_{\tilde{\Pi}}(\tilde{x})=D^G(y)^{1/2}\theta_M(y)$ if $\tilde{x}$ and $y$ are matched. By the main theorem in \cite{CG22}, observe 
$$
\epsilon_{\psi}(\tilde{\Pi})=\sum_{V}\mu(\det V)\cdot \left(\sum_{\pi\in \Pi_M}m_{V}(\pi)\right).
$$
The geometric multiplicity formula in Theorem \ref{maintheorem1} gives us 
$$
\epsilon_{\psi}(\tilde{\Pi})
=\underset{T\in{\mathcal{T}}^{\text{stab}}_{\text{ell}}\left(H_V\right)}{\sum}\frac{\gamma_\psi(T)|p_{V,\text{stab}}^{-1}(T)|}{\left|W\left(H_V,T\right)\right|}
\underset{s\rightarrow0^{+}}{\lim}\int_{T\left(F\right)}D^{G}\left(x\right)^{1/2}\theta_{M}\left(x\right)\frac{\mu\left(\det\left(1-x^{-1}\right)\right)}{\left|\det\left(1-x\right)\right|_{E}^{1/2-s}}dx.
$$
Since there is a bijection between $\mathcal{T}^{\text{stab}}_{\text{ell}}\left(H_V\right)$ and $\mathcal{T}^{\text{stab}}_{\text{ell}}(\tilde{N})$, it follows that
\begin{equation}\label{eqn8.6}
    \epsilon_{\psi}(\tilde{\Pi})=\underset{\tilde{T}\in{\mathcal{T}}_{\text{ell}}\left(\tilde{N}\right)}{\sum}\frac{\gamma_\psi(\tilde{T})}{\left|W\left(N,\tilde{T}\right)\right|}
\underset{s\rightarrow0^{+}}{\lim}\int_{\tilde{T}\left(F\right)/\theta_n}D^{\tilde{M}}\left(\tilde{x}\right)^{1/2}\theta_{\tilde{\Pi}}\left(\tilde{x}\right)\frac{\mu\left(\det\left(1-x^{-1}\right)\right)}{\left|\det\left(1-x\right)\right|_{E}^{1/2-s}}d\tilde{x}=\epsilon_{\text{geom}}(\tilde{\Pi}).
\end{equation}
From the equations (\ref{eqn8.5}), (\ref{eqn8.55}) and (\ref{eqn8.6}), we can see that $c_\chi=0$, thus finish our proof for Theorem \ref{twistedgeometriclinear}.
\end{proof}

\subsection{Proof of Theorem \ref{maintheorem}(ii)}
We now give our proof for part 2 of Theorem \ref{maintheorem}.
\begin{proof}
Let $M$ be a tempered $L$-parameter for $G(F)$. By Theorem \ref{maintheorem1}, observe
\begin{equation}\label{eqn8.7}
    \sum_{V}\mu(\det V)\cdot \left(\sum_{\pi\in \Pi_M}m_{V}(\pi)\right)=\underset{T\in{\mathcal{T}}^{\text{stab}}_{\text{ell}}\left(H_V\right)}{\sum}\frac{\gamma_\psi(T)|p_{V,\text{stab}}^{-1}(T)|}{\left|W\left(H_V,T\right)\right|}
\end{equation}
$$
\underset{s\rightarrow0^{+}}{\lim}\int_{T\left(F\right)}D^{G}\left(x\right)^{1/2}\theta_{M}\left(x\right)\frac{\mu\left(\det\left(1-x^{-1}\right)\right)}{\left|\det\left(1-x\right)\right|_{E}^{1/2-s}}dx.
$$
Let $\Pi$ be the representation of $M(F)$ corresponding to $M$ and $\tilde{\Pi}$ be its extension to $\tilde{M}(F)$. By Theorem \ref{geometrictwisted}, we have
\begin{equation}\label{eqn8.8}
        \epsilon_{\psi}(\tilde{\Pi})=\underset{\tilde{T}\in{\mathcal{T}}^{\text{stab}}_{\text{ell}}\left(\tilde{N}\right)}{\sum}\frac{\gamma_\psi(\tilde{T})|p_{\tilde{N},\text{stab}}^{-1}(\tilde{T})|}{\left|W\left(N,\tilde{T}\right)\right|}
\underset{s\rightarrow0^{+}}{\lim}\int_{\tilde{T}\left(F\right)/\theta_n}D^{\tilde{M}}\left(\tilde{x}\right)^{1/2}\theta_{\tilde{\Pi}}\left(\tilde{x}\right)\frac{\mu\left(\det\left(1-x^{-1}\right)\right)}{\left|\det\left(1-x\right)\right|_{E}^{1/2-s}}d\tilde{x}.
\end{equation}
Using equations (\ref{eqn8.7}) and (\ref{eqn8.8}), together with the matching of orbits described in Section \ref{llc4}, it follows that
$$
\epsilon_{\psi}(\tilde{\Pi})=\sum_{V}\mu(\det V)\cdot \left(\sum_{\pi\in \Pi_M}m_{V}(\pi)\right),
$$
which confirms part 2 of Theorem \ref{maintheorem}.
\end{proof}

\section{Endoscopic transfers and the twisted Gan-Gross-Prasad conjecture}\label{sec9}
\subsection{Endoscopic transfers of the linear form $m_{V,\text{geom}}$}
Recall the parametrization of conjugacy classes in unitary groups in Section \ref{llc4}. Let $d$ be a positive integer and $(G_d,H_d)=(\text{Res}_{K/F}\,U_{L/K}(d),U_{E/F}(d))$, where $U_{E/F}(d)$ is a quasi-split unitary group of rank $d$ and $U_{L/K}(d)=U_{E/F}(d)(K)$. For $\xi \in \Xi^*_{d,\text{reg}}$, we set $\gamma_\psi=\gamma_\psi(T_\xi)$. Let $\theta_d$ be a stably invariant quasi-character of $G_d(F)$, which can be viewed as a function on $\Xi^*_{d,\text{reg}}$. We fix a character $\mu$ of $E^\times$ whose restriction to $F^\times$ is $\omega_{E/F}$. For any character $\mu_d$ of $E^\times$ which is trivial on $F^\times$, we set
$$
S_{\mu_d}(\theta_d)=\lim_{s\rightarrow 0^+}\int_{\Xi_{d,\text{reg}}^*}\mu_d(P_\xi(1))|C(\xi)|D^d(\xi)\theta_{d}(\xi)\mu(P_{\xi^{-1}}(1))\Delta(\xi)^{-1/2+s}d\xi.
$$
We now define an elliptic endoscopic datum for $G$. Let $\left(V_{+},h_{+}\right)$
and $\left(V_{-},h_{-}\right)$ be quasi-split skew-hermitian spaces
over $L$ of dimension $n_{+}$ and $n_{-}$ respectively, where $n_{+}+n_{-}=n$.
We denote by $G_{+}$ and $G_{-}$ the unitary groups corresponding to $\left(V_{+},h_{+}\right)$ and $\left(V_{-},h_{-}\right)$. Let
$\mu_{+}$ and $\mu_{-}$ be continuous characters of $E^{\times}$
such that 
\[
\mu_{+\mid F^{\times}}=\omega_{E/F}^{n_{-}}\ \ \text{ and }\ \ \mu_{-\mid F^{\times}}=\omega_{E/F}^{n_{+}}.
\]
We set $\mu_+^K=\mu_+\circ \text{Nm}_{L/E}$ and $\mu_-^K=\mu_-\circ \text{Nm}_{L/E}$. Then $\left(G_{+}\times G_{-},\mu_{+}^K,\mu^K_{-}\right)$ determines
an elliptic endoscopic datum of $G$. For any stably invariant quasi-character $\theta=\theta_+ \otimes \theta_-\in\text{QC}_{\text{stab}}\left(G_{+}\times G_{-}\right)$,
we denote by $\theta_{\mu_{+},\mu_{-}}^{G}$ its endoscopic transfer
to $\text{QC}\left(G\right)$ via $\left(G_{+}\times G_{-},\mu^K_{+},\mu^K_{-}\right)$,
i.e. 
\[
D^{G}\left(x\right)^{1/2}\theta_{\mu_{+},\mu_{-}}^{G}\left(x\right)=\underset{y}{\sum}D^{G_{+}\times G_{-}}\left(y\right)^{1/2}\theta\left(y\right)\Delta_{\mu^K_{+},\mu^K_{-}}\left(y,x\right),\text{ for any }x\in G_{\text{reg}}\left(F\right).
\]
We compute $m_{V,\text{geom}}$ via an endoscopic transfer.
\begin{proposition}\label{maintheorem2}
    For any stably invariant quasi-character $\theta=\theta_+\otimes \theta_-$ on $G_{+}(F)\times G_{-}(F)$, we have
    \begin{equation}\label{eqn9.1}
        \sum_{V}\mu(\det V)\cdot m_{V,\text{geom}}(\theta_{\mu_{+},\mu_{-}}^{G})=S_{\mu_+^2}(\theta_+)S_{\mu_-^2}(\theta_-),
    \end{equation}
    where the above sum runs over over the equivalence classes of $n$-dimensional skew-Hermitian structures on $V$.
\end{proposition}
\begin{proof}
    We denote by $\Xi_{n,\text{reg}}^{K,*}$ the space parametrizing stable conjugacy classes of elliptic elements in $G(F)$. Let $(\xi,c) = (I,(F_{\pm i})_{i\in I},(F_{i})_{i\in I},(y_{i})_{i\in I},(c_{i})_{i\in I})$, where $\xi \in \Xi_{n,\text{reg}}^*$ and $c\in C(\xi)$. We now determine its image in $\Xi_{n,\text{reg}}^{K,*}$ precisely. We denote by $I^1$ the index set containing $i$ such that $F_{\pm i}$ does not contain $K$ and $I^2=I\setminus I^1$. For each $i\in I^1$, we set $K_{\pm i}=F_{\pm i}\otimes_F K$ and $K_{i}=F_{i}\otimes_F K$. For each $i\in I^2$, we set $K_{\pm i}^1=K_{\pm i}^2=F_{\pm i}$ and $K_{i}^1=K_{i}^2=F_{i}$. Denote $I^K=I^1 \sqcup I^2 \sqcup I^2$ and 
    $$
    \xi^K=(I^1,(K_{\pm i})_{i\in I^1},(K_{i})_{i\in I^1},(y_{i})_{i\in I^1}) \sqcup (I^2\sqcup I^2,(K_{\pm i}^1,K_{\pm i}^2)_{i\in I^2},(K_{i}^1,K_{i}^2)_{i\in I^2},(y_{i},y_i^\tau)_{i\in I^2}),
    $$
    here $\tau$ is the nontrivial element in $\text{Gal}(K/F)$. Moreover, we can identify $C(\xi^K)^1$ with $C(\xi)$.
    
    The left hand side of (\ref{eqn9.1}) is equal to 
    $$
    \sum_{V}\sum_{T\in \mathcal{T}_{\text{ell}}(H_V)}\frac{\gamma_\psi(T)}{\left|W\left(H_V,T\right)\right|}
\underset{s\rightarrow0^{+}}{\lim}\int_{T\left(F\right)}D^{G}\left(x\right)^{1/2}\theta_{\mu_{+},\mu_{-}}^{G}\left(x\right)\frac{\mu\left(\det\left(1-x^{-1}\right)\right)}{\left|\det\left(1-x\right)\right|_{E}^{1/2-s}}dx.
    $$
    $$
    =\lim_{s\rightarrow 0^+}\int_{\Xi^*_{n,\text{reg}}}\gamma_\psi(\xi)\left(\sum_{c\in C(\xi)}D^{n}(\xi)\theta_{\mu_{+},\mu_{-}}^{G}\left(x(\xi^K,c)\right)\right)\mu(P_{\xi^{-1}}(1))|\Delta(\xi)|^{1/2-s}d\xi.
    $$
    We would like to compute $\sum_{c\in C(\xi)}D^{n}(\xi)\theta_{\mu_{+},\mu_{-}}^{G}\left(x(\xi^K,c)\right)$. Observe
    $$
    D^{n}(\xi)\theta_{\mu_{+},\mu_{-}}^{G}(x(\xi^K,c))=\sum_{I_1,I_2}D^{n_+}(\xi(I_1))\theta_+(\xi(I_1)^K)D^{n_-}(\xi(I_2))\theta_-(\xi(I_2)^K)\Delta_{\mu_+,\mu_-}(\xi(I_1)^K,\xi(I_2)^K,c),
    $$
    where the above sum runs over pairs $(I_1,I_2)$ such that $I^K=I_1 \sqcup I_2$, $d_{I_1}=n_+$ and $d_{I_2}=n_-$. We consider two cases.

    \textbf{Case 1:} Assume that for any $i\in I^2$, either $I_1$ or $I_2$ contain both copies of $i$. In this case, we have
    $$
    \Delta_{\mu_+,\mu_-}(\xi(I_1)^K,\xi(I_2)^K,c)=\Delta_{\mu_+,\mu_-}(\xi(I_1),\xi(I_2),c)^2=\mu_+(P_{\xi(I_1)^K}(-1))\mu_-(P_{\xi(I_2)^K}(-1))
    $$
    $$
    =\mu_+(P_{\xi(I_1)}(-1))^2\mu_-(P_{\xi(I_2)}(-1))^2=\mu_+(P_{\xi(I_1)}(1))^2\mu_-(P_{\xi(I_2)}(1))^2.
    $$

    \textbf{Case 2:} Assume that there exists $i\in I^2$ such that both $I_1$ and $I_2$ contain one copy of it. Let $c^\prime=(c^\prime_j)_{j\in I}$ such that $c^\prime_i=-c_i$ and $c^\prime_j=c_j$ for any $j \neq i$. We have
    $$
    \Delta_{\mu_+,\mu_-}(\xi(I_1)^K,\xi(I_2)^K,c)=-\Delta_{\mu_+,\mu_-}(\xi(I_1)^K,\xi(I_2)^K,c^\prime).
    $$
    Therefore
    $$
    \sum_{c\in C(\xi)}D^{n}(\xi)\theta_{\mu_{+},\mu_{-}}^{G}\left(x(\xi^K,c)\right)=|C(\xi)|\sum_{(I_1,I_2)\text{ in case 1}}D^{n_+}(\xi(I_1))\theta_+(\xi(I_1)^K)D^{n_-}(\xi(I_2))\theta_-(\xi(I_2)^K)
    $$
    $$
    \mu_+(P_{\xi(I_1)}(1))^2\mu_-(P_{\xi(I_2)}(-1))^2.
    $$
    This gives us 
    $$
    \sum_{V}\mu(\det V)\cdot m_{V,\text{geom}}(\theta_{\mu_{+},\mu_{-}}^{G})=\left(\lim_{s\rightarrow 0^+}\int_{\Xi_{n_+,\text{reg}}^*}\mu_+(P_\xi(1))^2|C(\xi)|D^{n_+}(\xi)\theta_{+}(\xi)\mu(P_{\xi^{-1}}(1))\Delta(\xi)^{-1/2+s}d\xi\right)
    $$
    $$
    \cdot \left(\lim_{s\rightarrow 0^+}\int_{\Xi_{n_-,\text{reg}}^*}\mu_-(P_\xi(1))^2|C(\xi)|D^{n_-}(\xi)\theta_{-}(\xi)\mu(P_{\xi^{-1}}(1))\Delta(\xi)^{-1/2+s}d\xi\right)=S_{\mu_+^2}(\theta_+)S_{\mu_-^2}(\theta_-)
    $$
    as desired.
\end{proof}

\subsection{Proof of Theorem \ref{maintheorem}(iii)}
In this subsection, we give a proof for part (iii) of Theorem \ref{maintheorem}.
\begin{proof}
    Let $M$ be a tempered $L$-parameter of $G(F)$. It suffices to show that for any $s\in A_M$, we have
    \begin{equation}\label{9.1}
        \sum_{V}\sum_{\chi\in \hat{A}_M}\mu(\det V)\chi(s)m_{V}(\pi(M,\chi))=\epsilon\left(\frac{1}{2},[\text{As}(M^s)+\text{As}(M^{-s})]\cdot \mu^{-1},\psi_{E,e}\right)
    \end{equation}
    $$
    \cdot \omega_{E/F}(-1)^n\omega_{K/F}(k^2)^{n(n-1)/2},
    $$
    here $M^{-s}=M/M^s$ and $k$ is a nonzero $\text{Tr}_{K/F}$-zero element in $K$. We denote $n_+=\dim M^s$ and $n_-=\dim M^{-s}$. Let $G_+\times G_-$ be the endoscopic group corresponding to $s$. We choose continuous characters $\mu_+$ and $\mu_-$ of $E^\times$ so that $(G_+\times G_-,\mu^K_+,\mu^K_-)$ forms an elliptic endoscopic datum. We set $\theta_{M,s}=\sum_{\chi\in \hat{A}_M}\chi(s)\theta_{\pi(M,\chi)}$. Then $\gamma_{\mu^K_+,\mu^K_-}^G(\mu_+^{K,-1}M^s,\mu_-^{K,-1}M^{-s})\theta_{M,s}$ is the transfer of $\theta_{\mu_+^{K,-1}M^s}\otimes \theta_{\mu_-^{K,-1}M^{-s}}$ via the endoscopic datum $(G_+\times G_-,\mu^K_+,\mu^K_-)$. By \cite{BP16}[Proposition 8.4.1], observe 
    $$\gamma_{\mu^K_+,\mu^K_-}^G(\mu_+^{K,-1}M^s,\mu_-^{K,-1}M^{-s})=\gamma_{\psi_K}(\text{Nm}_{L/K})^{-n_+n_-}\omega_{L/K}(-2)^{n_+n_-}=\omega_{E/F}(k^2)^{n_+n_-}.
    $$ 
    By Theorem \ref{maintheorem1} and Proposition \ref{maintheorem2}, we have
    $$
    \sum_{V}\sum_{\chi\in \hat{A}_M}\mu(\det V)\chi(s)m_{V}(\pi(M,\chi))=\sum_{V}\mu(\det V)m_{V,\text{geom}}(\theta_{M,s})
    $$
    $$
    =\omega_{E/F}(k^2)^{n_+n_-}S_{\mu_+^2}(\theta_{\mu_+^{K,-1}M^s})S_{\mu_-^2}(\theta_{\mu_-^{K,-1}M^{-s}}).
    $$
    By Theorem \ref{geometrictwisted}, it follows that
    $$
    \epsilon\left(\frac{1}{2},\text{As}_{L/E}(M^s)\times \mu^{-1},\psi_{E,e}\right)\omega_{E/F}(-1)^{n_+}\omega_{E/F}(k^2)^{n_+(n_+-1)/2}=S_{\mu_+^2}(\theta_{\mu_+^{K,-1}M^s})
    $$
    and
    $$
    \epsilon\left(\frac{1}{2},\text{As}_{L/E}(M^{-s})\times \mu^{-1},\psi_{E,e}\right)\omega_{E/F}(-1)^{n_-}\omega_{E/F}(k^2)^{n_-(n_--1)/2}=S_{\mu_-^2}(\theta_{\mu_-^{K,-1}M^{-s}}).
    $$
    Therefore
    $$
    \sum_{V}\sum_{\chi\in \hat{A}_M}\mu(\det V)\chi(s)m_{V}(\pi(M,\chi))=\epsilon\left(\frac{1}{2},[\text{As}_{L/E}(M^s)+\text{As}_{L/E}(M^{-s})]\times \mu^{-1},\psi_{E,e}\right)
    $$
    $$
    \cdot\omega_{E/F}(-1)^{n_+}\omega_{E/F}(k^2)^{n_+(n_+-1)/2}\omega_{E/F}(-1)^{n_-}\omega_{E/F}(k^2)^{n_-(n_--1)/2}\omega_{E/F}(k^2)^{n_+n_-}
    $$
    $$
    =\epsilon\left(\frac{1}{2},[\text{As}_{L/E}(M^s)+\text{As}_{L/E}(M^{-s})]\times \mu^{-1},\psi_{E,e}\right)\omega_{E/F}(-1)^{n}\omega_{E/F}(k^2)^{n(n-1)/2}.
    $$
    We have finished our proof for Theorem \ref{maintheorem}(iii).
\end{proof}

\appendix

\section{The local Gan-Gross-Prasad conjecture for unitary groups}\label{appendixB}

In this section, combining the theta correspondence arguments from \cite{Xue23a, Xue23b} with the local trace formula approach, we present an alternative proof for the tempered case of the local Gan–Gross–Prasad conjecture for unitary groups over non-archimedean fields of odd residual characteristic. Our contributions include a simplified proof of the geometric side of the local trace formula and its twisted variant developed in \cite{BP14,BP20} for Bessel models, as well as an independent proof of the tempered case for Fourier–Jacobi models that does not rely on the results of \cite{GI16} (i.e. the refined statement of Prasad's conjectures).

\subsection{The local Gan-Gross-Prasad conjecture}

In this subsection, we revisit precise statements of the local Gan-Gross-Prasad conjecture for unitary groups. Let $F$  be a nonarchimedean field of odd residual characteristic and $E$  be a quadratic field extension. We first recall Bessel models. Let $r$ be a nonnegative integer. Let $V_n\subset V_{n+2r+1}$ be a pair of hermitian spaces relative to $E/F$ of dimensions $n$ and $n+2r+1$. The pair $(V_n,V_{n+2r+1})$ is relevant if there exists a subspace $Z_{2r+1}\subseteq V_{n+2r+1}$ such that $V_{n+2r+1}=V_{n}\oplus^\perp Z_{2r+1}$ , where  $Z_{2r+1}=\langle z_{\pm i} \rangle_{i=\overline{0,r}} $  and $h_V(z_i,z_j)=(-1)^n\delta_{i,-j}$, for all $i,j\in[-r,r]$.

Let $P$ be the parabolic subgroup of $U(V_{n+2r+1})$ stabilizing the flag
$$
\langle z_r\rangle \subseteq \langle z_r,z_{r-1}\rangle \subseteq \ldots \subseteq \langle z_r,\ldots, z_1\rangle,
$$
and $N$ be its unipotent radical. We set $H=U(V_n)\ltimes N$  and $G=U(V_{n+2r+1})\times U(V_n)$. Let $\psi$  be a nontrivial additive character of $F$. We define a character $\xi:N(F)\rightarrow \mathbb{C}^\times$  by 
$$
\xi(n)=\psi(\text{Tr}_{E/F}(\sum_{i=0}^r h_V(z_{-i-1},nz_i))),\text{ for } n\in N(F).
$$
We extend it trivially to a character of $H(F)$, which we also denote by $\xi$. Let $\pi$ be a representation of $U(V_{n+2r+1})$  and $\sigma$ be a representation of $U(V_n)$. We set
$$
m_B(\pi,\sigma)=\dim \text{Hom}_{H}(\pi\otimes \sigma,\xi).
$$
This branching problem is called a Bessel model. Let $\phi$  and $\phi^\prime$ be $L$-parameters for $U(V_{n+2r+1})$  and $U(V_n)$. We have component groups
$$
S_{\phi}=\prod_i (\mathbb{Z}/2\mathbb{Z})a_i\ \ \ \text{ and }\ \ \ S_{\phi^\prime}=\prod_j (\mathbb{Z}/2\mathbb{Z})b_j.
$$
We define a distinguished character $\eta_B$  of  $S_\phi\times S_{\phi^\prime}$ via
$$
\eta_B(a_i)=\epsilon(1/2,\phi_i\otimes \phi^\prime,\psi_{-2}^E)\ \ \ \text{ and }\ \ \ \eta_B(b_j)=\epsilon(1/2,\phi\otimes \phi_j^\prime,\psi_{-2}^E),
$$
for any $a_i$ and $b_j$, where $\psi^E_{-2}=\psi(-2\text{Tr}_{E/F}(\cdot))$.

We now recall Fourier-Jacobi models. Let $W_n\subset W_{n+2r}$ be a pair skew-hermitian spaces relative to $E/F$  of dimension $n$ and $n+2r$. We say $(W_n,W_{n+2r})$ is relevant if there exists a $2r$-dimensional split skew-hermitian subspace $Z_{2r}\subseteq W_{n+2r}$ such that $W_{n+2r}=W_n\oplus^\perp Z_{2r}$. We fix a basis $\{z_{\pm 1},\ldots , z_{\pm r}\}$ of $Z_{2r}$ such that $h_W(z_i,z_j)=\delta_{i,-j}$, for all $u,j=\overline{\pm 1, \pm r}$. 

Let $U$  be the unipotent radical of the parabolic subgroup of $U(W_{n+2r})$ stabilizing
$$
\langle z_r\rangle \subseteq \langle z_r,z_{r-1}\rangle \subseteq \ldots \subseteq \langle z_r,\ldots, z_1\rangle.
$$
In this setting, we denote  $H=U(W_n)\ltimes U$  and $G=U(W_{n+2r})\times U(W_{n})$. We define a character $\nu:U(F)\rightarrow \mathbb{C}^\times$  by 
$$
\nu(u)=\psi(-\text{Tr}_{E/F}(\sum_{i=0}^{r-1} h_W(z_{-i-1},uz_i))),\text{ for } u\in U(F).
$$
We extend it trivially to a character of $H(F)$, which we also denote by $\nu$. Let $\mu$  be a conjugate-symplectic character of $E^\times$ and $\omega_{W_n,\psi,\mu}$  be the Weil representation of $U(W_n)$ associated to $\psi$  and $\mu$ . Let $\pi$ be a representation of $U(W_{n+2r})$  and $\sigma$ be a representation of $U(W_n)$. We set
$$
m_{FJ}(\pi,\sigma)=\dim \text{Hom}_{H}(\pi\otimes \sigma,\omega_{V,\psi,\mu}\otimes \nu).
$$
The above branching problem is called a Fourier-Jacobi model. Let $\phi$  and $\phi^\prime$ be $L$-parameters for $U(V_{n+2r+1})$  and $U(V_n)$. We have component groups
$$
S_{\phi}=\prod_i (\mathbb{Z}/2\mathbb{Z})a_i\ \ \ \text{ and }\ \ \ S_{\phi^\prime}=\prod_j (\mathbb{Z}/2\mathbb{Z})b_j.
$$
Similar to the setting of Bessel models, we define a distinguished character $\eta_{FJ}$  of  $S_\phi\times S_{\phi^\prime}$, which depends on the parity of $n$ in this case. 
\begin{itemize}
    \item When $n$  is odd, we set  
$$
\eta_{FJ}(a_i)=\epsilon(1/2,\phi_i\otimes \phi^\prime\otimes \mu^{-1},\psi_{2}^E)\ \ \ \text{ and }\ \ \ \eta_{FJ}(b_j)=\epsilon(1/2,\phi\otimes \phi_j^\prime\otimes \mu^{-1},\psi_{2}^E),
$$
for any $a_i$ and $b_j$, where $\psi^E_{2}=\psi(2\text{Tr}_{E/F}(\delta\cdot))$ and $\delta$ is a nonzero trace-$0$ element in $E$.
\item When $n$ is even, we set $$
\eta_{FJ}(a_i)=\epsilon(1/2,\phi_i\otimes \phi^\prime\otimes \mu^{-1},\psi^E)\ \ \ \text{ and }\ \ \ \eta_{FJ}(b_j)=\epsilon(1/2,\phi\otimes \phi_j^\prime\otimes \mu^{-1},\psi^E),
$$
for any $a_i$ and $b_j$.
\end{itemize}
In Section \ref{secB4} and \ref{secB5}, we prove the following theorem by using induction on $n+r$.
\begin{theorem}\label{maintheoremB}
    Let the notation be as above. We consider two following statements.
    \begin{itemize}
        \item $\mathbf{(B)_{n,r}}:$  Let $\phi\times \phi^\prime$  be a tempered $L$-parameter for  $G=U(V_{n+2r+1})\times U(V_n)$. For any representation $\pi(\eta)\otimes \sigma(\eta^\prime)\in \Pi_\phi\times \Pi_{\phi^\prime}$  of a relevant pure inner form of $G$, we have
    $$
    m_{B}(\pi(\eta),\sigma(\eta^\prime))\neq 0 \Leftrightarrow \eta\times \eta^\prime = \eta_B.
    $$
    \item $\mathbf{(FJ)_{n,r}}:$  Let $\phi\times \phi^\prime$  be a tempered $L$-parameter for  $G=U(W_{n+2r})\times U(W_n)$. For any representation $\pi(\eta)\otimes \sigma(\eta^\prime)\in \Pi_\phi\times \Pi_{\phi^\prime}$  of a relevant pure inner form of $G$, we have
   $$
    m_{FJ}(\pi(\eta),\sigma(\eta^\prime))\neq 0 \Leftrightarrow \eta\times \eta^\prime = \eta_{FJ}.
    $$
    \end{itemize}
\end{theorem}

\subsection{The local theta correspondence}
In this subsection, we recall the local theta correspondence for unitary groups of (almost) equal rank. Let $V$  be a hermitian space and $W$  be a skew-hermitian space relative to $E/F$. We consider the local theta correspondence for the reductive dual pair $U(V)\times U(W)$. We fix the following data:
\begin{itemize}
    \item a nontrivial additive character $\psi$ of $F$;
    \item a pair of characters $(\mu_V,\mu_W)$  of $E^\times$ whose restriction to $F^\times$  is $(\omega_{E/F}^{\dim V},\omega_{E/F}^{\dim W})$;
    \item a trace-$0$ element $\delta$ in $E^\times$.
\end{itemize}
We can fix a conjugate-symplectic character $\mu$ of $E^\times$ and set $\mu_V=\mu^{\dim V}$ and $\mu_W=\mu^{\dim W}$. This gives us a natural map
$$
U(V)\times U(W) \rightarrow Sp(V\otimes W).
$$
We have a Weil representation $\omega_\psi$ of the metaplectic cover $Mp(V\otimes W)$. The data $(\mu_V,\mu_W)$ defines a splitting over $U(V)\times U(W)$, thus gives us a Weil representation $\omega_{\psi,\mu_V,\mu_W,V,W}$ of $U(V)\times U(W)$. For any $\pi\in \text{Irr}(U(W))$, we define
$$
\Theta_{\psi,\mu_V,\mu_W,V,W}(\pi)=(\omega_{\psi,\mu_V,\mu_W,V,W}\otimes\pi^\vee)_{U(W)}
$$
as a representation of $U(V)$ of finite length. By the Howe duality proved in \cite{Wal90,GT16}, the maximal semisimple quotient $\theta_{\psi,\mu_V,\mu_W,V,W}(\pi)$ of $\Theta_{\psi,\mu_V,\mu_W,V,W}(\pi)$ is either zero or irreducible.

Likewise, for each $\sigma \in  \text{Irr}(U(V))$, we set 
$$
\Theta_{\psi,\mu_V,\mu_W,W,V}(\sigma)=(\omega_{\psi,\mu_V,\mu_W,V,W}\otimes\sigma^\vee)_{U(V)}
$$
as a representation of $U(W)$ and $\theta_{\psi,\mu_V,\mu_W,W,V}(\sigma)$ to be its maximal semisimple quotient.

We summary some results in \cite{GI14} for later uses. We first consider the case when $\dim V = \dim W = n$.
\begin{theorem}\label{B2}
    Let $\phi$ be an $L$-parameter for $U(W)$ and  $\pi\in \Pi_\phi^W$. Then we have
    \begin{enumerate}
        \item $\Theta_{\psi,V,W}(\pi)$ is nonzero if and only if $$ \epsilon(1/2,\phi\otimes \mu_V^{-1},\psi^E_2)=\epsilon(V)\cdot \epsilon(W), $$ where $\epsilon(V)=\omega_{E/F}(\text{disc}(V))$ and $\epsilon(W)=\omega_{E/F}(\delta^{-n}\cdot \text{disc}(W))$.
        \item If  $\Theta_{\psi,V,W}(\pi)\neq 0$, then the $L$-parameter of  $\theta_{\psi,V,W}(\pi)$ is $$\theta(\phi)=\phi\otimes \mu_V^{-1}\mu_W.$$
        \item  The theta correspondence $\pi \mapsto  \theta_{\psi,V,W}(\pi)$ gives a bijection $$\Pi_\phi \leftrightarrow \Pi_{\theta(\phi)}.$$
        \item If $\phi$ is tempered and $\Theta_{\psi,V,W}(\pi)\neq 0$, then $\Theta_{\psi,V,W}(\pi)$ is irreducible.
    \end{enumerate}
\end{theorem}
We now consider the case when $\dim V= \dim W +1= n+1$.
\begin{theorem}\label{B3}
    Let $\phi$ be an $L$-parameter for $U(W)$ and  $\pi\in \Pi_\phi^W$. Then we have
    \begin{enumerate}
        \item Assume $\phi$ does not contain $\mu_V$.
        \begin{enumerate}
            \item For any $\pi\in \Pi_\phi^W$, $\Theta_{\psi,V,W}(\pi)$ is nonzero and $\theta_{\psi,V,W}(\pi)$ has $L$-parameter
            $$
            \theta(\phi)=(\phi\otimes \mu_V^{-1}\mu_W)\oplus \mu_W.
            $$
            \item The theta correspondence $\pi \mapsto  \theta_{\psi,V,W}(\pi)$ gives a bijection $$\Pi_\phi \leftrightarrow \Pi^V_{\theta(\phi)}.$$
        \end{enumerate}
        \item Assume that $\phi$ contains $\mu_V$.
        \begin{enumerate}
            \item For any $\pi\in \Pi_\phi^W$, exactly one of $\Theta_{\psi,V,W}(\pi)$ ore $\Theta_{\psi,V^\prime,W}(\pi)$ is nonzero. Here $V^\prime$ is the remaining $n+1$-dimensional hermitian space.
            \item If $\Theta_{\psi,V,W}(\pi)$ is nonzero, then $\theta_{\psi,V,W}(\pi)$ has $L$-parameter
            $$
            \theta(\phi)=(\phi\otimes \mu_V^{-1}\mu_W)\oplus \mu_W.
            $$
            \item The theta correspondence $\pi \mapsto  \theta_{\psi,V,W}(\pi)$ gives a bijection $$\Pi_\phi \leftrightarrow \Pi_{\theta(\phi)}.$$
        \end{enumerate}
        \item If $\pi$ is tempered and $\Theta_{\psi,V,W}(\pi)\neq 0$, then $\Theta_{\psi,V,W}(\pi)$ is irreducible.
    \end{enumerate}
\end{theorem}
\subsection{Geometric multiplicities of the local Gan-Gross-Prasad conjecture}
In this subsection, we revisit various geometric multiplicities for both Bessel and Fourier-Jacobi models.

\subsubsection{Bessel models}

We consider $G=U(V_{n+2r+1})\times U(V_n)$. We recall some linear forms in \cite{BP14,BP16,BP20}. For any $f\in \mathcal{C}_{\text{scusp}}(G(F))$, we set
$$
J_B(f)=\int_{H(F)\backslash G(F)}\int_{H(F)}f(x^{-1}hx)\xi(h)dhdx.
$$
For any quasi-character $\theta$ on $G(F)$, we define
$$
m_{\text{geom}}^{B}(\theta)=\lim_{s\rightarrow 0^+}\int_{\Gamma(G,H)}D^G(x)^{1/2}c_\theta(x)\Delta(x)^{s-1/2}dx,
$$
where the set $\Gamma(G,H)$ and its measure are given in \cite[Section 11.1]{BP20}. By \cite[Theorem 11.4.1]{BP20}, we have
$$
J_B(f)=m_{\text{geom}}^B(\theta_f),\ \ \text{ for all }f\in \mathcal{C}_{\text{scusp}}(G(F)).
$$
We now consider twisted spaces. Let $M=\text{Res}_{E/F}\,\text{GL}_{n+2r+1}(E)\times \text{GL}_{n}(E)$ and $N=\text{Res}_{E/F}\,\text{GL}_n(E)\ltimes U(E)$, where $U$ is the unipotent radical of the parabolic subgroup stabilizing the following flag 
$$
\langle e_1 \rangle \subseteq \ldots \subseteq \langle e_1,\ldots,e_{r} \rangle \subseteq \langle e_1,\ldots, e_{n+r+1} \rangle \subseteq \ldots   \subseteq \langle e_1,\ldots,e_{n+2r+1} \rangle.
$$
We define a character $\eta$  on $U(E)$ by
$$
\eta(u)=\psi_E(\sum_{i=1}^{n+2r}u_{i,i+1})
$$
and extend it to a character on $N(F)$. For each $d\geq 1$, let $\theta_d:g\mapsto J_d {}^t\bar{g}^{-1}J_d^{-1}$ be an involution on $\text{GL}_d(E)$, where $J_d=((-1)^i\delta_{i,d-j})_{i,j}$. This defines an involution $\theta$  on $M$ and $N$. We denote $\tilde{M}=M\theta$ and $\tilde{N}=N\theta$. We extend $\eta$ to a character on $\tilde{N}(F)$. We recall the following linear form studied in \cite{BP14}
$$
\tilde{J}_B(\tilde{f})=\int_{N(F)\backslash M(F)}\int_{\tilde{N}(F)}\tilde{f}(x^{-1}\tilde{n}x)\eta(\tilde{n})d\tilde{n}dx,
$$
for all $\tilde{f}\in \mathcal{C}_{\text{scusp}}(\tilde{M}(F))$. For any quasi-character $\tilde{\theta}$ on $\tilde{M}(F)$, we define the following linear form
$$
\epsilon^B_{\text{geom}}(\tilde{\theta})=\sum_{\tilde{T}\in \mathcal{T}}|W(N,\tilde{T})|^{-1}\lim_{s\rightarrow 0^+}\int_{\tilde{T}(F)/\theta}D^{\tilde{M}}(\tilde{t})^{1/2}c_{\tilde{\theta}}(\tilde{t})\Delta(\tilde{t})^{s-1/2}d\tilde{t},
$$
where $\mathcal{T}$ is defined in \cite[Section 3.2]{BP14}.

\subsubsection{Fourier-Jacobi models}

We now consider $G=U(W_{n+2r})\times U(W_n)$. We define some linear forms which are similar to those studied in the main body of this paper. For any $f\in \mathcal{C}_{\text{scusp}}(G(F))$, we set
$$
J_{FJ}(f)=\int_{H(F)\backslash G(F)}\sum_i\int_{U(W_n)}\int_{U(F)}f(x^{-1}hux)\nu(u)du\langle \phi_i,\omega_{W_n,\psi,\mu}(h)\phi_i \rangle dhdx,
$$
where $\{\phi_i\}_{i\in I}$ is an orthonormal basis for $\omega_{W_n,\psi,\mu}$.

Let ${\mathcal{T}}\left(G,H\right)$ be the set containing $H\left(F\right)$-conjugacy classes of tori $T\left(F\right)$ of $U\left(W_n\right)$ such that $T$ is an elliptic maximal torus of $U\left(W\right)$, where $W$ is a nondegenerate skew-hermitian subspace of $W_n$ such that there exists a split skew-hermitian subspace $W^\prime$ of $W_n$ satisfying $W_n=W\oplus^\perp W^\prime$. We define the linear form $m_{\text{geom}}^{FJ}$ on $\text{QC}(G(F))$, which depends on the parity of $n$.
\begin{itemize}
    \item When $\dim V$ is odd:
    $$
    m_{\text{geom}}^{FJ}\left(\theta\right)=\frac{1}{2}c_{\theta}\left(1\right)+\mu\left(\det W_n\right)\underset{T\in{\mathcal{T}}\left(G,H\right)}{\sum}\frac{\gamma_\psi(T)}{\left|W\left(H,T\right)\right|}\underset{s\rightarrow0^{+}}{\lim}\int_{T\left(F\right)}D^{G}\left(x\right)^{1/2}c_{\theta}\left(x\right)
    $$
    $$
    \mu\left(\det\left(1-x^{-1}\right)\right)\Delta(x)^{s-1/2}dx.
    $$
    \item When $\dim V$ is even:
    $$
    m_{\text{geom}}^{FJ}\left(\theta\right)=c_{\theta}\left(1\right)+\mu\left(\det W_n\right)\underset{T\in{\mathcal{T}}\left(G,H\right)}{\sum}\frac{\gamma_\psi(T)}{\left|W\left(H,T\right)\right|}\underset{s\rightarrow0^{+}}{\lim}\int_{T\left(F\right)}D^{G}\left(x\right)^{1/2}c_{\theta}\left(x\right)
    $$
    $$
    \mu\left(\det\left(1-x^{-1}\right)\right)\Delta(x)^{s-1/2}dx.
    $$
\end{itemize}
We revisit an analog of the twisted trace formula defined in Section \ref{sec8}. Let $M=\text{Res}_{E/F}\,\text{GL}_{n+2r}(E)\times \text{GL}_{n}(E)$ and $N=\text{Res}_{E/F}\,\text{GL}_n(E)\ltimes U^\prime(E)$, where $U^\prime$ is the unipotent radical of the parabolic subgroup stabilizing the following flag 
$$
\langle e_1 \rangle \subseteq \ldots \subseteq \langle e_1,\ldots,e_{r} \rangle \subseteq \langle e_1,\ldots, e_{n+r} \rangle \subseteq \ldots   \subseteq \langle e_1,\ldots,e_{n+2r} \rangle.
$$
We define a character $\eta$  on $U^\prime(E)$ by
$$
\eta(u)=\psi_E(\sum_{i=1}^{n+2r-1}u_{i,i+1})
$$
and extend it to a character on $N(F)$. Similar to Bessel models, we define an involution $\theta$  on $M$ and $N$. We denote $\tilde{M}=M\theta$ and $\tilde{N}=N\theta$. We extend $\eta$ to a character on $\tilde{N}(F)$. Moreover, we fix a character $\chi$ of $Z_M(F)$ such that it is invariant under the action of $\theta$. We now recall an analog of the twisted trace formula considered in Section \ref{sec8}.
$$
\tilde{J}_{FJ,\chi}(\tilde{f})=\int_{Z_M(F)N(F)\backslash M(F)}\sum_i\int_{Z_N(F)\backslash\tilde{N}(F)}\tilde{f}(x^{-1}\tilde{n}x)\eta(\tilde{n})\langle \phi^\prime_i,\tilde{\omega}_{\mu,\chi,\psi}(\tilde{n})\phi^\prime_i \rangle d\tilde{n}dx,
$$
for all $\tilde{f}\in \mathcal{C}_{\text{scusp}}(Z_M(F)\backslash\tilde{M}(F),\chi)$. Here $(\phi_i^\prime)_i$ is an orthonormal basis for the Weil representation $\omega_{\mu,\chi}$. See Secion \ref{sec7.1} for a more precise definition. Let ${\mathcal{T}}\left(\tilde{M},\tilde{N}\right)$ be the set containing $N\left(F\right)$-conjugacy classes of elliptic twisted tori $\tilde{T}$ of $\text{Res}_{E/F}\,\text{GL}\left(W_n\right)\theta_n$ such that $\tilde{T}$ is an elliptic maximal twisted torus of $\text{Res}_{E/F}\,\text{GL}\left(W\right)\theta_{\dim W}$, where $W$ is a nondegenerate skew-hermitian subspace of $W_n$ such that there exists a split skew-hermitian subspace $W^\prime$ of $W_n$ satisfying $W_n=W\oplus^\perp W^\prime$. For any quasi-character $\tilde{\theta}$ on $\tilde{M}(F)$, we define the following linear form
$$
\epsilon^{FJ}_{\text{geom}}(\tilde{\theta})=\sum_{\tilde{T}\in {\mathcal{T}}\left(\tilde{M},\tilde{N}\right)}\frac{\gamma_\psi(\tilde{T})}{\left|W\left(N,\tilde{T}\right)\right|}\lim_{s\rightarrow 0^+}\int_{\tilde{T}(F)/\theta}D^{\tilde{M}}(\tilde{t})^{1/2}c_{\tilde{\theta}}(\tilde{t})\Delta(\tilde{t})^{s-1/2}d\tilde{t},
$$
noting that the above definition does not depend on the sign of $W_n$.

\subsection{Proof of Theorem \ref{maintheoremB}: Fourier-Jacobi models}\label{secB4}

We now prove Theorem \ref{maintheoremB} via induction on $n+r$.  When $n+r=1$, Theorem \ref{maintheoremB} has been known. We consider the case when $n+r\geq 2$. We first prove $\mathbf{(FJ)_{n,r}}$. By the same argument to the main part of this paper, it suffices to prove the following theorem.
\begin{theorem}\label{B.4}
  \begin{enumerate}
      \item For any $f\in \mathcal{C}_{\text{scusp}}(G(F))$, we have $$J_{FJ}(f)=m_{\text{geom}}^{FJ}(\theta_f).$$
      \item For any $\tilde{f}\in \mathcal{C}_{\text{scusp}}(Z_M(F)\backslash \tilde{M}(F),\chi)$, we have $$\tilde{J}_{FJ,\chi}(\tilde{f})=\epsilon^{FJ}_{\text{geom}}(\theta_{\tilde{f}}).$$
  \end{enumerate}  
    
\end{theorem}
The proof Theorem \ref{B.4} follows from Theorem \ref{maintheorem1} and Theorem \ref{geometrictwisted}, up to a dependence on an instance of elliptic $L$-parameters, which was established in \cite{CG22} for the twisted Gan-Gross-Prasad conjecture. Since we want our argument to be independent of the results in \cite{GI16}, we prove a weaker version of the corresponding result in \cite{CG22} for Fourier-Jacobi models, which is still sufficient for our purposes.
\begin{proposition}\label{B5}
    There exists an elliptic $L$-parameter $\phi\times \phi^\prime$ for $G$ such that for any relevant pure inner form $U(W^\bullet_{n+2r})\times U(W_n^\bullet)$ of $G$, we have $$\sum_{\pi\in \Pi_{\phi,W^\bullet_{n+2r}}}\sum_{\sigma\in \Pi_{\phi^\prime,W^\bullet_{n}}}m_{FJ}(\pi,\sigma)\neq 0 \Leftrightarrow \epsilon(W_n^\bullet)=\epsilon(1/2,\phi\times \phi^\prime \times \mu^{-1},\psi_2^E)$$ when $n$ is odd or $$\sum_{\pi\in \Pi_{\phi,W^\bullet_{n+2r}}}\sum_{\sigma\in \Pi_{\phi^\prime,W^\bullet_{n}}}m_{FJ}(\pi,\sigma)\neq 0 \Leftrightarrow \epsilon(W_n^\bullet)=\epsilon(1/2,\phi\times \phi^\prime \times \mu^{-1},\psi^E)$$ when $n$ is even. Moreover, for such $W_n^\bullet$, the above sum is equal to $1$.
\end{proposition}
\begin{proof}
We consider 
    $$
    \phi=\phi_1 \oplus \ldots \oplus \phi_{n+2r}\ \ \text{ and }\ \ \phi^\prime=\phi_1^\prime \oplus \ldots \oplus \phi^\prime_{n},
    $$
where $\phi_1,\ldots,\phi_{n+2r}$ and $\phi_1^\prime,\ldots,\phi^\prime_{n}$ are pairwise distinct conjugate self-dual characters of $E^\times$ of sign $(-1)^{n-1}$ and $\phi_{n+2r}=\mu^{n+2r+1}$.  Let $\xi_1,\ldots,\xi_r$ be characters of $E^\times$ such that they are not conjugate self-dual. We set
$$
\phi^\prime_r=\phi^\prime\oplus \xi_1 \oplus \xi_1^{c,-1}\oplus \ldots \oplus \xi_r \oplus \xi_r^{c,-1}.
$$
For each $\sigma \in \Pi_{\phi^\prime}$, we denote $\sigma_r=\sigma \ltimes (\xi_1\times \ldots \times \xi_r)$.  Let $L=L_{(-1)^{n-1}}$ be an isotropic line of sign $(-1)^{n-1}$ and $V^\bullet$ be a nondegenerate hermitian space of dimension $n+2r$ satisfying
$$
\epsilon(V^\bullet)\epsilon(W_{n+2r}^\bullet)=\epsilon(1/2,\phi^{\prime,\vee}\otimes\mu^{-n-2r},\phi^E_2).
$$
Let $V^\bullet_{0},V_1^\bullet\subset V^\bullet$ such that $V^\bullet=V_{0}^\bullet \oplus^\perp L=V_1^\bullet\oplus^\perp V$, where $V$ is the split $2r$-dimensional hermitian subspace of $V^\bullet$. We consider the following seesaw diagram. 
$$\begin{tikzcd}
U(W_{n+2r}^\bullet)\times U(W_{n+2r}^\bullet) \arrow[d, no head] \arrow[rd, no head] & U(V^\bullet) \arrow[d, no head]     \\
U(W_{n+2r}^\bullet) \arrow[ru, no head]                                              & U(V_{0}^\bullet)\times U(L)
\end{tikzcd}$$
We denote $\phi_r=\theta_{V^\bullet,W^\bullet_{n+2r}}(\phi^{\prime,\vee}_r)$ to be an $L$-parameter for $U(V^\bullet)$. Let $\phi_0$ be an $L$-parameter for $U(V_0^\bullet)$ such that $\theta_{W^\bullet_{n+2r},V_0^\bullet}(\phi_0)=\phi$. Observe
$$
\phi_0=\phi_1\otimes \mu^{-1} \oplus \ldots \oplus \phi_{n+2r-1}\otimes\mu^{-1}\ \ \ \text{ and }\ \ \ \phi_r=\phi_r^{\prime,\vee}.
$$
We treat the case $n$ being odd. The case $n$ being even follows verbatim up to some modifications of notations. By Theorem \ref{B2}, Theorem \ref{B3} and the above seesaw diagram, for any $\pi_0\in \Pi_{\phi,V^\bullet_{0}}$ and $\sigma\in \Pi_{\phi^\prime,W^\bullet_{n}}$, we have
$$
m_{FJ}(\theta_{W^\bullet_{n+2r},V_0^\bullet}(\pi_0),\sigma)=m_{FJ}(\theta_{W^\bullet_{n+2r},V_0^\bullet}(\pi_0),\sigma_r)
$$
$$
=m_B(\theta_{V^\bullet,W^\bullet_{n+2r}}(\sigma_r^\vee),\pi_0^\vee)=m_B(\theta_{V_1^\bullet,W^\bullet_{n}}(\sigma^\vee),\pi_0^\vee).
$$
We now choose $W_{n+2r}^\bullet$ so that
$$
\epsilon(V_0^\bullet)=\epsilon(1/2,\phi^{\prime,\vee}\times \phi_0^\vee,\psi_{-2}^E)=\epsilon(1/2,\phi_0\times \phi^\prime,\psi_{2}^E).
$$
By the induction hypothesis, this is equivalent to the existence of  $\pi_0\in \Pi_{\phi_0,V^\bullet_{0}}$ and $\sigma\in \Pi_{\phi^\prime,W^\bullet_{n}}$ such that 
$$
m_B(\theta_{V^\bullet,W^\bullet_{n+2r}}(\sigma_r^\vee),\pi_0^\vee)\neq 0,
$$
which is to say $m_{FJ}(\theta_{W^\bullet_{n+2r},V_0^\bullet}(\pi_0),\sigma)\neq 0$. In this case, we have
$$
\epsilon(W_{n+2r}^\bullet)=\epsilon(1/2,\phi\times \phi^\prime \times \mu^{-1},\psi_2^E)
$$
as desired. Moreover, by Theorem \ref{B2}(iii) and Theorem \ref{B3}(i)(b), together with the induction hypothesis, it follows that
$$
\sum_{\pi\in \Pi_{\phi,W^\bullet_{n+2r}}}\sum_{\sigma\in \Pi_{\phi^\prime,W^\bullet_{n}}}m_{FJ}(\pi,\sigma)=1.
$$
\end{proof}
\subsection{Proof of Theorem \ref{maintheoremB}: Bessel models}\label{secB5}

In this subsection, we prove $\mathbf{(B)_{n,r}}$ and thus finish our inductive argument. As in \cite{BP14,BP15,BP16,BP20}, we deduce the given statement to the following theorem.
\begin{theorem}\label{B.6}
  \begin{enumerate}
      \item For any $f\in \mathcal{C}_{\text{scusp}}(G(F))$, we have $$J_{B}(f)=m_{\text{geom}}^{B}(\theta_f).$$
      \item For any $\tilde{f}\in \mathcal{C}_{\text{scusp}}(\tilde{M}(F))$, we have $$\tilde{J}_{B}(\tilde{f})=\epsilon^B_{\text{geom}}(\theta_{\tilde{f}}).$$
  \end{enumerate}  
    
\end{theorem}
Theorem \ref{B.6} was established in \cite{BP14,BP15,BP20}, which use some truncation arguments on infinitesimal variants of the corresponding trace formulas. See \cite[Theorem 10.8.1]{BP20} and \cite[Theorem 3.5.1]{BP15} for more details. We would like to provide an alternative way to prove Theorem \ref{B.6} without using such truncations. Similar to the previous subsection, it suffices to prove the following result.
\begin{proposition}
    There exists an elliptic $L$-parameter $\phi\times \phi^\prime$ for $G$ such that for any relevant pure inner form $U(V^\bullet_{n+2r+1})\times U(V_n^\bullet)$ of $G$, we have $$\sum_{\pi\in \Pi_{\phi,V^\bullet_{n+2r+1}}}\sum_{\sigma\in \Pi_{\phi^\prime,V^\bullet_{n}}}m_{B}(\pi,\sigma)\neq 0 \Leftrightarrow \epsilon(V_n^\bullet)=\epsilon(1/2,\phi\times \phi^\prime ,\psi_{-2}^E)$$
    and for such $V_n^\bullet$, the above sum is equal to $1$.
\end{proposition}
\begin{proof}
We follow the same strategy as in the proof of Proposition \ref{B5}. Consider 
    $$
    \phi=\phi_1 \oplus \ldots \oplus \phi_{n+2r+1}\ \ \text{ and }\ \ \phi^\prime=\phi_1^\prime \oplus \ldots \oplus \phi^\prime_{n},
    $$
where $\phi_1,\ldots,\phi_{n+2r+1}$ and $\phi_1^\prime,\ldots,\phi^\prime_{n}$ are pairwise distinct conjugate self-dual characters of $E^\times$ of sign $(-1)^{n}$, and $\phi_{n+2r+1}=\mu^{n+2r+2}$.  Let $\xi_1,\ldots,\xi_r$ be characters of $E^\times$ that are not conjugate self-dual, and define
$$
\phi^\prime_r=\phi^\prime\oplus \xi_1 \oplus \xi_1^{c,-1}\oplus \ldots \oplus \xi_r \oplus \xi_r^{c,-1}.
$$
For each $\sigma \in \Pi_{\phi^\prime}$, we have $\sigma_r\in \Pi_{\phi_r^\prime}$.  Let $L=L_{(-1)^{n}}$ be an isotropic line of sign $(-1)^{n}$, and let $W^\bullet$ be a nondegenerate skew-hermitian space of dimension $n+2r$ satisfying
$$
\epsilon(W^\bullet)\epsilon(V_{n}^\bullet)=\epsilon(1/2,\phi^{\prime,\vee}\otimes\mu^{-n-2r},\phi^E_2).
$$
Write $W^\bullet=W_{0}^\bullet \oplus^\perp W$, where $W$ is the split $2r$-dimensional skew-hermitian subspace of $W^\bullet$. The following seesaw diagram is then considered:
$$\begin{tikzcd}
U(V_{n+2r+1}^\bullet) \arrow[d, no head] \arrow[rd, no head] &  U(W^\bullet)\times U(W^\bullet) \arrow[d, no head]     \\
 U(V_{n+2r}^\bullet)\times U(L)\arrow[ru, no head]                                              & U(W^\bullet)
\end{tikzcd}$$
We denote $\phi_r=\theta_{W^\bullet,V^\bullet_{n+2r}}(\phi^{\prime,\vee}_r)$ to be an $L$-parameter for $U(W^\bullet)$. Let $\phi_0$ be an $L$-parameter for $U(W^\bullet)$ satisfying $\theta_{V^\bullet_{n+2r+1},W^\bullet}(\phi_0)=\phi$. Then
$$
\phi_0=\phi_1\otimes \mu^{-1} \oplus \ldots \oplus \phi_{n+2r}\otimes\mu^{-1}\ \ \ \text{ and }\ \ \ \phi_r=\phi_r^{\prime,\vee}.
$$
We treat the case when $n$ is odd; the case of even $n$ follows verbatim with minor notational adjustments. By Theorem \ref{B2}, Theorem \ref{B3}, together with the above seesaw diagram, for any $\pi_0\in \Pi_{\phi_0,W^\bullet}$ and $\sigma\in \Pi_{\phi^\prime,V^\bullet_{n}}$, we have
$$
m_{B}(\theta_{V^\bullet_{n+2r+1},W^\bullet}(\pi_0),\sigma)=m_{B}(\theta_{V^\bullet_{n+2r+1},W^\bullet}(\pi_0),\sigma_r)
$$
$$
=m_{FJ}(\theta_{W^\bullet,V^\bullet_{n+2r}}(\sigma_r^\vee),\pi_0^\vee)=m_{FJ}(\theta_{W_0^\bullet,V^\bullet_{n}}(\sigma^\vee),\pi_0^\vee).
$$
We choose $V_{n}^\bullet$ so that
$$
\epsilon(W^\bullet)=\epsilon(1/2,\phi^{\prime,\vee}\times \phi_0^\vee\times \mu^{-1},\psi_{2}^E)=\epsilon(1/2,\phi_0\times \phi^\prime\times \mu,\psi_{-2}^E).
$$
By $\mathbf{(FJ)_{n,r}}$, this is equivalent to the existence of  $\pi_0\in \Pi_{\phi_0,W^\bullet}$ and $\sigma\in \Pi_{\phi^\prime,V^\bullet_{n}}$ such that 
$$
m_{FJ}(\theta_{W_0^\bullet,V^\bullet_{n}}(\sigma^\vee),\pi_0^\vee)\neq 0,
$$
which is to say $m_{B}(\theta_{V^\bullet_{n+2r+1},W^\bullet}(\pi_0),\sigma)\neq 0$. In this case, we have
$$
\epsilon(V_{n}^\bullet)=\epsilon(1/2,\phi\times \phi^\prime,\psi_{-2}^E).
$$
As in Theorem \ref{B2}(iii) and Theorem \ref{B3}(i)(b), since our theta correspondences give bijection between the corresponding $L$-packets, $\mathbf{(FJ)_{n,r}}$ gives us
$$
\sum_{\pi\in \Pi_{\phi,V^\bullet_{n+2r+1}}}\sum_{\sigma\in \Pi_{\phi^\prime,V^\bullet_{n}}}m_{B}(\pi,\sigma)=1.
$$
\end{proof}

\end{document}